\documentclass[11pt]{amsart}
\usepackage[dvipsnames]{xcolor}
\usepackage{amsfonts,amssymb,amsmath,amscd,amstext}
\usepackage{mathrsfs}
\usepackage[colorlinks=true,linkcolor=teal,citecolor=purple,urlcolor=Plum]{hyperref}
\usepackage[utf8]{inputenc}
\usepackage[final]{graphicx}
\usepackage{aligned-overset}
\usepackage{float}
\usepackage{comment}
\usepackage[noabbrev,capitalize,nameinlink]{cleveref}
\usepackage{mathdots}
\usepackage[a4paper,top=2.3cm,bottom=2.3cm,left=1.4cm,right=1.4cm]{geometry}
\renewcommand{\leq}{\leqslant}
\renewcommand{\geq}{\geqslant}
\renewcommand{\le}{<}
\renewcommand{\ge}{>}

\newcommand{\hhh}{{\mathcal{H}}}
\newcommand{\hn}{\mathbb{H}^n}
\newcommand{\A}{\mathbf{A}}
\newcommand{\B}{\mathbf{B}}

\newcommand{\rifun}{\mathcal{Q}}
\newcommand{\penafun}{\mathcal{P}}
\newcommand{\n}[1]{\nabla^\varepsilon_{ #1 }}
\newcommand{\tmc}{{\mathscr{H}}}
\newcommand{\timc}{{\left(\mathscr{H}^\hhh\right)^{-1}}}

\newcommand{\vh}{\nu}

\newcommand{\rr}{{\mathbb{R}}}

\newcommand{\hh}{{\mathbb{H}}}

\newcommand{\Om}{\Omega}
\newcommand{\eps}{\varepsilon}

\newcommand{\area}{\sigma}

\newcommand{\s}{\mathcal S}

\newcommand{\Y}{\mathcal X}

\newcommand{\E}{\mathbf{E}}
\newcommand{\X}{\mathbf{X}}
\newcommand{\lagrange}{\Lambda}
\newcommand{\Z}{\mathbf{Z}}
\renewcommand{\n}{\mathbf{N}}
\newcommand{\nnn}{\mathbf{n}}

\newcommand{\average}{{\mathchoice {\kern1ex\vcenter{\hrule height.4pt
				width 6pt
				depth0pt} \kern-9.7pt} {\kern1ex\vcenter{\hrule height.4pt width 4.3pt
				depth0pt}
			\kern-7pt} {} {} }}

\definecolor{champagne}{rgb}{0.97, 0.91, 0.81}

\definecolor{asparagus}{rgb}{0.53, 0.66, 0.42}

\DeclareMathOperator{\divv}{div}

\DeclareMathOperator{\trace}{trace}

\newcommand{\jacobi}{\mathcal{J}}

\DeclareMathOperator{\supp}{supp}
\DeclareMathOperator{\spann}{span}

\newtheorem{theorem}{Theorem}[section]
\newtheorem{proposition}[theorem]{Proposition}

\newtheorem{lemma}[theorem]{Lemma}

\theoremstyle{definition}
\newtheorem{remark}[theorem]{Remark}
\newtheorem{example}[theorem]{Example}

\theoremstyle{remark}

\numberwithin{equation}{section}

\author[M.~Fogagnolo]{Mattia Fogagnolo}
\address[M.~Fogagnolo]{Dipartimento di Matematica ``Tullio Levi-Civita'', Università degli Studi di Padova\protect\newline
\indent  via Trieste 63, 35131 Padova (PD), Italy}
\email{mattia DOT fogagnolo AT unipd DOT it}

\author[A.~Pinamonti]{Andrea Pinamonti} \address[A.~Pinamonti]{Dipartimento di Matematica, Università di Trento\protect\newline
\indent Via Sommarive, 14, 38123 Povo TN} 
\email{andrea DOT pinamonti AT unitn DOT it}

\author[S.~Verzellesi]{Simone Verzellesi}
\address[S. Verzellesi]{Department of Decision Sciences and BIDSA, Bocconi University \protect\newline
\indent Via R\"ontgen 1, 20136, Milano (MI), Italy}
\email{simone DOT verzellesi AT unibocconi DOT it}
\title[Unexpected phenomena for mean curvature functionals in the Heisenberg group]{Unexpected phenomena for mean curvature functionals in the Heisenberg group}

\date{\today}

\thanks{\textit{Memberships and funding information.} The authors are members of the Istituto Nazionale di Alta Matematica (INdAM), Gruppo Nazionale per l'Analisi Matematica, la Probabilità e le loro Applicazioni (GNAMPA). 
M.~Fogagnolo and S.~Verzellesi are supported by the University of Padova. A.~Pinamonti is supported by the University of Trento. A.~Pinamonti and S.~Verzellesi received funding through INdAM-GNAMPA 2026 Project \emph{Variational, Geometric, and Analytic Perspectives on Regularity}, CUP E53C25002010001. S.~Verzellesi received funding by the European Union: the European
Research Council (ERC), through StG “ANGEVA”, project number: 101076411. Views
and opinions expressed are however those of the authors only and do not necessarily reflect
those of the European Union or the European Research Council. Neither the European
Union nor the granting authority can be held responsible for them}

\begin{document}
\begin{abstract}

The Euclidean paradigm that \emph{spheres optimize mean curvature variational problems} breaks down in the sub-Riemannian Heisenberg group: neither the Pansu sphere nor the Korányi sphere is optimal for the variational problems associated with the Minkowski and Heintze-Karcher inequalities. Motivated by this phenomenon, we develop a variational theory for geometric problems driven by the horizontal mean curvature, focusing on the \emph{total mean curvature} functional and the related \emph{Minkowski inequality}, introducing suitable notions of \emph{non-characteristic} stationarity and stability.  
We fully classify critical surfaces, identifying a new one-parameter family of rotationally invariant surfaces, which we call \emph{Pansu-Minkowski spheres}.
We show that a distinguished member, the \emph{optimal Pansu-Minkowski sphere}, emerges as the unique critical point of the Minkowski quotient, and uniquely minimizes it among rotationally invariant competitors.
We prove non-characteristic stability and local minimality of Pansu-Minkowski spheres under rotationally invariant perturbations, while showing their instability under unrestricted perturbations. Finally, we show that the Minkowski quotient does not admit any positive lower bound: in the first sub-Riemannian Heisenberg group, no Minkowski inequality holds.
\end{abstract}
\maketitle
%\tableofcontents
\section{Introduction}
In the Euclidean space, the most relevant geometric variational problems driven by mean curvature exhibit a remarkable rigidity in the shape of their optimal configurations. In their appropriate class of competitors, the \emph{isoperimetric inequality}  \cite{MR2976521}, the \emph{Minkowski inequality} \cite{MR3544938,MR2522433,MR1511220}, and the \emph{Heintze-Karcher inequality} \cite{MR533065,MR1173047,MR996826,MR3702549} all single out the Euclidean sphere as the unique optimal shape. As we shall see, this Euclidean paradigm breaks down in a rather unexpected way in sub-Riemannian geometry. 
%From a variational perspective, spheres simultaneously minimize area under volume constraints, total mean curvature under area constraints, and total inverse mean curvature under volume constraints.
%\textcolor{red}{In this paper we show that this Euclidean intuition fails in a rather dramatic way. The two most natural candidates in the Heisenberg group, the Pansu sphere and the Korányi sphere, do not minimize the analogues of the Minkowski and Heintze–Karcher quotients.}METTEREI QUESTA FRASE PER ANTICIPARE IL COLPO DI SCENA

\medskip
In the sub-Riemannian Heisenberg group $\hh^1$, the prototypical model of sub-Riemannian \cite{MR3971262} and pseudohermitian \cite{MR2312336} geometry, the picture is far less understood. Considerable progress has been achieved for what concerns the isoperimetric problem  \cite{MR0829003,MR2402213,MR2548252,MR2000099,MR2898770,MR2435652}, although it is still unknown whether the closed,  constant \emph{horizontal mean curvature} surface known as \emph{Pansu sphere} (\Cref{examplepansu}) - the unique closed volume-preserving area-stationary surface - is indeed the isoperimetric set. On the other hand, essentially nothing is known for more general geometric variational problems driven by the \emph{horizontal mean curvature} (cf. \cite{cheng2025minimizingsurfacescrinvariant,MR3836671,MR2990131} for further geometric variational problems in related frameworks).

\medskip

In this setting, a general mean curvature driven functional takes the form 
\begin{equation}\label{tmcfunsubriemderfgt5ryhythyh}
    \tmc^\hhh_f(S)=\int_S f\left(H^\hhh\right)\,d\sigma^\hhh.
\end{equation} 
Here $S\subseteq\hh^1$ is a smooth, embedded, closed surface, $H^\hhh$ is its (horizontal) mean curvature, $d\sigma^\hhh$ is its (sub-Riemannian) area element and $f:\mathbb{R}\to\mathbb{R}$ is a smooth function. We refer to \Cref{sec_preliminaries} for the related definitions. Particularly relevant examples are the \emph{(horizontal) area functional}, the \emph{total (horizontal) mean curvature}  and the \emph{total inverse (horizontal) mean curvature}, defined respectively by
 \begin{equation*}
    \sigma^\hhh(S)=\int_S\,d\sigma^\hhh,\qquad \tmc^\hhh(S)=\int_S H^\hhh\,d\sigma^\hhh,\qquad \timc(S)=\int_S\frac{1}{H^\hhh}\,d\sigma^\hhh.
 \end{equation*}
As for the Euclidean case, the homogeneous structure of $\hh^1$ induced by its \emph{intrinsic dilations} relates the minimization of such functionals, under suitable geometric constraints, to the validity of the corresponding sharp geometric inequalities. Precisely, minimizing the total mean curvature under area constraints, i.e.
\begin{equation}\label{optiprobuno}
     \inf\left\{\tmc^\hhh(S)\,:\,S\text{ is mean convex, }\area^\hhh(S)=k\right\},\qquad k>0,
 \end{equation}
 is equivalent to minimizing the (scaling invariant) \emph{Minkowski quotient}
\begin{equation}\label{minkquointro}
    \rifun_{\mathrm{mink}}^\hhh(S)=\left(\area^\hhh(S)\right)^{-\frac{2}{3}}\tmc^\hhh(S).
\end{equation}
Similarly, minimizing the total inverse mean curvature under volume constraints, i.e. 
  \begin{equation}\label{optiprobdue}
     \inf\left\{\timc(S)\,:\,S\text{ is strictly mean convex, }|\Om(S)|=k\right\},\qquad k>0.
 \end{equation}
corresponds to minimizing the (scaling invariant) \emph{Heintze-Karcher quotient}
\begin{equation}\label{hkqouintro} \rifun_{\mathrm{hk}}^\hhh(S)=|\Omega(S)|^{-1}\timc(S),
\end{equation}
where $|\Omega(S)|$ is the Lebesgue measure of the volume enclosed by $S$. In turn, sharp lower bounds for \eqref{minkquointro} and \eqref{hkqouintro} would imply (sub-Riemannian analogs of) sharp Minkowski and Heintze-Karcher inequalities.

\medskip
Beyond the analogy with the Euclidean framework, several reasons suggest that the Pansu sphere, henceforth denoted by $S_{\frac{1}{2}}$ for further convenience, should play a distinguished role in both optimization problems. Regarding the former problem, \eqref{minkquointro} is the correct first higher-order analogue of the \emph{isoperimetric quotient}
\begin{equation}\label{isoqueintro}
    \rifun^\hhh_{\mathrm{isop}}(S)=|\Om(S)|^{-\frac{3}{4}}\sigma^\hhh(S).
\end{equation}
The underlying conceptual picture is that area describes the first-order behavior of volume, while mean curvature governs the first-order behavior of area. 
Regarding \eqref{hkqouintro}, a structural connection with the isoperimetric problem is suggested by the behavior of \eqref{isoqueintro} under evolution by \emph{horizontal inverse mean curvature flow}. Loosely speaking, a family of closed embeddings $(S^t)_{t>0}$ evolves by horizontal inverse mean curvature flow if 
\begin{equation}\label{himcf}
    \frac{\partial S^t}{\partial t}=\frac{1}{H_t^\hhh}\vh_t,
\end{equation}
where $\vh_t$ denotes the \emph{horizontal unit normal} to $S_t$.
Recalling \cite{MR1731639} that 
\begin{equation*}
    \frac{d|\Om(S^t)|}{dt}=\timc(S^t),\qquad \frac{d\sigma^\hhh(S^t)}{dt}=\sigma^\hhh(S^t),
\end{equation*}
a formal computation combined with the forthcoming \Cref{pansunotmin} yields
\begin{equation*}
    \frac{d}{dt}\left(\rifun^\hhh_{\mathrm{isop}}\right)(S^t)=\frac{3}{4}\rifun^\hhh_{\mathrm{isop}}\left(S^t\right)\left(\rifun^\hhh_{\mathrm{hk}}\left(S_\frac{1}{2}\right)-\rifun^\hhh_{\mathrm{hk}}\left(S^t\right)\right).
\end{equation*}
Were the Pansu sphere the optimal configuration for \eqref{hkqouintro}, the isoperimetric quotient would be non-decreasing along \eqref{himcf}.
This perspective is supported by another remarkable fact. The \emph{Korányi sphere} (\Cref{ex_koranyisphere}) - the boundary of the unit ball associated with the \emph{Korányi gauge} - shares the same Heintze-Karcher quotient of the Pansu sphere (\Cref{Koranyinotmin}), and is in fact self-similar under inverse mean curvature flow. 
Finally, since many approaches to the Euclidean Heintze-Karcher inequality provide effective proofs of the Aleksandrov theorem \cite{MR102114} both in the smooth \cite{MR1173047,MR996826} and in the measure-theoretic \cite{MR3921314} framework, it would therefore appear natural to expect an analogous phenomenon in the Heisenberg setting.

\medskip
The first surprising outcome of the present paper is that these intuitions are in fact incorrect.
\begin{theorem}\label{intro_teo}
   Neither the Pansu sphere nor the Korányi sphere minimizes \eqref{minkquointro} or \eqref{hkqouintro}. 
\end{theorem}
This fact has several unexpected consequences. First, it exhibits a substantial lack of symmetry in the shape of optimal configurations for the main geometric variational problems in $\hh^1$. Moreover, since the Pansu sphere is non-optimal for \eqref{hkqouintro}, \Cref{intro_teo} implies that a possible Heintze-Karcher inequality cannot be used as a direct tool toward the Aleksandrov theorem in the sub-Riemannian setting. 
The proof of \Cref{intro_teo} reveals an even more striking phenomenon. The Pansu and Korányi spheres are not merely non-minimizing: they are not even critical points, neither of the normalized functionals \eqref{minkquointro} and \eqref{hkqouintro} (\Cref{pansunotmin}, \Cref{Koranyinotmin}), nor of the associated constrained problems \eqref{optiprobuno} and \eqref{optiprobdue} (\Cref{remarkpunticriticipertuttiiproblemimacaratteristicipero}). The key observation is that \eqref{minkquointro} and \eqref{hkqouintro}, while being preserved by intrinsic dilations, are not invariant under Euclidean dilations. This fact allows one to produce suitable deformations along which both functionals strictly decrease.

\medskip
It therefore becomes natural to investigate the structure of the actual equilibrium configurations of these problems. In the present work, we mainly focus on \eqref{optiprobuno}. Besides its intrinsic interest, this choice is also motivated by analytical considerations. Indeed, as will become apparent from \Cref{teoremavariazionigeneralimainsubriem}, the total mean curvature is, together with the area functional, the only mean curvature functional whose linearization is of zero-order in $H^\hhh$
 and whose stability operator remains second-order. 

\medskip
The study of geometric variations 
must account for the singular phenomena arising from the interaction between the differential and sub-Riemannian structures. 
Specialized to hypersurfaces, the latter gives rise to the appearance of the so-called \emph{characteristic points}. 
Precisely, denoting by $\hhh$ the underlying \emph{horizontal distribution} (\Cref{sec_preliminaries}), a point $p\in S$ is characteristic when $\hhh$ is tangent to $S$ at $p$. In the \emph{characteristic set} $S_0$, the horizontal geometry collapses, and this typically creates substantial analytical difficulties. 

\medskip
Accordingly, to characterize critical points of \eqref{optiprobuno}, we introduce the notion of \emph{non-characteristic variation} (\Cref{variationssectionssssss}), a smooth variation of the ambient space supported away from a neighborhood of the characteristic set. %Despite this restriction, non-characteristic variations still suffice to detect critical points of the area functional, both with and without volume constraints \cite{MR2609016,MR2435652}. 
Therefore, appropriate sub-Riemannian variation formulas, as established in \cite{formuledivariazioneriem} in great generality, constitute a fundamental ingredient in the whole variational analysis (\Cref{teoremavariazionigeneralimainsubriem}). 

\medskip
%As observed above, equilibrium configurations for \eqref{minkquointro} (along non-characteristic variations) coincide with critical points of the total horizontal mean curvature under area constraint.
The natural first-order investigation of \eqref{optiprobuno} consists in seeking critical points of the total mean curvature functionals under non-characteristic variations which preserve the area.
In turn, exploiting the first variation formulas provided by \Cref{varformulaaererfvjecnvorevjnprop} and \Cref{varfortmccnrvrjvnefrjcne}, we show (\Cref{characterizationcriticalpointsproposition}) that this notion of stationarity is equivalent to the stationarity of the \emph{penalized functional} 
\begin{equation}\label{intro_penalfun}
    \penafun_\lagrange^\hhh(S)=\tmc^\hhh(S)-\lagrange\sigma^\hhh(S)
\end{equation}
under arbitrary non-characteristic variations, as well as to the validity of the \emph{Euler-Lagrange equation}
\begin{equation}\label{intor_el}
                -4\left(J(\vh)\alpha+\alpha^2\right)-\lagrange H^\hhh=0.
\end{equation}
In the above formulas, $J$ is the \emph{complex structure} of $\hh^1$, so that, denoting by $\vh$ the \emph{horizontal unit normal} to $S$,  $J(\vh)$ generates the \emph{horizontal tangent space} to $S$ (\Cref{sec_preliminaries}). Moreover, $\alpha$ is the \emph{fundamental function} associated with $S$ (\Cref{sec_preliminaries}), and $\lagrange$ is the appropriate Lagrange multiplier arising from the area constraint.

\medskip
Combining topological arguments with a detailed analysis of \eqref{intor_el}, in \Cref{sec_rotinvcritpointstmc} we provide a full characterization of closed, mean convex solutions to \eqref{intor_el}.
Motivated by the symmetries of the Pansu sphere, we first look for solutions to \eqref{intor_el} in the class of \emph{rotationally invariant surfaces} (\Cref{sec_rotinv}), which are symmetric with respect to rotations around a distinguished \emph{vertical direction}.  
Precisely, setting $4L=\lagrange$, we show that every rotationally invariant solution to \eqref{intor_el} is obtained by vertical rotation of the profile curve $\gamma_L=(x_L,t_L):[-\arccos\sqrt{1-2L},\arccos\sqrt{1-2L}]\to\rr^2$, where $L\in\left(0,\frac{1}{2}\right]$ and 
\begin{equation}\label{pansuminkparamtheorem}
\begin{cases}
x_L(s)=\dfrac{1}{2L}\left(\cos s-\dfrac{1-2L}{\cos s}\right), \\[1em]
t_L(s)=\dfrac{1}{4L^2}\left(\dfrac{s}{2}+\dfrac{\sin 2s}{4}-(1-2L)^2\tan s\right),
\end{cases}
\end{equation}
thus obtaining a one-parameter family of rotationally invariant surfaces, say $S_L$ for $0<L\leq\frac{1}{2},$ which are critical points of the total mean curvature along area-preserving non-characteristic variations. Subsequently, we prove that any closed, mean convex  solution \eqref{intor_el} is necessarily rotationally invariant, thus completing the first-order description of the problem. 

\medskip
This characterization carries with it a further series of consequences. First, unlike in the Euclidean setting, where the sphere is the unique equilibrium configuration of the total mean curvature under area-preserving variations \cite[Theorem 5.3]{MR3544938}, in the Heisenberg group one finds infinitely many critical points. Most notably, one easily realizes (\Cref{examplepansu} and \Cref{pansuiscriticalpoint}) that the limiting configuration corresponding to $L=\frac{1}{2}$ is precisely the Pansu sphere $S_\frac{1}{2}$, whence the choice of notation. Surprisingly, although the Pansu sphere fails to be stationary for $\tmc^\hhh$ under \emph{arbitrary} area-preserving variations (\Cref{remarkpunticriticipertuttiiproblemimacaratteristicipero}), it is nevertheless stationary for $\tmc^\hhh$ along \emph{non-characteristic} area-preserving variations. Remarkably, the stationarity of the Pansu sphere fails even for \emph{amost} non-characteristic variations. Indeed, the variations we perform to prove \Cref{intro_teo}, although not non-characteristic, do not move characteristic points. Besides confirming that the Pansu sphere still plays a role in this variational framework, this phenomenon also reveals a substantial difference with the sub-Riemannian isoperimetric problem: the Pansu sphere is indeed the unique closed, constant horizontal mean curvature surface, e.g. among rotationally invariant surfaces \cite{MR2271950} or even among competitors with at least one isolated characteristic point \cite{MR2435652}.
Both because of their connection with the Minkowski problem and because the Pansu sphere itself arises as a limiting configuration within the family, we refer to the above surfaces as \emph{Pansu-Minkowski spheres}.

\medskip
Among Pansu-Minkowski spheres, one particular element nevertheless plays a privileged role. This phenomenon becomes apparent when comparing the \emph{constrained} stationarity problem associated with \eqref{optiprobuno} with the \emph{unconstrained} stationarity problem associated with \eqref{minkquointro}. Although being equivalent from the viewpoint of minimization, this equivalence breaks down at the \emph{non-characteristic} first-order level. Indeed, while every critical point of $\rifun^\hhh_\mathrm{mink}$ along non-characteristic variations is also stationary for $\tmc^\hhh$ along area-preserving non-characteristic variations (\Cref{propimplicazionenonequaivmeovtrbvnrtvonrgvo}), the converse fails in general. More precisely, within the family of Pansu-Minkowski spheres, the only surface which is critical for \eqref{minkquointro} is the Pansu-Minkowski sphere corresponding to $L=\frac{1}{4}$. It is therefore not surprising that this configuration plays a distinguished role in the minimization of the Minkowski quotient: $S_\frac{1}{4}$ is the unique minimizer of \eqref{minkquointro} within the class of Pansu-Minkowski spheres. Accordingly, we refer to it as the \emph{optimal Pansu-Minkowski sphere}, since the above facts naturally single it out as the primary candidate minimizer for the Minkowski quotient.
\begin{figure}[H]
    \centering
    \includegraphics[
        width=\textwidth,
        trim={0.5cm 2.3cm 0cm 2.3cm},
        clip
    ]{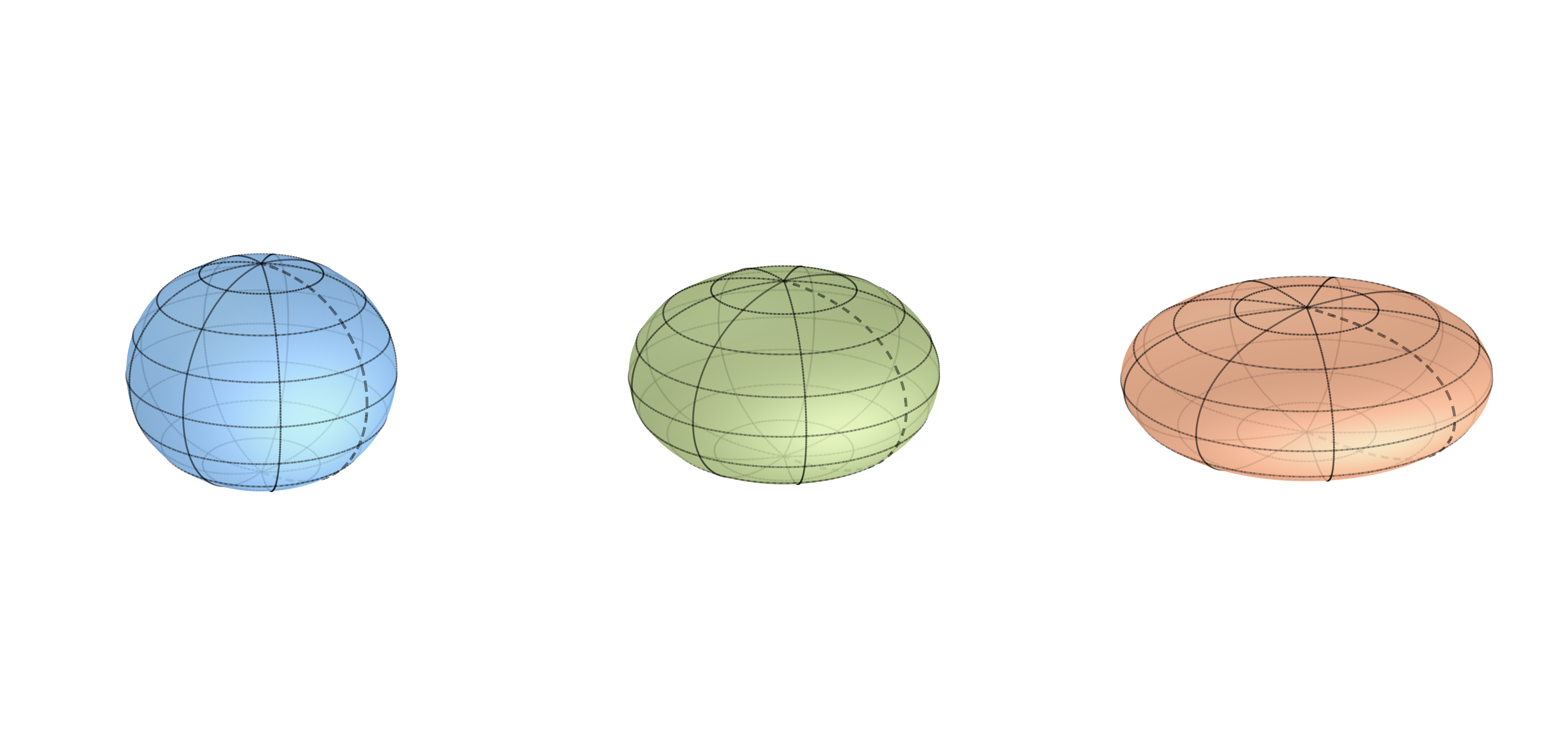}
    \caption{From left to right, the Pansu-Minkowski sphere for $L=\frac{1}{8}$, the optimal Pansu-Minkowski sphere $\left(L=\frac{1}{4}\right)$ and the Pansu sphere $\left(L=\frac{1}{2}\right)$. The three spheres are rescaled via intrinsic dilations to have unit sub-Riemannian area.}
    \label{fig:pm_spheres}
\end{figure}
The next statement summarizes the above considerations.
\begin{theorem}\label{intro_teo2}
   For any   $L\in\left(0,\frac{1}{2}\right]$, define $\gamma_L$ as in \eqref{pansuminkparamtheorem}.
    Denote by $S_L$ its associated rotationally invariant surface. 
    Let $S\subseteq\hh^1$ be a closed, connected, mean convex surface. The following are equivalent:
    \begin{itemize}
           \item [(i)] $S$ is a critical point for $\tmc^\hhh$ along area-preserving non-characteristic variations;
               \item [(ii)] $S$ is, up to intrinsic dilations and left-translations,  a Pansu-Minkowski sphere $S_L$ for some $L\in\left(0,\frac{1}{2}\right]$.
    \end{itemize}
    Moreover, the following are equivalent:
    \begin{itemize}
           \item [(iii)] $S$ is a critical point for $\rifun^\hhh_{\mathrm{mink}}$ along non-characteristic variations;
               \item [(iv)] $S$ is, up to intrinsic dilations and left-translations, the
 optimal Pansu-Minkowski sphere $S_{\frac{1}{4}}$.
    \end{itemize}
    Finally, the
 optimal Pansu-Minkowski sphere $S_{\frac{1}{4}}$ minimizes $\mathcal \rifun_{\mathrm{mink}}^\hhh $ within Pansu-Minkowski spheres:
        \begin{equation*}
        \rifun _{\mathrm{mink}}^\hhh(S_L)\geq (18\pi)^\frac{1}{3},\qquad  \rifun_{\mathrm{mink}} ^\hhh(S_L)= (18\pi)^\frac{1}{3}\text{ if and only if }L=\frac{1}{4}.
    \end{equation*}
\end{theorem}
A further natural question concerns the stability of Pansu-Minkowski spheres under non-characteristic perturbations. As for the isoperimetric problem \cite{MR0731682,MR0917854},
the equivalence between the constrained problem \eqref{optiprobuno} and the unconstrained problem associated with \eqref{intro_penalfun} breaks down at second-order. Exploiting the second variation formulas established in \cite{formuledivariazioneriem} and recalled in \Cref{teoremavariazionigeneralimainsubriem}, we show that stability of critical points of $\tmc^\hhh$ under area-preserving non-characteristic variations is equivalent to stability of critical points of $\penafun^\hhh_L$ under non-characteristic variations which are area-preserving at first-order (\Cref{carstabprop}). 
Consequently, our analysis is carried out at the level of the penalized functional $\penafun^\hhh_L$. Evaluated at the corresponding critical point $S_L$, its second variation along an \emph{arbitrary} non-characteristic variation $\Phi$ is given (\Cref{sec_stab}) by
\begin{equation}\label{secvarformintro}
    \delta^2\penafun^\hhh_L\left(S_L\right)[\Phi]=4\int_{S_L}\Big(\big(-\s\varphi\,J(\vh)\varphi-L\left(J(\vh)\varphi\right)^2\big)+\big(2\s\alpha-H^\hhh\left(\alpha^2+4L^2\right)\big)\varphi^2\Big)\,d\sigma^\hhh.
\end{equation}
In this formula, $\s$ is (up to normalization) the unique vector field which completes $J(\vh)$ to an orthogonal frame of $TS$ (cf. \Cref{sec_preliminaries}), while $\varphi$ is an appropriate projection of the \emph{velocity} of $\Phi$ (\Cref{sec_var_form}).

\medskip
Within the class of rotationally invariant non-characteristic perturbations, all Pansu-Minkowski spheres exhibit a much stronger property than mere stability. Indeed, we prove that the quadratic form associated with \eqref{secvarformintro} is uniformly coercive with respect to \emph{arbitrary} non-characteristic variations. More precisely,
\begin{equation}\label{stabilitformulateoremaequazidrvorjvntrgvojtvnrtvojn}
  \delta^2\penafun^\hhh_L(S_L)[\Phi]\geq 4(1-L)\int_{S_L}\left(J(\vh)\varphi\right)^2\,d\sigma^\hhh+\frac{8L(1-2L)}{1-L}\int_{S_L}\varphi^2\,d\sigma^\hhh.
    \end{equation}
We stress that the above lower bound does not require any first-order area constraint, and is therefore substantially stronger than standard stability. This is in sharp contrast with more familiar settings. For instance, Euclidean spheres do not satisfy such a strong stability property for the penalized functional associated with the Euclidean isoperimetric problem \cite{MR0731682}.
As a direct consequence of \eqref{stabilitformulateoremaequazidrvorjvntrgvojtvnrtvojn}, all Pansu--Minkowski spheres locally minimize the total mean curvature under sufficiently small rotationally invariant non-characteristic perturbations. The above facts can be summarized as follows.
\begin{theorem}\label{intro_teo3}
    Let $L\in\left(0,\frac{1}{2}\right].$ Let $\Phi$ be a non-characteristic rotationally invariant variation (\Cref{sec_stab}). Fix $\delta>0$, and set $I(\delta)=\left(-\arccos\sqrt{1-2L}+\delta,\arccos\sqrt{1-2L}-\delta\right)$. Then:
    \begin{itemize}
        \item [(i)] the lower bound \eqref{stabilitformulateoremaequazidrvorjvntrgvojtvnrtvojn} holds.
    In particular, $S_L$ is a stable critical point of $\tmc^\hhh$ along area-preserving, non-characteristic, rotationally invariant variations;
    \item [(ii)] there exists $\varepsilon=\varepsilon(\delta,L)>0$ such that, if $\varphi\in C^\infty_c(I(\delta))$ satisfies 
    \begin{equation*}
        \area^\hhh\left(S_L^\varphi\right)=\area^\hhh\left(S_L\right),\qquad\|\varphi\|_{C^2(I(\delta))}\leq\varepsilon,
    \end{equation*}
    where $S^\varphi_L$ is a (rotationally invariant) horizontally normal graph over $S_L$ (\Cref{sec_minloc}), then 
    \begin{equation*}
        \tmc^\hhh\left(S_L\right)\leq\tmc^\hhh\left(S_L^\varphi\right).
    \end{equation*}
    \end{itemize}
\end{theorem}

The perhaps most surprising phenomenon emerges once rotational symmetry is removed. We show that Pansu-Minkowski spheres become unstable under general perturbations. Precisely, we construct highly oscillatory, first-order area-preserving angular variations along which the second variation of $\penafun^\hhh_L$ becomes negative. By \Cref{carstabprop}, this yields instability of Pansu-Minkowski spheres for the constrained problem \eqref{optiprobuno}.
\begin{theorem}\label{teononstabintro}
     Let $L\in\left(0,\frac{1}{2}\right].$ There is a non-characteristic first-order area-preserving variation $\Phi$ such that 
    \begin{equation*}
  \delta^2\penafun^\hhh_L(S_L)[\Phi]<0.
    \end{equation*}
  Pansu-Minkowski spheres are then unstable for $\tmc^\hhh$ along area-preserving non-characteristic variations.
\end{theorem}
An important consequence of \Cref{intro_teo2} and \Cref{teononstabintro} is the lack of any strictly mean-convex global minimizer of \eqref{minkquointro}: otherwise (cf. the proof of \Cref{rotmink:theorem}), the latter would be a solution to \eqref{intor_el}, whence a Pansu-Minkowski sphere, which is nevertheless unstable by \Cref{teononstabintro}. On the other hand, \Cref{intro_teo3} suggests that the optimal Pansu-Minkowski sphere may be actually a \emph{global} minimizer in the class of rotationally invariant competitors. We show that this is the case, as stated in the next result. 
\begin{theorem}\label{rotmink:theorem}
Let $S\subseteq\mathbb{H}^{1}$ be a rotationally invariant, closed, mean
convex surface. Then
\begin{equation}\label{rotmink:sharp}
 \mathscr{H}^{\mathcal{H}}(S)
 \geqslant (18\pi)^{\frac13}
 \sigma^{\mathcal{H}}(S)^{\frac23}.
\end{equation}
Equality holds if and only if $S$ is, up to intrinsic dilations and
vertical translations, the optimal Pansu--Minkowski sphere
$S_{\frac14}$.
\end{theorem}
The absence of a strictly mean-convex global minimizer of \eqref{minkquointro} does not \emph{a priori} prevent the validity of a (possibly non-sharp) Minkowski inequality even beyond the rotationally invariant class. In our last result, we show that this is not the case: no Minkowski inequality holds in the first Heisenberg group.
\begin{theorem}\label{ellipsoids_zero_infimum}
Among embedded, closed, mean convex surfaces $S\subseteq\hh^1$,  one has
\[
    \inf_S\rifun^\hhh_{\mathrm{mink}}(S)=0.
\]
Precisely, set 
\[
    S_\lambda=\left\{(x,y,t)\in\hh^1:
        \lambda^{-2}x^2+\lambda^2y^2+t^2=1\right\},\qquad\lambda\geq 1.
\]
Then $S_\lambda$ is smooth, embedded, closed, strictly mean convex. Moreover,
\begin{equation}
\label{twoquantitiesrevlambdaellpisdoistametjne}
    \tmc^\hhh(S_\lambda)=2\pi^2,
    \qquad
    \sigma^\hhh(S_\lambda)\to+\infty
    \quad\text{as }\lambda\to+\infty.
\end{equation}
\end{theorem}
Concerning future perspectives, we expect similarly exotic phenomena to arise in connection with the Heintze-Karcher problem \eqref{optiprobdue}. Indeed, although the Pansu sphere is not stationary for \eqref{optiprobdue} under arbitrary volume-preserving variations, it becomes stationary when one restricts to \emph{non-characteristic} volume-preserving variations (\Cref{remark_hkelforpansuinremarkfrefv}). It therefore remains somewhat mysterious that the Pansu sphere systematically emerges in these variational problems, despite not realizing the corresponding optimal configuration. 

\subsection*{Plan of the paper} The paper is organized as follows. In \Cref{sec_preliminaries} we collect some preliminaries on Heisenberg groups (\Cref{subsec_hg}) and on the geometry of hypersurfaces (\Cref{subsechypersurf}). In \Cref{sec_rotinv} we introduce rotationally invariant surfaces, discuss their properties (\Cref{secdefproprotinv}), give some examples (\Cref{subsec_exam}), introduce the relevant functionals (\Cref{subsecfundrivenbyhormincurv}) and prove \Cref{intro_teo} (\Cref{pansunotmin}, \Cref{Koranyinotmin}). In \Cref{sec_var_form} we focus on the variation formulas for \eqref{tmcfunsubriemderfgt5ryhythyh}. After introducing variations (\Cref{variationssectionssssss}), we state \Cref{teoremavariazionigeneralimainsubriem} (\Cref{subsec_mainvarfotm}) and we specialize it (\Cref{sezionetmcstatstab}) to \eqref{optiprobuno} (\Cref{varformulaaererfvjecnvorevjnprop}, \Cref{varfortmccnrvrjvnefrjcne}), discussing the relevant notions of stationarity (\Cref{characterizationcriticalpointsproposition}) and stability (\Cref{carstabprop}). In \Cref{sec_rotinvcritpointstmc} we characterize critical points of \eqref{optiprobuno}, proving \Cref{intro_teo2}.
 In \Cref{sec_stab} and \Cref{sec_minloc} we discuss their stability,  instability and local minimality, proving \Cref{intro_teo3} and \Cref{teononstabintro}. In \Cref{sec_glominprop} we deal with the global properties, proving \Cref{rotmink:theorem} and \Cref{ellipsoids_zero_infimum}.
\subsection*{Acknowledgments} The authors thank R. Monti, J. Pozuelo, M. Ritoré and D. Vittone for fruitful discussion about the addressed topics. Part of this research was carried out while S. Verzellesi was visiting the Department of Mathematics at the University of Trento.

\section{Preliminaries}\label{sec_preliminaries}
%In this section we collect some basic preliminaries on Heisenberg groups and on the geometry of its hypersurfaces.
\subsection{Heisenberg groups}\label{subsec_hg} Fix $n\geq 1$. The $n$-th Heisenberg group $(\hn,\cdot)$ is $\mathbb R^{2n+1}$ with the group law
\begin{equation}
\label{eq:Hproduct}
    p\cdot p'=( x, y,t)\cdot (x',y',t')=\left( x+ x', y+ y', t+t'+\sum_{j=1}^n\left(x_j'y_j-x_jy_j'\right)\right),
\end{equation}
where $p=(x,y,t)=(x_1,\ldots,x_n,y_1,\ldots,y_n,t)$.
This group law realizes $\hn$ as stratified Lie group. Its \emph{horizontal distribution} $\hhh$ is generated by the left-invariant vector fields
\begin{equation*}
    Z_i=X_i=\frac{\partial}{\partial x_i}+y_i\frac{\partial}{\partial t},\qquad Z_{n+i}=Y_i=\frac{\partial}{\partial y_i}-x_i\frac{\partial}{\partial t},\qquad i=1,\ldots,n.
\end{equation*}
 A vector field which is tangent to $\hhh$ at every point is called \emph{horizontal}. %Accordingly, the \emph{horizontal gradient} of a sufficiently regular function is the horizontal vector field defined by 
 %\begin{equation*}
 %    \nabla^\hhh f=\sum_{i=1}^{2n}Z_ifZ_i.
% \end{equation*}
 Setting $Z_{2n+1}=T=\frac{\partial}{\partial t}$, then $
Z_1,\ldots,Z_{2n+1}
$ is a global frame of left-invariant vector fields. We may identify a point $p\in\hh^n$ with
\begin{equation*}\label{identification}
    \sum_{i=1}^{2n+1}p_j Z_j(p)\in T_p\hh^n.
\end{equation*}
The only nontrivial commutation relations among $Z_1,\ldots,Z_{2n+1}$ are
\begin{equation*}
    [X_i,Y_i]=-2T,\qquad i=1,\ldots,n.
\end{equation*}
The \emph{complex structure} $J:\Gamma(T\hh^n)\to\Gamma(T\hh^n)$ is the unique $C^\infty(\hh^n)$-linear map which satisfies
\begin{equation*}
    J(X_i)=Y_i,\qquad J(Y_i)=-X_i\qquad\text{and}\qquad J(T)=0,\qquad i=1,\ldots,n.
\end{equation*}
 The triple $(\hh^n,\hhh,J)$ is a prototypical \emph{pseudohermitian manifold} (cf. \cite[Appendix]{MR2165405}). 
Moreover, $\hn$ inherits a sub-Riemannian structure $\left(\hh^n,\hhh,\langle\cdot,\cdot\rangle|_\hhh\right)$ by restricting to $\hhh$ the unique Riemannian metric $\langle\cdot,\cdot\rangle$  making $X_1,\ldots,X_n,Y_1,\ldots,Y_n,T$ orthonormal. In the following, orthogonality is meant with respect to $\langle\cdot,\cdot\rangle$.
%We denote the Riemannian gradient by 
%\begin{equation*}
%    \nabla f=\sum_{i=1}^{2n+1}Z_i f Z_i.
%\end{equation*}
 The \emph{pseudohermitian connection} $\nabla$ (cf. \cite{MR4193432}) is the unique metric connection with torsion 
\begin{equation}\label{pseudotorsion}
        \nabla_\A\B-\nabla_\B \A-[\A,\B]=2\langle J(\A),\B\rangle T,\qquad \A,\B\in\Gamma(T\hh^n).
\end{equation}
 We recall that $\nabla$ vanishes along left-invariant vector fields (cf. \cite{MR2898770}), i.e.
\begin{equation}\label{phflat}
    \nabla Z_i=0,\qquad i=1,\ldots,2n+1.
\end{equation}
In addition, $\hh^n$ carries a \emph{homogeneous structure} provided by \emph{intrinsic dilations} (cf. ~\cite{MR2363343}). 
Namely, we set 
\begin{equation*}
\delta_\lambda(x,y,t) = (\lambda x,\lambda y,\lambda^2 t) \qquad \text{for every $\lambda \geq 0$, $(x,y,t) \in\hh^n$}.
\end{equation*}
In this way, $\delta_\lambda$ is a Lie group isomorphism of $\hh^n$ for any $\lambda>0$. The Riemannian volume induced by $\langle\cdot,\cdot\rangle$ is the Haar measure of the group, i.e. the standard Lebesgue measure. It satisfies the homogeneity condition
\begin{equation*} \label{eq_dilation_Lebesgue}
|\delta_\lambda(E)|=\lambda^Q|E| \qquad \text{for every $E \subseteq\hh^n$ measurable, $\lambda \geq 0$,}
\end{equation*}
where $Q\coloneqq 2n+2$ is known as \emph{homogeneous dimension} of $(\hh^n,\cdot)$ (cf. \cite{MR3587666}). Therefore, the Riemannian divergence induced by $\left\langle\cdot,\cdot\right\rangle$ is the Euclidean divergence, and can be computed by 
\begin{equation}\label{divergenza4567}
    \divv\A=\divv\left(\sum_{i=1}^{2n+1}A^iZ_i\right)=\sum_{i=1}^{2n+1}Z_iA^i,\qquad \A\in\Gamma(T\hh^n).
\end{equation}
\subsection{Hypersurfaces}\label{subsechypersurf}  Let $S\subseteq\hh^n$ be a connected, embedded, two-sided hypersurface. If $S$ is closed, we denote by $\Om(S)$ the bounded region it encloses. Recall that $p\in S$ is called a \emph{characteristic point} if $\hhh_p=T_p S$. The set of characteristic points of $S$ is denoted by $S_0$. Throughout the paper, we assume that $S$ is of class $C^2$ and that $S\setminus S_0$ is smooth. At non-characteristic points, the \emph{horizontal tangent space} $\hhh TS$ is the smooth, $(2n-1)$-dimensional distribution defined by 
\begin{equation*}
    \hhh T_p S=\hhh_p\cap T_p S,\qquad p\in S\setminus S_0.
\end{equation*}
Denote by $\n$ the Riemannian unit normal to $S$, and by $\n^\hhh$ its orthogonal projection onto $\hhh$. 
Then, the \emph{horizontal unit normal }  $$\vh=\frac{\n^\hhh}{|\n^\hhh|}$$ 
is well-defined on $S\setminus S_0$, and is the unique, up to sign, horizontal unit vector field orthogonal to $\hhh TS$. Notice that $p$ is a characteristic point if and only if $|\n^\hhh|=0$. Close to every non-characteristic point, it is always possible to extend $\vh$ to a full neighborhood in $\hh^n$ by setting 
   \begin{equation}\label{normcondist}
    \vh=\nabla^\hhh d,
\end{equation}
   where $d$ is the signed \emph{Carnot-Carathéodory distance} from $S$ (cf. \cite{simons,MR4193432}). %Indeed, locally near $p$, $d$ is smooth and satisfies the eikonal equation $|\nabla^\hhh d|=1$ . 
   Henceforth, $\vh$ is always extended as in \eqref{normcondist}.
   The \emph{fundamental function} $\alpha$ is defined on $S\setminus S_0$ as the unique smooth function such that 
   \begin{equation*}\label{firstdefinitionofthevecotrs}
       \s\coloneqq T-\alpha\vh\in\Gamma(TS).
   \end{equation*}
   It is known (cf. \cite{simons,MR4923606}) that $\alpha=Td$. Moreover, if $n=1$, then \cite[Proposition 3.1]{simons} yields
   \begin{equation}\label{nablaessevh}
       \nabla_\s\vh=(J(\vh)\alpha+2\alpha^2)J(\vh).
   \end{equation}
   Denote by $\hhh' TS$ the distribution defined by
\begin{equation*}
    \hhh'T_p S=\hhh T_pS\cap J\left(\hhh T_p S\right),\qquad p\in S\setminus S_0.
\end{equation*}
Then, $\hhh'TS$ is a $(2n-2)$-dimensional sub-bundle of $\hhh TS$, and the latter can be orthogonally decomposed as $
    \hhh TS=\hhh' TS\oplus\spann J(\vh). $
Notice that, in the first Heisenberg group, $\hhh' TS=\{0\}$. 
It is easy to check that 
\begin{equation*}
    \n=\frac{1}{\sqrt{1+\alpha^2}}\vh+\frac{\alpha}{\sqrt{1+\alpha^2}}T,\qquad \alpha=\frac{\left\langle \n,T\right\rangle}{|\n^\hhh|}.
\end{equation*}
%\begin{equation}
  %  TS=\hhh' TS\oplus\spann J(\vh)\oplus\spann \s,
%\end{equation}
%where 
%\begin{equation}\label{formadiesseepsilon}
 %   \s=-\frac{\alpha}{\sqrt{1+\eps^2\alpha^2}}\vh+\frac{1}{\sqrt{1+\eps^2\alpha^2}}\eps T.
%\end{equation}
%Moreover, set
%\begin{equation*}
%    \s\coloneqq\lim_{\eps\to 0}\frac{1}{\eps}\s^\eps=-\alpha\vh+T.
%\end{equation*}
%We recall that the \emph{pseudohermitian connection} of $\hh^n$, say $\nabla$, is the unique metric connection such that
%\begin{equation*}
%    \nabla_{Z_i}Z_j=0,\qquad i,j=1,\ldots,2n+1.
%\end{equation*}
%Accordingly, if $\A,\B,\C\in\Gamma(\hhh)$, \eqref{levi_civita} implies that
%\begin{equation}\label{relazionetraleconnessioni}
%    \left\langle\nabla^\eps_\A\B,\C\right\rangle_\eps=\left\langle\nabla_\A\B,\C\right\rangle.
%\end{equation}
The \emph{horizontal second fundamental form} $h^\hhh$ and the \emph{symmetric horizontal second fundamental form} (cf. \cite{simons,MR4193432}) $\tilde h^\hhh$ are defined on $S\setminus S_0$ respectively by
\begin{equation*}
    h^\hhh(\A,\B)\coloneqq\left\langle A^\hhh(\A),\B\right\rangle,\qquad \tilde h^\hhh(\A,\B)\coloneqq\frac{h^\hhh(\A,\B)+h^\hhh(\B,\A)}{2},\qquad\A,\B\in\Gamma(\hhh TS),
\end{equation*}
 where $ A^\hhh(\A)\coloneqq \nabla_\A\vh$ is the \emph{horizontal shape operator}. The forms $h^\hhh$ and $\tilde h^\hhh$ are related (cf. \cite{simons}) by
\begin{equation*}\label{rapptrahetildacca2026}
    \tilde h^\hhh(\A,\B)=h^\hhh(\A,\B)+\alpha\left\langle J(\A),\B\right\rangle,\qquad\A,\B\in\Gamma(\hhh TS).
\end{equation*}
%For any $\eps>0$, denote by $\sigma^\eps$ the Riemannian surface measure induced by $g_\eps$. It is easy to check that 
%\begin{equation}\label{areaelementeps}
 %   \sigma^\eps=\frac{\sqrt{1+\eps^2\alpha^2}}{\eps\sqrt{1+\alpha^2}}\sigma^1.
%\end{equation}
The \emph{horizontal mean curvature} is then defined on $S\setminus S_0$ by
\begin{equation*}
    H^\hhh=\trace h^\hhh=\trace \tilde h^\hhh.
\end{equation*}
 We will make use of the following commutation identity.
   \begin{lemma}
       Let $n=1$. Then 
       \begin{equation}\label{commutaturejvhesse}
           [J(\vh),\s]=-\alpha H^\hhh J(\vh)-2\alpha\s.
       \end{equation}
   \end{lemma}
   \begin{proof}
      Notice that 
      \begin{equation*}
          [J(\vh),\s]\overset{\eqref{pseudotorsion}}{=}\nabla_{J(\vh)}\s-\nabla_\s J(\vh)-2\langle J(J(\vh)),\s\rangle T.
      \end{equation*}
       But
       \begin{align*}
           \nabla_{J(\vh)}\s\overset{\eqref{phflat}}&{=}-\nabla_{J(\vh)}\alpha\vh =-J(\vh)\alpha\,\vh-\alpha\nabla_{J(\vh)}\vh=-J(\vh)\alpha\,\vh-\alpha H^\hhh\,J(\vh),\\
           -\nabla_\s J(\vh)&=-J\left(\nabla_\s \vh\right)\overset{\eqref{nablaessevh}}{=}(J(\vh)\alpha+2\alpha^2)\vh,\\
           -2\langle J(J(\vh)),\s\rangle T&=2\langle v,\s\rangle T=-2\alpha T=-2\alpha \s-2\alpha^2\vh,
       \end{align*}
       whence \eqref{commutaturejvhesse} follows.
   \end{proof}
When $n=1$, locally in $S\setminus S_0$, there exists a smooth function $\vartheta$ such that 
\begin{equation}
    \vh=\cos\vartheta X+\sin\vartheta Y.
\end{equation}
Then 
\begin{equation*}
    \nabla_\B\vh=\B\vartheta \,J(\vh),\qquad \B\in\Gamma(TS).
\end{equation*}
In particular, 
\begin{equation}\label{duederivatedivartetajvhesse}
    J(\vh)\vartheta=H^\hhh,\qquad \s\vartheta\overset{\eqref{nablaessevh}}{=}J(\vh)\alpha +2\alpha^2.
\end{equation}
It is well-known that 
\begin{equation*}\label{menounoomogenh}
    H^\hhh_{\delta_\lambda(S)}(\delta_\lambda(p))=\frac{1}{\lambda}H^\hhh_S(p),\qquad p\in S\setminus S_0,\lambda>0. 
\end{equation*}
We say that $S$ is \emph{minimal} if $H^\hhh=0$ on $S\setminus S_0$, \emph{mean convex} if $H^\hhh\geq 0$ on $S\setminus S_0$, and \emph{strictly mean convex} if $H^\hhh>0$ on $S\setminus S_0$. 
Finally, the relevant sub-Riemannian surface measure $\sigma^\hhh$ is defined (cf. \cite{MR2354992, MR2262196}) by
\begin{equation}\label{areaelementorizz}
    \sigma^\hhh\coloneqq \frac{1}{\sqrt{1+\alpha^2}}\sigma,
\end{equation}
where $\sigma$ is the Riemannian surface measure induced by $\langle\cdot,\cdot\rangle$. Again, $\sigma^\hhh$ satisfies the homogeneity property
\begin{equation*}\label{qmenounoomogensigmah}
    \sigma^\hhh(\delta_\lambda(S))=\lambda^{Q-1}\sigma^\hhh(S).
\end{equation*}
The \emph{tangent pseudohermitian connection} $\nabla^S$ is the affine connection defined on $S$ by
\begin{equation*}
    \nabla^S _\A \B=\nabla_\A\B-\langle\nabla_\A\B,\vh\rangle\vh,\qquad \A,\B\in\Gamma(\hhh TS).
\end{equation*}
Let $\E_1,\ldots,\E_{2n-1}$ be any local orthonormal frame of $\hhh TS$. If $\varphi\in C^\infty(S\setminus S_0)$ and $\A\in\Gamma(\hhh TS)$ is supported in $S\setminus S_0$, the \emph{horizontal tangential gradient} of $\varphi$ and the  \emph{horizontal tangential divergence} of $\A$ are defined by 
\begin{equation*}
    \nabla ^{\hhh,S}\varphi=\sum_{i=1}^{2n-1}\left(\E_i\varphi\right)\E_i,\qquad \divv^{\hhh,S}\A\coloneqq\sum_{i=1}^{2n-1}\left\langle \nabla^S_{\E_i}\A,\E_i\right\rangle.
\end{equation*}
The \emph{horizontal tangential Laplacian} and the \emph{modified horizontal tangential Laplacian} are then defined by
\begin{equation*}
    \Delta^{\hhh,S} \varphi=\divv^{\hhh,S}\nabla^{\hhh,S}\varphi,\qquad \hat\Delta^{\hhh,S}\varphi\coloneqq \Delta^{\hhh,S}\varphi +2\alpha J(\vh)\varphi.
\end{equation*}
   Finally, we denote by $\jacobi^\hhh$ the \emph{horizontal Jacobi operator}
    \begin{equation*}\label{horizontaljacopdefmain}
    \jacobi^\hhh \varphi=-\hat\Delta^{\hhh,S}\varphi-\varphi\left(|\tilde h^\hhh|^2+4J(\vh)\alpha+(2n+2)\alpha^2\right),\qquad\varphi\in C^\infty(S\setminus S_0).
    \end{equation*}
    Unlike $\Delta^{\hhh,S}$  (cf. \cite{MR2354992}), both $\hat\Delta^{\hhh,S}$ and $\jacobi^\hhh$ are self-adjoint on $C^\infty_c(S\setminus S_0)$.

\section{Rotationally invariant surfaces}\label{sec_rotinv}
\subsection{Definition and properties}\label{secdefproprotinv} We specialize \Cref{subsechypersurf} to rotationally invariant surfaces in $\hh^1.$ A rotationally invariant surface $S$ is, by definition (cf. \cite{MR2271950}), a surface which is invariant under rotation around the $t$-axis. In particular, $S\setminus \{(0,0,t)\}$ can be smoothly parametrized, up to removing a meridian, by the map 
\begin{equation*}
    P(s,\theta)= (x(s)\cos \theta,x(s)\sin\theta,t(s)),\qquad s\in I,\,\theta\in (0,2\pi),
\end{equation*}
where $I\subseteq\rr^n$ is an open, possibly unbounded, interval, and $\gamma(s)=(x(s),t(s))$ is a smooth, regular, embedded curve in $\rr^2$, parametrized counterclockwise and such that $x>0$. The latter is known as the \emph{profile} of $S$. The parametrization $P$ induces local coordinates $(s,\theta)$ on $S$.
The tangent space of $S$ at $P(s,\theta)$ is generated by 
\begin{equation}\label{localcoordsrotation}
\begin{split}
   \frac{\partial}{\partial s}&=\left. \frac{\partial P}{\partial s}\right|_{(s,\theta)}=\dot x(s)\cos\theta\frac{\partial}{\partial x}+\dot x(s)\sin\theta\frac{\partial}{\partial y}+\dot t(s)\frac{\partial}{\partial t}=\dot x(s)\cos\theta X+\dot x(s)\sin\theta Y+\dot t(s)T,\\  
    \frac{\partial}{\partial\theta} &=\left.\frac{\partial P}{\partial \theta}\right|_{(s,\theta)}=- x(s)\sin\theta\frac{\partial}{\partial x}+x(s)\cos\theta\frac{\partial}{\partial y}=- x(s)\sin\theta X+x(s)\cos\theta Y+x(s)^2 T.
    \end{split}
\end{equation}
In particular,
\begin{equation*}
    \n=\frac{1}{\sqrt{(1+x(s)^2)\dot x(s)^2+\dot t(s)^2}}\left(\left(\dot t(s)\cos\theta-x(s)\dot x(s)\sin\theta \right)X+\left(\dot t(s)\sin\theta+x(s)\dot x(s)\cos\theta \right)Y-\dot x(s)T\right),
\end{equation*}
so that
\begin{equation}\label{espressionedinh}
    \left|\n^\hhh\right|=\frac{\sqrt{\dot t(s)^2+x(s)^2\dot x(s)^2}}{\sqrt{(1+x(s)^2)\dot x(s)^2+\dot t(s)^2}}.
\end{equation}
By \eqref{espressionedinh}, we conclude that $S\setminus \{(0,0,t)\}$ is non-characteristic. On the other hand, every possible intersection between $S$ and the vertical axis is a characteristic point of $S$. In conclusion, $S_0=S\cap  \{(0,0,t)\}$.
Therefore, by our previous computations, we deduce that, on $S\setminus S_0$, 
\begin{align}
    \alpha&=-\frac{\dot x(s)}{\sqrt{\Dot t(s)^2+x(s)^2\Dot x(s)^2}}\label{alfasurotinvggggg},\\
    \vh&=\left(\frac{\Dot t(s)\cos\theta-x(s)\Dot x(s)\sin\theta}{\sqrt{\Dot t(s)^2+x(s)^2\Dot x(s)^2}}\right)X+\left(\frac{\Dot t(s)\sin\theta+x(s)\Dot x(s)\cos\theta}{\sqrt{\Dot t(s)^2+x(s)^2\Dot x(s)^2}}\right)Y\label{normalerotinv},\\
        J(\vh)&=-\left(\frac{\Dot t(s)\sin\theta+x(s)\Dot x(s)\cos\theta}{\sqrt{\Dot t(s)^2+x(s)^2\Dot x(s)^2}}\right)X+ \left(\frac{\Dot t(s)\cos\theta-x(s)\Dot x(s)\sin\theta}{\sqrt{\Dot t(s)^2+x(s)^2\Dot x(s)^2}}\right)Y,\label{jvleftinvfields}\\
       \s&=\left(\frac{\dot x(s)\Dot t(s)\cos\theta-x(s)\Dot x(s)^2\sin\theta}{\Dot t(s)^2+x(s)^2\Dot x(s)^2}\right)X+\left(\frac{\dot x(s)\Dot t(s)\sin\theta+x(s)\Dot x(s)^2\cos\theta}{\Dot t(s)^2+x(s)^2\Dot x(s)^2}\right)Y+T\label{exprsrotinv}.
\end{align}
For further convenience, it is useful to express $J(\vh)$ and $\s$ in local coordinates. 
\begin{lemma}\label{jvwithresptolocalchart}
    It holds that, on $S\setminus S_0$,
    \begin{align}
        J(\vh)&=-\left(\frac{x(s)}{\sqrt{\Dot t(s)^2+x(s)^2\Dot x(s)^2}}\right)\frac{\partial}{\partial s}+\left(\frac{\dot t(s)}{x(s)\sqrt{\Dot t(s)^2+x(s)^2\Dot x(s)^2}}\right)\frac{\partial}{\partial \theta}\label{jvinlocalcoords},\\
        \s&=\left(\frac{\dot t (s)}{\Dot t(s)^2+x(s)^2\Dot x(s)^2}\right)\frac{\partial}{\partial s}+\left(\frac{\dot x(s)^2}{\Dot t(s)^2+x(s)^2\Dot x(s)^2}\right)\frac{\partial}{\partial \theta}\label{sinloccords}.
    \end{align}
    \end{lemma}
    \begin{proof}
        Let $a_1,a_2$ be such that $J(\vh)=a_1\frac{\partial}{\partial s}+a_2\frac{\partial}{\partial\theta}$. Then
        \begin{equation}\label{exprjvinproofloccord}
            J(\vh)=a_1\frac{\partial}{\partial s}+a_2\frac{\partial}{\partial\theta}\overset{\eqref{localcoordsrotation}}{=}(a_1\dot x\cos\theta-a_2 x\sin\theta)X+(a_1\dot x\sin\theta+a_2 x\cos\theta)Y+(a_1\dot t+a_2 x^2) T.
        \end{equation}
        Therefore, comparing \eqref{jvleftinvfields} with \eqref{exprjvinproofloccord},
         \begin{equation}\label{oversisteminproofuno} 
\begin{cases}
a_1\dot x\cos\theta-a_2 x\sin\theta&=\,-\dfrac{\Dot t\sin\theta+x\Dot x\cos\theta}{\sqrt{ \dot t^2+x^2\Dot x^2 }},\\
a_1\dot x\sin\theta+a_2 x\cos\theta&=\,\dfrac{\Dot t\cos\theta-x\Dot x\sin\theta}{\sqrt{ \dot t^2+x^2\Dot x^2 }},\\
a_1\dot t+a_2 x^2&=\, 0.
\end{cases}
\end{equation}
      By the third equation of \eqref{oversisteminproofuno}, and since $x\neq 0$ on $S\setminus S_0$,
      %, then $a_2=-\tfrac{a_1\dot t}{x^2}$,
   %     so that 
        \begin{align*}
            a_1\dot x\cos\theta-a_2 x\sin\theta&=a_1\left(\frac{x\dot x\cos\theta+\dot t\sin\theta}{x}\right)\\
            a_1\dot x\sin\theta+a_2 x\cos\theta&=a_1\left(\frac{x\dot x\sin\theta-\dot t\cos\theta}{x}\right).
        \end{align*}
        Therefore
       \begin{equation*}
\begin{cases}
a_1\left(x\dot x\cos\theta+\dot t\sin\theta\right)&=\,-\dfrac{x}{\sqrt{ \dot t^2+x^2\Dot x^2 }}\left(x\Dot x\cos\theta+\Dot t\sin\theta\right),\\[8pt]
a_1\left(x\dot x\sin\theta-\dot t\cos\theta\right)&=\,-\dfrac{x}{\sqrt{ \dot t^2+x^2\Dot x^2 }}\left(x\dot x\sin\theta-\dot t\cos\theta\right).
\end{cases}
\end{equation*}
        Since either $\langle J(\vh),X\rangle \neq 0$ or $\langle J(\vh),Y\rangle \neq 0$, \eqref{jvinlocalcoords} follows. 
        Let $b_1,b_2$ be such that $\s=b_1\frac{\partial}{\partial s}+b_2\frac{\partial}{\partial \theta}$. Then
        \begin{equation}\label{exprsbunobdue}
            \s=b_1\frac{\partial}{\partial s}+b_2\frac{\partial}{\partial \theta}\overset{\eqref{localcoordsrotation}}{=}(b_1\dot x\cos\theta-b_2 x\sin\theta)X+(b_1\dot x\sin\theta+b_2 x\cos\theta)Y+(b_1\dot t+b_2 x^2) T.
        \end{equation}
         Therefore, comparing \eqref{exprsrotinv} with \eqref{exprsbunobdue},
         \begin{equation}\label{oversisteminproofdueswsssssss} 
\begin{cases}
b_1\dot x\cos\theta-b_2 x\sin\theta&=\,\dfrac{\dot x\Dot t\cos\theta-x\Dot x^2\sin\theta}{ \dot t^2+x^2\Dot x^2 },\\
b_1\dot x\sin\theta+b_2 x\cos\theta&=\,\dfrac{\dot x\Dot t\sin\theta+x\Dot x^2\cos\theta}{ \dot t^2+x^2\Dot x^2 },\\
b_1\dot t+b_2 x^2&=\, 1.
\end{cases}
\end{equation}
By the third equation of \eqref{oversisteminproofdueswsssssss},
%,
%\begin{equation*}
 %   b_2=-\frac{b_1\dot t}{x^2}+\frac{1}{x^2}, 
%\end{equation*}
%whence
\begin{align*}
    b_1\dot x\cos\theta-b_2 x\sin\theta&=b_1\left(\frac{x\dot x\cos\theta+\dot t\sin\theta}{x}\right)-\frac{\sin\theta}{x}\\
    b_1\dot x\sin\theta+b_2 x\cos\theta&=b_1\left(\frac{x\dot x\sin\theta-\dot t\cos\theta}{x}\right)+\frac{\cos\theta}{x}.
\end{align*}
Therefore,
\begin{equation*}
\begin{cases}
b_1\left(x\dot x\cos\theta+\dot t\sin\theta\right)&=\,\dfrac{\dot t}{ \dot t^2+x^2\Dot x^2 }\left(x\dot x\cos\theta+\dot t\sin\theta\right),\\[6pt]
b_1\left(x\dot x\sin\theta-\dot t\cos\theta\right)&=\,\dfrac{\dot t}{ \dot t^2+x^2\Dot x^2 }\left(x\dot x\sin\theta-\dot t\cos\theta\right).
\end{cases}
\end{equation*}
Arguing as above, \eqref{sinloccords} follows.
%we conclude that 
%\begin{equation*}
%    b_1=\frac{\dot t}{W},\qquad b_2=\frac{\dot x^2}{W}.
%\end{equation*}
    \end{proof}
%Assume that $S\subseteq\rr^3$ is an embedded, closed surface of class $C^2$.
%Let us denote by $\Om$ the open set enclosed by $S$. 
%Assume in addition that $S$ is invariant under rotations around the $t$-axis. Assume without loss of generality that $(0,0,0)\in S$.Let $\gamma(s)=(x(s),t(s))$ be such that $\gamma\in C^2([a,b],\rr^2)$, $\gamma$ is an embedding onto $\Gamma=\gamma([a,b])$, $x(s)> 0$ for any $s\in(a,b)$, $\gamma(a)=(0,0)$, $x(b)=0$ and 
%\begin{equation*}
 %   (s,\varphi)\overset{P}{\mapsto} (x(s)\cos \varphi,x(s)\sin\varphi,t(s))
%\end{equation*}
%is a parametrization of $S$ of class $C^2$, where $(s,\theta)\in[a,b]\times[0,2\pi]$.
Next, noticing that
\begin{equation*}
    \left\langle\frac{\partial}{\partial s},\frac{\partial}{\partial s}\right\rangle\left\langle\frac{\partial}{\partial \theta},\frac{\partial}{\partial \theta}\right\rangle-\left\langle\frac{\partial}{\partial s},\frac{\partial}{\partial \theta}\right\rangle^2\overset{\eqref{localcoordsrotation}}{=}\left(\dot x(s)^2+\dot t(s)^2\right)\left(x(s)^2+x(s)^4\right)-x(s)^4\dot t(s)^2
\end{equation*}
 and recalling that $x>0$, we deduce
 \begin{equation*}
     d\sigma=x(s)\sqrt{\left(1+x(s)^2\right)\dot x(s)^2+\dot t(s)^2}\,ds\,d\theta,
 \end{equation*}
 hence \begin{equation}\label{subriemareaelementrotinv}
     d\sigma ^\hhh\overset{\eqref{areaelementorizz},\eqref{alfasurotinvggggg}}{=}x(s)\sqrt{\dot t(s)^2+x(s)^2\dot x(s)^2}\,ds\,d\theta.
 \end{equation}
 If $S$ is closed, the divergence theorem yields 
\begin{equation*}
    |\Om(S)|\overset{\eqref{divergenza4567}}{=}\frac{1}{2}\int_\Om\divv\left(x X+y Y\right)\,d\mathcal L^{3}=\left|\int_S\left(x\left\langle \n,X\right\rangle+y\left\langle \n,Y\right\rangle\right)\,d\sigma\right|=\pi\left|\int_Ix(s)^2\dot t(s)\,ds\right|.
\end{equation*}
In particular, when $\dot t\geq 0$ on $I$, 
\begin{equation}\label{volumewhendottgeqzero}
    |\Om(S)|=\pi\int_Ix(s)^2\dot t(s)\,ds.
\end{equation}
 Finally, the horizontal mean curvature of rotationally invariant surfaces can be expressed as follows (cf. \cite{MR2271950}). 
 \begin{lemma}\label{lemmaespressionecurvaturarotinv}
     It holds that,  on $S\setminus S_0$,
\begin{equation}\label{hotinvinlemma}
H^\hhh=\dfrac{x(s)^3  (\dot{x}(s)\ddot{t}(s)-\ddot{x}(s)\dot{t}(s))+\dot{t}(s)^3  }{x(s)\left(x(s)^2\dot{x}(s)^2 +\dot{t}(s)^2 \right)^{3/2}}.
\end{equation}
    
 \end{lemma}
 \begin{proof}
     Notice that
     \begin{equation*}
         H^\hhh=\left\langle \nabla_{J(\vh)}\vh,J(\vh)\right\rangle\overset{\eqref{phflat}}{=}\left\langle J(\vh),X\right\rangle J(\vh) \left\langle \vh,X\right\rangle+\left\langle J(\vh),Y\right\rangle J(\vh) \left\langle \vh,Y\right\rangle.
     \end{equation*}
    It suffices to show that
     \begin{align}
         J(\vh) \left\langle \vh,X\right\rangle&= \left(\frac{x(s)^3  (\dot{x}(s)\ddot{t}(s)-\ddot{x}(s)\dot{t}(s))+\dot{t}(s)^3  }{x(s)\left( \dot{t}(s)^2 +x(s)^2 \dot{x}(s)^2 \right)^{3/2}}\right)\left\langle J(\vh),X\right\rangle,\label{dimmeancurv1}\\
           J(\vh) \left\langle \vh,Y\right\rangle&= \left(\frac{x(s)^3  (\dot{x}(s)\ddot{t}(s)-\ddot{x}(s)\dot{t}(s))+\dot{t}(s)^3  }{x(s)\left( \dot{t}(s)^2 +x(s)^2 \dot{x}(s)^2 \right)^{3/2}}\right)\left\langle J(\vh),Y\right\rangle.\label{dimmeancurv2}
     \end{align}
     Indeed,
\begin{equation*}
    \begin{split}
         J&(\vh) \left\langle \vh,X\right\rangle\overset{\eqref{normalerotinv},\eqref{jvinlocalcoords}}{=}\left(\frac{1}{x\sqrt{ \dot t^2+x^2\Dot x^2 }}\right)\left(-x^2\frac{\partial}{\partial s}\left(\frac{\Dot t\cos\theta-x\Dot x\sin\theta}{\sqrt{\Dot t^2+x^2\Dot x^2}}\right)+\dot t\frac{\partial}{\partial \theta}\left(\frac{\Dot t\cos\theta-x\Dot x\sin\theta}{\sqrt{\Dot t^2+x^2\Dot x^2}}\right)\right)\\
         &=\left(\frac{1}{x( \dot t^2+x^2\Dot x^2 )^2}\right)\left(-x^2\left((\ddot t\cos\theta-\dot x^2\sin\theta-x\ddot x\sin\theta)(  \dot t^2+x^2\Dot x^2 )-(\dot t\cos\theta-x\dot x\sin\theta)(\dot t\ddot t+x\dot x^3+x^2\dot x\ddot x)\right)\right)\\
         &\quad+\left(\frac{1}{x( \dot t^2+x^2\Dot x^2 )^2}\right)\left(\dot t(-\dot t\sin\theta-x\dot x\cos\theta)( \dot t^2+x^2\Dot x^2 )\right)\\
 %    &=\left(\frac{1}{x( \dot t^2+x^2\Dot x^2 )^2}\right)\left(-x^2\left(-\dot x^2\dot t^2\sin\theta-x\ddot x\dot t^2\sin\theta+x^2\dot x^2\ddot t\cos\theta-x\dot x^3\dot t\cos\theta-x^2\dot x\ddot x\dot t\cos\theta+x\dot x\dot t\ddot t \sin\theta\right)\right)\\
 %    &\quad+\left(\frac{1}{x( \dot t^2+x^2\Dot x^2 )^2}\right)\left(\dot t \left(-\dot t^3\sin\theta-x^2\dot x^2\dot t\sin\theta-x\dot x\dot t^2\cos\theta-x^3\dot x^3\cos\theta\right)\right)\\
     &=\left(\frac{1}{x( \dot t^2+x^2\Dot x^2 )^2}\right)\left(x^3\ddot x\dot t^2\sin\theta-x^4\dot x^2\ddot t\cos\theta+x^4\dot x\ddot x\dot t\cos\theta-x^3\dot x\dot t\ddot t\sin\theta-\dot t^4\sin\theta-x\dot x\dot t^3\cos\theta\right)\\
      &=\left(\frac{1}{x( \dot t^2+x^2\Dot x^2 )^2}\right)\left(-\dot t\sin\theta-x\dot x\cos\theta\right)\left(x^3(\dot x\ddot t-\ddot x\dot t)+\dot t^3\right)\\
      \overset{\eqref{jvleftinvfields}}&{=}\left\langle J(\vh),X\right\rangle\left(\frac{x^3(\dot x\ddot t-\ddot x\dot t)+\dot t^3}{\left( \dot t^2+x^2\Dot x^2 \right)^\frac{3}{2}}\right),
    \end{split}
\end{equation*}
which is \eqref{dimmeancurv1}. Finally, \eqref{dimmeancurv2} follows by a similar computation.
 \end{proof}
 \subsection{Examples} \label{subsec_exam}
 We collect some relevant instances of rotationally invariant surfaces.
 \begin{example}\label{horizontalplanewxample}(Horizontal plane)
     The horizontal plane is the rotationally invariant surface  $S=\{(x,y,t)\in \mathbb{R}^3\ |\ t=0\}$. Its profile is the curve $\gamma:(0,\infty)\to\rr^2$ given by
     \begin{equation*}
        \gamma(s)=(s,0). 
     \end{equation*}
     The horizontal plane is a simple instance of minimal surface with characteristic points, because $S_0=\{0\}$.
 \end{example}
  \begin{example}[Vertical cylinder]\label{exvertcil}
     The vertical cylinder  $S=\{(x,y,t)\in \mathbb{R}^3\ |\ x^2+y^2=1, t\in \mathbb{R}\}$ has profile $\gamma:\rr\to\rr^2$ given by
     \begin{equation*}
        \gamma(s)=(1,s). 
     \end{equation*}
     It is a non-characteristic surface with constant horizontal mean curvature $H^\hhh\equiv 1$.
 \end{example}
 \begin{example}[Pansu sphere]\label{examplepansu}
     The Pansu sphere is the closed surface with profile $\gamma:(0,2\pi)\to \rr^2$ given by 
     \begin{equation*}
         \gamma(s)=\left(\sin\left(\frac{s}{2}\right),\frac{1}{4}\left(s-\sin(s)-\pi\right)\right).
     \end{equation*}
     The Pansu sphere has two characteristic points, $S_0=\left\{\left(0,0,-\frac{\pi}{4}\right),\left(0,0,\frac{\pi}{4}\right)\right\},$ and it has constant horizontal mean curvature $H^\hhh\equiv 2$. Up to intrinsic dilations, it is the unique rotationally invariant closed surface of constant horizontal mean curvature of class $C^2$ in $\hh^1$ (cf. \cite{MR2271950}). For further convenience, we point out that the profile of the Pansu sphere can be equivalently parametrized by $\gamma:\left(-\frac{\pi}{2},\frac{\pi}{2}\right)\to\rr^2$, where
     \begin{equation*}
         \gamma(s)=\left(\cos (s),\frac{s}{2}+\frac{\sin 2s}{4}\right).
     \end{equation*}
 \end{example}
 \begin{example}[Korányi sphere]\label{ex_koranyisphere}
     The Korányi sphere is the $1$-level set of the \emph{Korányi norm}
     \begin{equation*}
         \|(x,y,t)\|\coloneqq \left(\left(x^2+y^2\right)^2+4t^2\right)^\frac{1}{4}.
     \end{equation*}
     It is a closed, rotationally invariant surface such that $S_0=\left\{\left(0,0,-\frac{1}{2}\right),\left(0,0,\frac{1}{2}\right)\right\}$, and its profile can be parametrized by $\gamma:\left(-\frac{\pi}{2},\frac{\pi}{2}\right)\to\rr^2,$ where 
     \begin{equation*}
         \gamma(s)=\left(\sqrt{\cos s},\frac{\sin s}{2}\right).
     \end{equation*}
 \end{example}
\subsection{Functionals driven by horizontal mean curvature}\label{subsecfundrivenbyhormincurv} By means of \eqref{subriemareaelementrotinv} and \eqref{hotinvinlemma}, functionals driven by the horizontal mean curvature of rotationally invariant surfaces are essentially one-dimensional.
%The easiest instance is the area functional. Namely, if $S$ is as in \Cref{secdefproprotinv} and if $\gamma=(x,t):I\to\rr^2$ is its profile, then 
Precisely, the (sub-Riemannian) area functional reads as
\begin{equation*}
\area^\hhh(S)=2\pi\int_Ix\sqrt{\dot t^2+x^2 \dot x^2}\,ds,
\end{equation*}
while the (horizontal) total mean curvature and total inverse mean curvature have the form
\begin{equation*}
\tmc^\hhh (S)=2\pi\int_I\frac{x^3  (\dot{x}\ddot{t}-\ddot{x}\dot{t})+\dot{t}^3  }{ \dot{t}^2 +x^2 \dot{x}^2 }\,\mathrm{d}s,\qquad   \timc(S)=2\pi\int_I\frac{x^2\left(\dot t^2+x^2 \dot x^2\right)^2}{x^3  (\dot{x}\ddot{t}-\ddot{x}\dot{t})+\dot{t}^3}\,ds.
\end{equation*} 
As already mentioned in the introduction, the above two quantities gain a geometric meaning provided that $S$ is, respectively, mean convex and strictly mean convex. Exploiting the particular shape of the functionals, we begin by showing that Pansu spheres are critical points neither for \eqref{minkquointro} nor for \eqref{hkqouintro}.
\begin{theorem}\label{pansunotmin}
   For any $R>0$, let $S^R$ be the rotationally invariant surface whose profile $\gamma^R:(0,2\pi)\to\rr^2$ is
   \begin{equation*}
       \gamma^R(s)=\left(R\sin\left(\frac{s}{2}\right),\frac{1}{4}\left(s-\sin s-\pi\right)\right).
   \end{equation*}
   Then $S^R$ is a closed, strictly mean convex surface. Moreover,
   \begin{equation}\label{deriunopansu}
      \rifun^\hhh_\mathrm{mink}(S^1)=2\pi^\frac{2}{3},\qquad \left.\frac{d}{dR}\right|_{R=1}\left(\rifun^\hhh_\mathrm{mink}(S^R)\right)=\pi^\frac{2}{3}
   \end{equation}
   and 
   \begin{equation}\label{deriduepansu}
        \rifun^\hhh_\mathrm{hk}(S^1)=\frac{4}{3},\qquad\left.\frac{d}{dR}\right|_{R=1}\left(\rifun^\hhh_\mathrm{hk}(S^R)\right)=-\frac{2}{3}.
   \end{equation}
   In particular, the Pansu sphere is not a minimum of \eqref{minkquointro} and \eqref{hkqouintro}.
\end{theorem}
\begin{proof}
By definition,
\begin{equation*}
    \Dot x=\frac{R}{2}\cos \left(\frac{s}{2}\right),\qquad\Dot t=\frac{1}{2}\sin^2\left(\frac{s}{2}\right),\qquad\Ddot x=-\frac{R}{4}\sin \left(\frac{s}{2}\right),\qquad\Ddot t=\frac{1}{2}\sin\left(\frac{s}{2}\right)\cos\left(\frac{s}{2}\right).
\end{equation*}
Notice that
\begin{equation*}
    \dot x\ddot t-\ddot x\dot t
    %=\frac{R}{4}\sin\left(\frac{s}{2}\right)\cos^2\left(\frac{s}{2}\right)+\frac{R}{8}\sin^3\left(\frac{s}{2}\right)
    =\frac{R}{8}\sin\left(\frac{s}{2}\right)\left(1+\cos^2\left(\frac{s}{2}\right)\right),\qquad  \dot t^2+x^2\Dot x^2 =\frac{R^4}{4}\sin^2\left(\frac{s}{2}\right)\cos^2\left(\frac{s}{2}\right)+\frac{1}{4}\sin^4\left(\frac{s}{2}\right).
\end{equation*}
%and
%\begin{equation*}
%     \dot t^2+x^2\Dot x^2 =\frac{R^4}{4}\sin^2\left(\frac{s}{2}\right)\cos^2\left(\frac{s}{2}\right)+\frac{1}{4}\sin^4\left(\frac{s}{2}\right).
%\end{equation*}
Then
\begin{equation*}
    d\sigma^\hhh=\frac{R}{2}\sin^2\left(\frac{s}{2}\right)\sqrt{R^4\cos^2\left(\frac{s}{2}\right)+\sin^2\left(\frac{s}{2}\right)},\qquad  H^\hhh=\frac{ R^4\left(1+\cos^2\left(\frac{s}{2}\right)\right)+\sin^2\left(\frac{s}{2}\right)}{R\left(R^4\cos^2\left(\frac{s}{2}\right)+\sin^2\left(\frac{s}{2}\right)\right)^{\frac{3}{2}}}.
\end{equation*}
By the above computation, $S^R$ is strictly mean convex, and moreover
\begin{align*}
\tmc^\hhh(S^R)=\pi\int_{0}^{2\pi}\frac{R^4 \sin^2\left(\frac{s}{2}\right)\left(1+\cos^2\left(\frac{s}{2}\right)\right)+\sin^4\left(\frac{s}{2}\right)}{R^4\cos^2\left(\frac{s}{2}\right)+\sin^2\left(\frac{s}{2}\right)}\,\mathrm{d}s.
\end{align*}
and 
\begin{equation*}
\area^\hhh(S^R)=\pi\int_0^{2\pi} R\sin^2\left(\frac{s}{2}\right)\sqrt{R^4\cos^2\left(\frac{s}{2}\right)+\sin^2\left(\frac{s}{2}\right)}\,\mathrm{d}s
\end{equation*}
Set 
\begin{equation*}
    g(R,s)=\frac{R^4 \sin^2\left(\frac{s}{2}\right)\left(1+\cos^2\left(\frac{s}{2}\right)\right)+\sin^4\left(\frac{s}{2}\right)}{R^4\cos^2\left(\frac{s}{2}\right)+\sin^2\left(\frac{s}{2}\right)},\qquad h(R,s)=R\sin^2\left(\frac{s}{2}\right)\sqrt{R^4\cos^2\left(\frac{s}{2}\right)+\sin^2\left(\frac{s}{2}\right)}.
\end{equation*}
Since 
\begin{equation}\label{boundbelowpansuproof}
    R^4\cos^2\left(\frac{s}{2}\right)+\sin^2\left(\frac{s}{2}\right)\geq\frac{1}{2}
\end{equation}
for any $R$ sufficiently close to $1$, then
\begin{equation*}
    \left.\frac{d}{dR} \area^\hhh(S^R)\right|_{R=1}=\pi\int_{0}^{2\pi}\left.\frac{\partial h(R,s)}{\partial R}\right|_{R=1}\,ds=\pi\int_{0}^{2\pi}\left(\sin^2\left(\frac{s}{2}\right)+\frac{1}{2}\sin^2 s\right)\,ds=\frac{3}{2}\pi^2
\end{equation*}
and 
\begin{equation*}
    \left.\frac{d}{d R} \tmc^\hhh(S^R)\right|_{R=1}=\pi\int_{0}^{2\pi}\left.\frac{\partial g(R,s)}{\partial R}\right|_{R=1}\,ds=4\pi\int_0^{2\pi}\sin^4\left(\frac{s}{2}\right)\,ds=3\pi^2.
\end{equation*}
Moreover,
\begin{equation*}
    \area^\hhh(S^1)=\pi\int_0^{2\pi}\sin^2\left(\frac{s}{2}\right)\,ds=\pi^2,\qquad\tmc^\hhh(S^1)=2\pi\int_0^{2\pi}\sin^2\left(\frac{s}{2}\right)\,ds=2\pi^2.
\end{equation*}
By the above computations,  \eqref{deriunopansu} follows. We prove \eqref{deriduepansu}. Indeed, 
\begin{equation*}
    |\Om(S^R)|\overset{\eqref{volumewhendottgeqzero}}{=}\frac{\pi R^2}{2}\int_0^{2\pi}\sin^4\left(\frac{s}{2}\right)\,ds=\frac{3\pi^2 R^2}{8},\qquad|\Om(S^1)|=\frac{3}{8}\pi^2\qquad\left.\frac{d}{dR}|\Om(S^R)|\right|_{R=1}=\frac{3}{4}\pi^2
\end{equation*}
and
\begin{equation*}
    \timc(S^R)=\pi\int_0^{2\pi} \frac{\sin^2\left(\frac{s}{2}\right)\left(R^5 \cos^2\left(\frac{s}{2}\right)+R\sin^2\left(\frac{s}{2}\right)\right)^2}{R^4\left(1+\cos^2\left(\frac{s}{2}\right)\right)+\sin^2\left(\frac{s}{2}\right)}\,ds,\qquad \timc(S^1)=
    %\frac{\pi}{2}\int_0^{2\pi}\sin^2\left(\frac{s}{2}\right)\,ds=
    \frac{1}{2}\pi^2.
\end{equation*}
Set 
\begin{equation*}
    g(R,s)=\frac{\sin^2\left(\frac{s}{2}\right)\left(R^5 \cos^2\left(\frac{s}{2}\right)+R\sin^2\left(\frac{s}{2}\right)\right)^2}{R^4\left(1+\cos^2\left(\frac{s}{2}\right)\right)+\sin^2\left(\frac{s}{2}\right)}.
\end{equation*}
By \eqref{boundbelowpansuproof}, we infer that
\begin{equation*}
      \left.\frac{d}{d R} \timc(S^R)\right|_{R=1}=\pi\int_{0}^{2\pi}\left.\frac{\partial g(R,s)}{\partial R}\right|_{R=1}\,ds=\frac{3}{4}\pi\int_0^{2\pi}\sin^2 s\,ds=\frac{3}{4}\pi^2.
\end{equation*}
Combining the above computations, \eqref{deriduepansu} follows.
\end{proof}
The same conclusions of \Cref{pansunotmin} hold as well for the Korányi sphere. 
%To state the next result, we recall that the \emph{gamma function} (cf. \cite{MR167642}) is defined by 
%\begin{equation*}
%    \Gamma(p)\coloneqq\int_0^\infty s^{p-1}e^{-s}\,ds,\qquad p>0.
%\end{equation*}
%
\begin{theorem}\label{Koranyinotmin}
  Fix $R>0$. Let $S^R$ be the rotationally invariant surface whose profile $\gamma_R:\left(-\frac{\pi}{2},\frac{\pi}{2}\right)\to\rr^2$ is
   \begin{equation*}
       \gamma^R(s)=\left(R\sqrt{\cos s},\frac{1}{2}\sin s\right).
   \end{equation*}
   Then $S^R$ is a closed, strictly mean convex surface. Moreover,
   \begin{equation}\label{deriunokor}
      \rifun^\hhh_{\mathrm{mink}}(S^1)=6\left(\frac{4\Gamma\left(\frac{3}{4}\right)}{\Gamma\left(\frac{1}{4}\right)}\right)^{-\frac{2}{3}},\qquad \left.\frac{d}{dR}\right|_{R=1}\rifun^\hhh_{\mathrm{mink}}(S^R)=\frac{52}{15}\left(\frac{4\Gamma\left(\frac{3}{4}\right)}{\Gamma\left(\frac{1}{4}\right)}\right)^{-\frac{2}{3}},
   \end{equation}
   where $\Gamma$ is the well-known \emph{gamma function} (cf. \cite{MR167642}), and 
   \begin{equation}\label{deriduekor}
        \rifun^\hhh_{\mathrm{hk}}(S^1)=\frac{4}{3},\qquad\left.\frac{d}{dR}\right|_{R=1}\rifun^\hhh_{\mathrm{hk}}(S^R)=\frac{8}{9}.
   \end{equation}
   In particular, the Korányi sphere is not a minimum of \eqref{minkquointro} and \eqref{hkqouintro}.
\end{theorem}
\begin{proof}
    By definition,
    \begin{equation*}
        \dot x=-\frac{R^2\sin s}{2x},\qquad\dot t=\frac{1}{2}\cos s,\qquad\ddot x=\frac{-R^2 x\cos s+R^2\dot x\sin s }{2x^2},\qquad\ddot t=-\frac{1}{2}\sin s.
    \end{equation*}
    Then
    \begin{equation*}
        \dot x\ddot t-\ddot x\dot t%=\frac{R^2\sin^2 s}{4x}+\frac{R^2\cos^2 s}{4x}-\frac{R^2\dot x\sin s\cos s}{4x^2}=\frac{R^2}{4x}\left(1+\frac{R^2\sin^2 s\cos s}{2x^2}\right)
        =\frac{R^2}{8x}\left(2+\sin^2 s\right),\qquad    \dot t^2+x^2\Dot x^2 =\frac{1}{4}\left(R^4\sin^2 s+\cos^2s\right).
    \end{equation*}
  %  and 
 %   \begin{equation*}
    %     \dot t^2+x^2\Dot x^2 =\frac{1}{4}\left(R^4\sin^2 s+\cos^2s\right).
  %  \end{equation*}
    Therefore
    \begin{equation*}
       d\sigma^\hhh=\frac{R}{2}\sqrt{\cos s}\sqrt{R^4\sin^2 s+\cos^2s},\qquad H^\hhh=\frac{\sqrt{\cos s}\left(R^4\left(2+\sin^2 s\right)+\cos^2 s\right)}{R\left(R^4\sin^2 s+\cos^2 s\right)^\frac{3}{2}}.
    \end{equation*}
    In particular, $S^R$ is strictly mean convex. Moreover,
    \begin{align*}
\tmc^\hhh(S^R)=\pi\int_{-\frac{\pi}{2}}^{\frac{\pi}{2}}\frac{\cos s\left(R^4\left(2+\sin^2 s\right)+\cos^2 s\right)}{R^4\sin^2 s+\cos^2 s}\,\mathrm{d}s
\end{align*}
and 
\begin{equation*}
\area^\hhh(S^R)=\pi\int_{-\frac{\pi}{2}}^{\frac{\pi}{2}}R\sqrt{\cos s}\sqrt{R^4\sin^2 s+\cos^2s} \,\mathrm{d}s.
\end{equation*}
Set 
\begin{equation*}
    g(R,s)=\frac{\cos s\left(R^4\left(2+\sin^2 s\right)+\cos^2 s\right)}{R^4\sin^2 s+\cos^2 s},\qquad h(R,s)=R\sqrt{\cos s}\sqrt{R^4\sin^2 s+\cos^2s}.
\end{equation*}
Since 
\begin{equation}\label{boundbelowkoranyproof}
    R^4\sin^2s+\cos^2s\geq\frac{1}{2}
\end{equation}
for any $R$ sufficiently close to $1$, then
\begin{equation*}
    \tmc^\hhh(S^1)=3\pi\int_{-\frac{\pi}{2}}^{\frac{\pi}{2}}\cos s\,\mathrm{d}s=6\pi
\end{equation*}
and 
\begin{equation*}
   \left.\frac{d}{d R} \tmc^\hhh(S^R)\right|_{R=1}=\pi\int_{-\frac{\pi}{2}}^{\frac{\pi}{2}}\left.\frac{\partial g(R,s)}{\partial R}\right|_{R=1}\,ds=8\pi\int_{-\frac{\pi}{2}}^{\frac{\pi}{2}}\cos^3 s\,ds=\frac{32}{3}\pi.
\end{equation*}
In order to deal with the area functional, we recall (cf. \cite{MR167642}) that the \emph{beta function} is defined by
\begin{equation*}
    B(p,q)=2\int_0^{\frac{\pi}{2}}\left(\sin s\right)^{2p-1}(\cos s)^{2q-1}\,ds,\qquad p,q\in\rr,\,p,q>0,
\end{equation*}
and is related to the gamma function by the identity
\begin{equation}\label{betagamma}
    B(p,q)=\frac{\Gamma(p)\Gamma(q)}{\Gamma(p+q)}.
\end{equation}
%.....and satisfies the recursive property
Therefore, recalling (cf. \cite{MR167642}) that
\begin{equation}\label{gammafattoriale}
   \Gamma(p+1)=p\Gamma(p),\qquad\Gamma\left(\frac{1}{2}\right)=\sqrt{\pi},
\end{equation}
then
\begin{equation*}
    \begin{split}
        \area^\hhh (S^1)=\pi B\left(\frac{1}{2},\frac{3}{4}\right)\overset{\eqref{betagamma}}{=}\pi\left(\frac{\Gamma\left(\frac{1}{2}\right)\Gamma\left(\frac{3}{4}\right)}{\Gamma\left(\frac{5}{4}\right)}\right)\overset{\eqref{gammafattoriale}}{=}4\pi^\frac{3}{2}\left(\frac{\Gamma\left(\frac{3}{4}\right)}{\Gamma\left(\frac{1}{4}\right)}\right).
    \end{split}
\end{equation*}
Moreover, since
\begin{equation*}
     \left.\frac{d}{dR} \area^\hhh(S^R)\right|_{R=1}=\pi\int_{-\frac{\pi}{2}}^{\frac{\pi}{2}}\left.\frac{\partial h(R,s)}{\partial R}\right|_{R=1}\,ds=\pi\int_{-\frac{\pi}{2}}^{\frac{\pi}{2}}\sqrt{\cos s}\left(1+ 2\sin^2s\right)\,ds,
\end{equation*}
then
\begin{equation*}
      \left.\frac{d}{dR} \area^\hhh(S^R)\right|_{R=1}=4\pi^\frac{3}{2}\left(\frac{\Gamma\left(\frac{3}{4}\right)}{\Gamma\left(\frac{1}{4}\right)}\right)+2\pi B\left(\frac{3}{2},\frac{3}{4}\right)\overset{\eqref{betagamma},\eqref{gammafattoriale}}{=}
      %4\pi^\frac{3}{2}\left(\frac{\Gamma\left(\frac{3}{4}\right)}{\Gamma\left(\frac{1}{4}\right)}\right)+2\pi\left(\frac{\Gamma\left(\frac{3}{2}\right)\Gamma\left(\frac{3}{4}\right)}{\Gamma\left(\frac{9}{4}\right)}\right)\overset{\eqref{gammafattoriale}}{=}
      \frac{36}{5}\pi^\frac{3}{2}\left(\frac{\Gamma\left(\frac{3}{4}\right)}{\Gamma\left(\frac{1}{4}\right)}\right).
\end{equation*}
Combining the above computations, \eqref{deriunokor} follows. We prove \eqref{deriduekor}. Indeed,
\begin{equation*}
    |\Om(S^R)|=\frac{\pi R^2}{2}\int_{-\frac{\pi}{2}}^\frac{\pi}{2}\cos^2 s\,ds=\frac{\pi^2 R^2}{4},\qquad  |\Om(S^1)|=\frac{\pi^2}{4},\qquad\left.\frac{d}{dR}|\Om(S^R)|\right|_{R=1}=\frac{\pi^2}{2}.
\end{equation*}
Moreover,
\begin{equation*}
    \timc(S^R)=\pi\int_{-\frac{\pi}{2}}^\frac{\pi}{2}\frac{\left(R^5\sin^2 s+R\cos^2 s\right)^2}{\left(R^4\left(2+\sin^2 s\right)+\cos^2 s\right)}\,ds,\qquad\timc(S^1)=\frac{1}{3}\pi^2.
\end{equation*}
Set 
\begin{equation*}
    g(R,s)=\frac{\left(R^5\sin^2 s+R\cos^2 s\right)^2}{\left(R^4\left(2+\sin^2 s\right)+\cos^2 s\right)}.
\end{equation*}
By \eqref{boundbelowkoranyproof}, we conclude that 
\begin{equation*}
      \left.\frac{d}{d R} \timc(S^R)\right|_{R=1}=\pi\int_{-\frac{\pi}{2}}^\frac{\pi}{2}\left.\frac{\partial g(R,s)}{\partial R}\right|_{R=1}\,ds=\pi\int_{-\frac{\pi}{2}}^\frac{\pi}{2}\frac{20 \sin^2 s-2}{9}\,ds=\frac{8}{9}\pi^2.
\end{equation*}
Combining the above computations, \eqref{deriduekor} follows.
\end{proof}
\begin{remark}\label{remarkpunticriticipertuttiiproblemimacaratteristicipero}
    By \Cref{pansunotmin} and \Cref{Koranyinotmin}, neither the Pansu sphere nor the Korányi sphere are critical points of $\rifun^\hhh_{\mathrm{mink}}$ and $\rifun^\hhh_{\mathrm{hk}}$. It is simple to exploit the previous constructions to show that, in addition, they are not critical points neither for \eqref{optiprobuno} nor for \eqref{optiprobdue}. We prove this for the Pansu sphere, the Korányi case being analogous. To this end, we construct two families of competitors preserving respectively the area and the enclosed volume of the Pansu sphere. Let $\left(S^R\right)_{R>0}$ be as in the proof of \Cref{pansunotmin}. 
    Set 
    \begin{equation*}
        \beta(R)=\left(\frac{\sigma^\hhh(S^1)}{\sigma^\hhh(S^R)}\right)^\frac{1}{3},\qquad \gamma(R)=\left(\frac{|\Om(S^1)|}{|\Om(S^R)|}\right)^\frac{1}{4},\qquad R>0.
    \end{equation*}
    Notice that $\beta(1)=\gamma(1)=1$. Moreover,
    \begin{equation*}
        \sigma^\hhh\left(\delta_{\beta(R)}\left(S^R\right)\right)=\beta(R)^3\sigma^\hhh\left(S^R\right)=\sigma^\hhh\left(S^1\right)
    \end{equation*}
    and
    \begin{equation*}
        \left.\frac{\partial}{\partial R}\right|_{R=1}\tmc^\hhh\left(\delta_{\beta(R)}\left(S^R\right)\right)=\left.\frac{\partial}{\partial R}\right|_{R=1}\left(\beta(R)^2\tmc^\hhh\left(S^R\right)\right)=\sigma^\hhh\left(S^1\right)^\frac{2}{3}\left.\frac{\partial}{\partial R}\right|_{R=1}\rifun^\hhh_\mathrm{mink}\left(S^R\right)\neq 0.
    \end{equation*}
    Similarly,
    \begin{equation*}
        \left|\Om\left(\delta_{\gamma(R)}\left(S^R\right)\right)\right|=\left|\Om\left(S^1\right)\right|,\qquad \left.\frac{\partial}{\partial R}\right|_{R=1}\timc\left(\delta_{\gamma(R)}\left(S^R\right)\right)=\left|\left(S^1\right)\right|\left.\frac{\partial}{\partial R}\right|_{R=1}\rifun^\hhh_\mathrm{hk}\left(S^R\right)\neq 0.
    \end{equation*}
\end{remark}
\section{Variation formulas}\label{sec_var_form}
Motivated by \Cref{pansunotmin} and \Cref{Koranyinotmin}, we seek a general strategy to identify the correct critical configurations for problems such as \eqref{optiprobuno} and \eqref{optiprobdue}. The key tool is provided by the sub-Riemannian first and second variation formulas established in \cite{formuledivariazioneriem}.
First, we recall them in full generality, namely for arbitrary hypersurfaces and arbitrary functionals driven by the horizontal mean curvature. Subsequently, we specialize this framework to the total mean curvature functional in the first Heisenberg group. 
\subsection{Variations}\label{variationssectionssssss}
    A \emph{variation} is a smooth map $\Phi:I\times \hh^n\to \hh^n$, where $I\subseteq\rr$ is any open neighborhood of $0$, such that:
\begin{itemize}
    \item $p\mapsto \Phi(\tau,p)$ is a diffeomorphism for any $\tau\in I$;
    \item $\Phi(0,p)=p$ for any $p\in \hh^n$. 
\end{itemize}
We adopt the notation $\Phi_\tau(p)\coloneqq\Phi(\tau,p)$. A variation is \emph{compactly supported} if $\Phi_\tau(p)=p$ outside a compact set $K(\Phi)$. In the following, we tacitly assume that a variation is compactly supported. This is clearly not restrictive when computing variations of a closed hypersurface. Define the time-dependent vector field $\Y$ by
\begin{equation*}
    \Y(\tau,q)=\left.\frac{\partial}{\partial \sigma}\right|_{\sigma=0}\left(\Phi_{\sigma+\tau}\circ\Phi_\tau^{-1}\right)(q)\qquad\text{for any $\tau\in I$, $q\in \hh^n$.}
\end{equation*}
%Then, the \emph{velocity} $\X$ and the \emph{acceleration} $\X'$ of $\Phi$ are defined respectively by 
Define
\begin{equation*}
    \X(q)=\Y(0,q),\qquad \Z(q)=\left.\frac{\partial }{\partial \tau}\right|_{\tau=0}\Y(\tau,q)+\left(\nabla_\X\X\right)(q)\qquad\text{for any $q\in \hh^n$,}
\end{equation*}
where in the above definition $\nabla$ is the pseudohermitian connection \eqref{pseudotorsion}.
We call the vector fields $\X$ and $\Z$ respectively  \emph{variational velocity field} and \emph{variational acceleration field} of $\Phi$.  Given any couple of vector fields $\X,\Z$ it is always possible to construct a variation having $\X$ and $\Z$ respectively as velocity and acceleration.
\begin{proposition}\label{costruirevariazionesubriem}
    Let $\X,\Z\in\Gamma(T\hh^n)$. Assume that $\X$,$\Z$ are compactly supported in $\hh^n$. Set 
    \begin{equation}\label{costruirevariazioniinhn}
        \Phi_\tau(p)\coloneqq p\cdot \left(\tau \X(p)+\frac{\tau^2}{2}\Z(p)\right),\qquad \tau\in I,\,p\in\hh^n.
    \end{equation}
  %  where we employed the identification \eqref{identification}.
    If $I$ is sufficiently small, then $\Phi$ is a variation. Moreover, $\Phi$ has velocity $\X$ and acceleration $\Z$.
\end{proposition}
\begin{proof}
    First, $\Phi$ is a smooth map such that $\Phi_0(p)=p$. This and the fact that $\X,\Z$ are compactly supported imply that, for $I$ sufficiently small, $p\mapsto\Phi_\tau(p)$ is a diffeomorphism. Hence $\Phi$ is a variation. 
    Set $\Phi_\tau(q)=\left(\Phi_\tau(q)_1,\ldots,\Phi_{\tau}(q)_{2n+1}\right)$, $\Phi^{-1}_\tau(q)=\left(\Phi^{-1}_\tau(q)_1,\ldots,\Phi^{-1}_{\tau}(q)_{2n+1}\right)$, and denote respectively by $(X^i)_{i=1}^{2n+1}$ and $(X^i)_{i=1}^{2n+1}$ the components of $\X$ and $\Z$ with respect to $Z_1,\ldots,Z_{2n+1}.$
 %   \begin{equation*}
  %      \X=\sum_{i=1}^nX^iZ_i,\qquad\Z=\sum_{i=1}^nZ^iZ_i.
 %   \end{equation*}
    Notice that
    \begin{align*}
        \Phi_{\sigma+\tau}(q)_i\overset{\eqref{eq:Hproduct}}&{=}q_i+(\sigma+\tau)X^i(q)+\frac{(\sigma+\tau)^2}{2}Z^i(q),\qquad i=1,\ldots,2n,\\
        \Phi_{\sigma+\tau}(q)_{2n+1}\overset{\eqref{eq:Hproduct}}&{=}q_{2n+1}+(\sigma+\tau)X^{2n+1}(q)+\frac{(\sigma+\tau)^2}{2}Z^{2n+1}(q)\\
        &\quad+\sum_{i=1}^nq_{n+i}\left((\sigma+\tau)X^i(q)+\frac{(\sigma+\tau)^2}{2}Z^i(q)\right)-\sum_{i=1}^nq_{i}\left((\sigma+\tau)X^{n+i}(q)+\frac{(\sigma+\tau)^2}{2}Z^{n+i}(q)\right),\\
    \end{align*}
    whence
\begin{equation*}
    \begin{split}
        \left.\frac{\partial}{\partial \sigma}\right|_{\sigma=0}\Phi_{\sigma+\tau}(q)&=\sum_{i=1}^{2n}\left(X^i(q)+\tau Z^i(q)\right)\frac{\partial}{\partial z_i}+\left(X^{2n+1}(q)+\tau Z^{2n+1}(q)\right)\frac{\partial}{\partial t}\\
        &\quad+\left(\sum_{i=1}^nq_{n+i}\left(X^i(q)+\tau Z^i(q)\right)-\sum_{i=1}^nq_{i}\left(X^{n+i}(q)+\tau Z^{n+i}(q)\right)\right)\frac{\partial}{\partial t}.
    \end{split}
\end{equation*}
Therefore,
\begin{equation*}
    \begin{split}
      &\Y(\tau,p)=\sum_{i=1}^{2n}\left(X^i\left(\Phi^{-1}_\tau(p)\right)+\tau Z^i\left(\Phi^{-1}_\tau(p)\right)\right)\frac{\partial}{\partial z_i}+\left(X^{2n+1}\left(\Phi^{-1}_\tau(p)\right)+\tau Z^{2n+1}\left(\Phi^{-1}_\tau(p)\right)\right)\frac{\partial}{\partial t}\\
        &\qquad+\left(\sum_{i=1}^n\Phi^{-1}_\tau(p)_{n+i}\left(X^i(\Phi^{-1}_\tau(p))+\tau Z^i(\Phi^{-1}_\tau(p))\right)-\sum_{i=1}^n\Phi^{-1}_\tau(p)_{i}\left(X^{n+i}(\Phi^{-1}_\tau(p))+\tau Z^{n+i}(\Phi^{-1}_\tau(p))\right)\right)\frac{\partial}{\partial t}.
    \end{split}
\end{equation*}
In particular, recalling that $\Phi_0^{-1}(p)=p$, we deduce that $\Y(0,p)=\X(p)$. Next, as $   \Phi\left(\tau,\Phi_\tau^{-1}(p)\right)=p$,
\begin{equation}\label{derivatainversofivariaz}
    \begin{split}
        0        %=\left.\frac{\partial}{\partial \tau}\right|_{\tau=0}\left(   \Phi\left(\tau,\Phi_\tau^{-1}(p)\right)\right)\\
=\frac{\partial\Phi}{\partial \tau}(0,p)+\sum_{i,j=1}^{2n+1}\frac{\partial\Phi_i}{\partial z_j}(0,p)\left(\left.\frac{\partial}{\partial \tau}\right|_{\tau=0}\Phi_t^{-1}(p)_j\right)\frac{\partial}{\partial z_i}        =\X(p)+\left.\frac{\partial}{\partial \tau}\right|_{\tau=0}\Phi_\tau^{-1}(p).
    \end{split}
\end{equation}
%Therefore
%\begin{equation}
 %   \left.\frac{\partial}{\partial \tau}\right|_{\tau=0}\Phi_\tau^{-1}(p)=-\frac{\partial\Phi}{\partial \tau}(0,p)=-\X(p).
%\end{equation}
Therefore,
\begin{equation*}
    \begin{split}
        \left.\frac{\partial}{\partial \tau}\right|_{\tau=0}\Y(\tau,p)&=\Z(p)+\sum_{i=1}^{2n}\left(\left.\frac{\partial}{\partial \tau}\right|_{\tau=0}X^i\left(\Phi^{-1}_\tau(p)\right)\right)\frac{\partial}{\partial z_i}\\
        &\quad+\left(\sum_{i=1}^n \left.\frac{\partial}{\partial \tau}\right|_{\tau=0}\left(\Phi^{-1}_\tau(p)_{n+i}\right)X^i(p)-\sum_{i=1}^n \left.\frac{\partial}{\partial \tau}\right|_{\tau=0}\left(\Phi^{-1}_\tau(p)_{i}\right)X^{n+i}(p)\right)\frac{\partial}{\partial t}\\
        &\quad+\left(\sum_{i=1}^np_{n+i} \left.\frac{\partial}{\partial \tau}\right|_{\tau=0}\left(X^i(\Phi^{-1}_\tau(p))\right)-\sum_{i=1}^np_{i} \left.\frac{\partial}{\partial \tau}\right|_{\tau=0}\left(X^{n+i}(\Phi^{-1}_\tau(p))\right)\right)\frac{\partial}{\partial t}\\
        \overset{\eqref{derivatainversofivariaz}}&{=}\Z(p)-\sum_{i=1}^{2n+1}\X\left\langle \X,\Z_i\right\rangle(p)Z_i(p)\\
        \overset{\eqref{phflat}}&{=}\Z(p)-\left(\nabla_\X\X\right)(p).
    \end{split}
\end{equation*}
\end{proof}
If a hypersurface $S$ is fixed, we say that a variation is \emph{non-characteristic} whenever it is compactly supported outside $S_0$. Roughly speaking, non-characteristic variations do not move $S$ close to its characteristic points. Throughout this section we restrict ourselves to consider non-characteristic variations. 
%In this case, there exist
%$\varphi\in C^\infty_c(S\setminus S_0)$ and $\X^{\hhh TS}\in\Gamma(\hhh TS)$ such that $\X|_S$ admits the $\langle\cdot,\cdot\rangle$-orthogonal decomposition
%\begin{equation}\label{decoxinsubriemvarformmain}
 %  \X|_S=\varphi\left(\vh+\alpha T\right)+X^{\hhh TS}.
%\end{equation}
%When $\varphi_T=0$ and $\X^{\hhh TS}=0$, we say that $\Phi$ is \emph{horizontally normal}.
\subsection{First and second variation formulas}\label{subsec_mainvarfotm}
  The following crucial result, proved in  \cite[Theorem 3.14]{formuledivariazioneriem}, establishes the variation formulas for \eqref{tmcfunsubriemderfgt5ryhythyh}. 
 We state it for closed hypersurfaces. If instead one computes variations on non-compact hypersurfaces, it suffices to restrict the relevant functionals to $K(\Phi)$. 
%As it will be clearer in a while, the first variation formula of $\tmc^\hhh_f$ depends only on $\X|_S$, while the second variation depends only on $\X|_S$ and $\Z|_S$, accordingly, we define
%\begin{equation*}
  %%  \delta\tmc^\hhh_f(S)[\X]\coloneqq \left.\frac{d}{d\tau}\right|_{\tau=0}\tmc^\hhh\left(\Phi_\tau(S)\right),\qquad   \delta^2\tmc^\hhh_f(S)[\X,\Z]\coloneqq \left.\frac{d^2}{d\tau^2}\right|_{\tau=0}\tmc\left(\Phi^\hhh_\tau(S)\right).
%\end{equation*}  
\begin{theorem}\label{teoremavariazionigeneralimainsubriem}
        Let $S\subseteq\hh^n$ be an embedded, closed hypersurface. Fix a function $f:\rr\to\rr$ which is smooth in a neighborhood of $\left\{H^\hhh(p)\,:\,p\in S\setminus S_0\right\}$. Let $\Phi$ be a non-characteristic variation.  Denote by $\X$ and $\Z$ its velocity and acceleration respectively. Define $\varphi,\psi\in C^\infty_c(S\setminus S_0)$ by
        \begin{equation}\label{lefunzioniimportanti}
        \begin{split}        \varphi&\coloneqq\left\langle \X,\vh+\alpha T\right\rangle,\\
            \psi&\coloneqq \left\langle\Z,\vh+\alpha T\right\rangle-2\left(\X-\varphi\vh\right)\varphi-\Big\langle\nabla_{\X-\varphi\vh}\left(\X-\varphi\vh\right),\vh+\alpha T\Big\rangle-2\alpha\varphi\left\langle \X,J(\vh)\right\rangle.
        \end{split}
        \end{equation}
        Set 
        \begin{equation*}
 \delta\tmc^\hhh_f(S)[\Phi]\coloneqq \left.\frac{d}{d\tau}\right|_{\tau=0}\tmc^\hhh_f\left(\Phi_\tau(S)\right),\qquad   \delta^2\tmc^\hhh_f(S)[\Phi]\coloneqq \left.\frac{d^2}{d\tau^2}\right|_{\tau=0}\tmc^\hhh_f\left(\Phi_\tau(S)\right).
\end{equation*} 
        Then 
            \begin{align}            \delta\tmc^\hhh_f(S)[\varphi]\coloneqq\delta\tmc^\hhh_f(S)[\Phi]&=  \int_S\varphi\Big(\jacobi^\hhh f'(H^\hhh)+f(H^\hhh) H^\hhh\Big)\,d\sigma^\hhh,\label{variazprimasubriemmain}\\
            \delta^2\tmc_f^\hhh(S)[\Phi]&=\delta\tmc^\hhh_f(S)\left[\psi\right]+\int_S\varphi\mathcal L^\hhh\varphi\,d\sigma^\hhh, \label{variazionesecondasubriemannianamain}
    \end{align}
    where $ \mathcal L^\hhh$ is the self-adjoint operator defined on $S\setminus S_0$ by
    \begin{equation*}
        \begin{split}
            \mathcal L^\hhh\varphi&=\jacobi^\hhh\left(f''(H^\hhh)\jacobi^\hhh\varphi\right)\\
            &\quad+2\divv^{\hhh,S} \left\langle f'(H^\hhh)A^\hhh\left(\nabla^{\hhh,S}\varphi\right)\right)+4 f'(H^\hhh)\s J(\vh)\varphi-\left(f'(H^\hhh)H^\hhh+f(H^\hhh)\right)\hat\Delta^{\hhh,S}\varphi\\
            &\quad+4\alpha f'(H^\hhh)\Big(\tilde h^\hhh\left(\nabla^{\hhh,S}\varphi,J(\vh)\right)-H^\hhh J(\vh)\varphi\Big)- 4 f''(H^\hhh)\Big(\tilde h^\hhh\left(\nabla^{\hhh,S}H^\hhh,\nabla^{\hhh,S}\varphi\right)+ \s \varphi J(\vh)H^\hhh\Big)\\
            &\quad+\left(f''(H^\hhh)H^\hhh-2f'(H^\hhh)\right)\left\langle \nabla^{\hhh,S} H^\hhh,\nabla^{\hhh,S}\varphi\right\rangle\\
            &\quad+\varphi f'(H^\hhh)\Big( 2\trace\left(\left(\tilde h^\hhh\right)^3\right)+12\tilde h^\hhh\left(\nabla^{\hhh,S}\alpha,J(\vh)\right)+8\s\alpha+6\tilde h^\hhh(J(\vh),J(\vh))\alpha^2+2 H^\hhh\alpha^2\Big)\\
            &\quad+\varphi \Big ( f(H^\hhh)\left( H^\hhh\right)^2-\left(2f'(H^\hhh) H^\hhh+f(H^\hhh)\right)\left(|\tilde h^\hhh|^2+4J(\vh)\alpha+(2n+2)\alpha^2\right)\Big).
        \end{split}
    \end{equation*}
   Here $\jacobi^\hhh$ and $A^\hhh$ are the horizontal Jacobi operator and the horizontal shape operator (cf. \Cref{subsechypersurf}).
    \end{theorem} 
    \subsection{Total mean curvature: stationarity and stability}\label{sezionetmcstatstab}
    In this section we focus on the total mean curvature functional
    \begin{equation*}
    \tmc^\hhh(S)=\int_S H^\hhh\,d\sigma^\hhh
\end{equation*}
in the first Heisenberg group, describing the correct notions of stationarity and stability when area constraints are prescribed. First, we specialize \Cref{teoremavariazionigeneralimainsubriem} to area and total mean curvature. In both cases, \eqref{variazprimasubriemmain} and \eqref{variazionesecondasubriemannianamain} simplify drastically. 
   \begin{proposition}  [Variation formulas - Area]\label{varformulaaererfvjecnvorevjnprop}
            Let $S\subseteq\hh^1$ be an embedded, closed surface. Let $\Phi$ be a non-characteristic variation. 
            %Denote by $\X$ and $\Z$ its velocity and acceleration respectively. Decompose $\X|_S$ as in \eqref{decoxinsubriemvarformmain}.
        Then
    \begin{align}
            \delta\area^\hhh(S)[\Phi]&=  \int_S\varphi H^\hhh\,d\sigma^\hhh,\label{variazioneprimaareasubriemmain}\\     
            \delta^2\area^\hhh(S)[\Phi]&=\delta\area^\hhh(S)\left[\psi\right]+\int_S\Big(\left(J(\vh)\varphi\right)^2-4\varphi^2\left(J(\vh)\alpha+\alpha^2\right)\Big)\,d\sigma^\hhh.\label{variazsecondaareastatementsubriem}
    \end{align}   
    \end{proposition}
    Similar versions of \Cref{varformulaaererfvjecnvorevjnprop} can be found e.g. in \cite{MR2354992,MR2435652}.
   \begin{proposition}[Variation formulas - Total mean curvature]\label{varfortmccnrvrjvnefrjcne}
            Let $S\subseteq\hh^1$ be an embedded, closed surface. Let $\Phi$ be a non-characteristic variation. 
            %Denote by $\X$ and $\Z$ its velocity and acceleration respectively. Decompose $\X|_S$ as in \eqref{decoxinsubriemvarformmain}.
        Then
    \begin{align}
            \delta\tmc^\hhh(S)[\Phi]&=  -4\int_S\varphi\left(J(\vh)\alpha+\alpha^2\right)\,d\sigma^\hhh,\label{primavariaztmcdtatementmain}\\
            \delta^2\tmc^\hhh(S)[\Phi]&=\delta\tmc^\hhh(S)\left[\psi\right]+4\int_S\Big(-\s\varphi\,J(\vh)\varphi+\varphi^2\left(2\s\alpha-H^\hhh\alpha^2\right)\Big)\,d\sigma^\hhh.\label{variazsecondatmcsubriemmainstatement}
    \end{align}   
    \end{proposition}
    \begin{proof}
        Since $f(H^\hhh)\equiv H^\hhh$, $f'(H^\hhh)\equiv 1$ and $f''(H^\hhh)\equiv 0$. Moreover, as $n=1$, $\hhh TS=\spann J(\vh)$, whence 
        \begin{equation*}
            \nabla^{\hhh,S}g=\left(J(\vh) g\right)J(\vh),\qquad\tilde h^\hhh(\A,\B)=H^\hhh\left\langle J(\vh),\A\right\rangle \left\langle J(\vh),\B\right\rangle,\qquad |\tilde h^\hhh|^2=\left(H^\hhh\right)^2
        \end{equation*}
        for $g\in C^\infty(S\setminus S_0)$ and $\A,\B\in\Gamma(\hhh T S)$. Then \eqref{primavariaztmcdtatementmain} directly follows. Finally,
   $\mathcal L^\hhh$ simplifies as 
    \begin{equation*}
        \begin{split}
            \mathcal L^\hhh\varphi&=2\divv^{\hhh,S} \left( H^\hhh J(\vh)\varphi J(\vh)\right)+4 \s J(\vh)\varphi-2H^\hhh\hat\Delta^{\hhh,S}\varphi-2\left\langle \nabla^{\hhh,S} H^\hhh,\nabla^{\hhh,S}\varphi\right\rangle+4\varphi \Big( 2\s\alpha- H^\hhh\alpha^2\Big)\\
            &=2\divv^{\hhh,S} \left( H^\hhh J(\vh)\varphi J(\vh)\right)+4 \s J(\vh)\varphi-2\divv^{\hhh,S}\left(H^\hhh \nabla^{\hhh,S}\varphi\right)-4\alpha H^\hhh J(\vh)\varphi+4\varphi \Big( 2\s\alpha- H^\hhh\alpha^2\Big)\\
            &=4 \s J(\vh)\varphi-4\alpha H^\hhh J(\vh)\varphi+4\varphi \Big( 2\s\alpha- H^\hhh\alpha^2\Big).
        \end{split}
    \end{equation*}
    The thesis follows integrating by parts the first term of the above formula (cf. \cite[Proposition 3.10]{formuledivariazioneriem}).
 %   \begin{equation*}
   %     \int_S\varphi\mathcal L^\hhh\varphi\,d\sigma^\hhh=4\int_S\Big(-\s\varphi\,J(\vh)\varphi+\varphi^2\left(2\s\alpha-H^\hhh\alpha^2\right)\Big)\,d\sigma^\hhh.
%    \end{equation*}
    \end{proof}
    We are interested in variations which preserve the area. Accordingly, we say that a variation $\Phi$ is:
    \begin{itemize}
        \item [(i)] \emph{area-preserving} if $\area^\hhh(\Phi_\tau(S))\equiv \area^\hhh (S)$ for every $\tau\in I$.
        \item [(ii)] \emph{first-order area-preserving} if $\delta\area^\hhh(S)[\Phi]=0$.
    \end{itemize}
    Area-preserving variations are first-order area-preserving, while the latter condition is the first-order approximation of the area-preserving property. Nevertheless, owing to a classical argument (cf. \cite[Lemma 2.1]{MR0731682}), every first-order area-preserving variation can be modified to produce an area-preserving variation enjoying its same first-order behavior. 
    Before doing this, we need to exclude the existence of closed minimal surfaces.
    \begin{lemma}\label{nominimalandclosed}
        Let $S\subseteq\hh^1$ be an embedded, closed surface. There exists $p\in S\setminus S_0$ such that $H^\hhh(p)\neq 0$.
    \end{lemma}
    \begin{proof}
        Assume not by contradiction. Then $H^\hhh\equiv 0$ on $S\setminus S_0$. Let $\tilde p\in S\setminus S_0$. By \cite[Theorem 4.8]{MR2435652}, and since $S$ is closed, there exists $s_1\in(0,\infty)$ maximal with the property that the straight line segment $\{\tilde p\cdot\delta_s\left(J(\vh_{\tilde p})\right)\,:\,0\leq s\leq s_1\}$ is contained in $S$. Set $q=\tilde p\cdot\delta_{s_1}\left(J(\vh_{\tilde p})\right)$. Again by \cite[Theorem 4.8]{MR2435652}, $q\in S_0$. Therefore, by \cite[Theorem 4.15]{MR2435652}, $q$ is either an isolated characteristic point or it is part of a $C^1$-curve of characteristic points. In both cases, \cite[Theorem 4.15]{MR2435652} ensures the non-maximality of $s_1$, a contradiction.
    \end{proof}
    \begin{proposition}\label{docarmo}
        Let $S\subseteq\hh^1$ be an embedded, closed surface. Let $\Phi$ be a  non-characteristic variation of $S$. Denote by $\X$ its velocity. If $\Phi$ is first-order area-preserving, then there exists a non-characteristic, area-preserving variation with velocity $\X$.
    \end{proposition}
    \begin{proof}
        Let $\A$ be a smooth vector field with compact support outside $S_0$. Define the smooth map $\tilde\Phi$ by
        \begin{equation*}
            \tilde\Phi(\tau,\sigma,p)\coloneqq p\cdot\left(\tau\X(p)+\sigma\A(p)\right),\qquad\tau,\sigma\in I,\,p\in\hh^1.
        \end{equation*}
        Notice that $\area^\hhh\left(\tilde\Phi(0,0,S)\right)=\area^\hhh (S)$. Moreover, arguing as in the proof of \Cref{costruirevariazionesubriem} and by \eqref{variazioneprimaareasubriemmain}, 
        \begin{equation*}
            \left.\frac{\partial}{\partial \tau}\right|_{\tau=0}\area^\hhh\left(\tilde\Phi(\tau,0,S)\right)=\delta\area^\hhh(S)[\X]=0
        \end{equation*}
        and
        \begin{equation*}
            \left.\frac{\partial}{\partial \sigma}\right|_{\sigma=0}\area^\hhh\left(\tilde\Phi(0,\sigma,S)\right)=\delta\area^\hhh(S)[\A]=\int_S\left(\langle\A,\vh\rangle+\alpha\langle\A,T\rangle\right)H^\hhh\,d\sigma^\hhh.
        \end{equation*}
       By \Cref{nominimalandclosed}, $H^\hhh$ cannot vanish identically on $S\setminus S_0$. 
       %Otherwise, \cite[Theorem 4.17]{MR2435652} and \cite[Theorem 4.12]{MR2435652} would imply that $\area^\hhh(S)=0$, a contradiction. 
       Therefore, we can choose $\A$ such that 
        \begin{equation*}
            \left.\frac{\partial}{\partial \sigma}\right|_{\sigma=0}\area^\hhh\left(\tilde\Phi(0,\sigma,S)\right)\neq 0.
        \end{equation*}
        Therefore, the implicit function theorem yields the existence of a function $\sigma=\sigma(\tau)$, smooth in a neighborhood of $0$, such that $\sigma(0)=0$, $\sigma'(0)=0$ and $\area^\hhh\left(\tilde\Phi(\tau,\sigma(\tau),S)\right)\equiv\area^\hhh(S)$. Then $\tilde\Phi(\tau,p)\coloneqq\tilde\Phi(\tau,\sigma(\tau),p)$ satisfies the desired requirements.
    \end{proof}
    
    According to the above definitions, we say that $S$ is an \emph{area-preserving critical point along non-characteristic variations} (for the total mean curvature) if $\delta\tmc^\hhh(S)[\Phi]=0$ for any area-preserving, non-characteristic variation. As customary, stationarity is equivalent to the validity of an appropriate Euler-Lagrange equation, as well as to the stationarity of a suitable penalized functional. We just recall the following, elementary, linear algebra result.
    \begin{lemma}\label{lemmalinearalgebramoltiplicatore}
    Let $V$ be a vector space. Let $L_1,L_2:V\to\rr$ be linear functionals. Then, $\ker L_2\subseteq\ker L_1$ if and only if there exists $\lambda\in\rr$ such that $L_1=\lambda L_2$.
\end{lemma}
%\begin{proof}
 %   If $L_2=0$, then $\ker L_2=V$. Then, by assumption, $L_1=0$, and the thesis follows for any $\lambda\in\rr$. Assume $L_2\neq 0$. Then there exists $v_0$ such that $L_2(v_0)=1$. Set $\lambda=L_1(v_0)$. Let $v\in V$. Notice that $v=\left(v-L_2(v)v_0\right)+L_2(v)v_0$. Since $L_2(v_0)=1$, then $L_2(v-L_2(v)v_0)=0$. Therefore, by assumption, $L_1(v-L_2(v)v_0)=0$. We conclude that $L_1(v)=L_1\left(L_2(v)v_0\right)=\lambda L_2(v)$.
%\end{proof}
    \begin{proposition}[Characterization of stationarity]\label{characterizationcriticalpointsproposition}
         Let $S\subseteq\hh^1$ be an embedded, closed surface. The following are equivalent: 
         \begin{itemize}
             \item [(i)] $S$ is an area-preserving critical point along non-characteristic variations, i.e.
             \begin{equation*}
                 \delta\tmc^\hhh(S)[\Phi]=0
             \end{equation*}
             for any area-preserving, non-characteristic variation $\Phi$;
             \item [(ii)] there exist $\lagrange\in\rr$ such that
             \begin{equation}\label{variazioneprimazerotmcstatement}
                 \delta\tmc^\hhh(S)[\Phi]-\lagrange \delta\area^\hhh(S)[\Phi]=0
             \end{equation}
          for any non-characteristic variation $\Phi$.   
         \end{itemize}
         In these cases, $\lagrange$ is unique, and 
         \begin{equation}\label{derivationeulerlagrangeminkgeneral}
                 -4\left(J(\vh)\alpha+\alpha^2\right)=\lagrange H^\hhh\qquad\text{on }S\setminus S_0.
             \end{equation}
    \end{proposition}
     \begin{proof}
         The proof of (ii)$\implies$(i) follows because area-preserving variations are first-order area-preserving. We prove (i)$\implies$(ii). Let $\Phi$ be any non-characteristic variation. Denote by $\X$ its velocity.  Assume that $\delta\area^\hhh(S)[\Phi]$=0. By \Cref{docarmo}, there exists an area-preserving, non-characteristic variation $\tilde\Phi$ with velocity $\X$. By (i), $\delta\tmc^\hhh(S)[\Phi]=\delta\tmc^\hhh(S)[\tilde\Phi]=0$. Therefore, (ii) follows by \Cref{lemmalinearalgebramoltiplicatore}. In addition, \eqref{derivationeulerlagrangeminkgeneral} easily follows by (ii), \eqref{variazioneprimaareasubriemmain} and \eqref{primavariaztmcdtatementmain}. Finally, uniqueness of $\lagrange$ follows by \eqref{derivationeulerlagrangeminkgeneral} and the already known fact that $H^\hhh$ cannot vanish identically on $S\setminus S_0$.
     \end{proof}
     One may wonder whether the above-mentioned notions of stationarity are equivalent to stationarity for the Minkowski quotient $\rifun^\hhh_{\mathrm{mink}}$. One implication is trivial. 
\begin{proposition}\label{propimplicazionenonequaivmeovtrbvnrtvonrgvo}
    If $S$ is a critical point for $\rifun^\hhh_{\mathrm{mink}}$ along non-characteristic variations, then $S$ is an area-preserving critical point for $\tmc^\hhh$ along non-characteristic variations.
\end{proposition}
\begin{proof}
    Let $\Phi$ be an area-preserving non-characteristic variation. Then
    \begin{equation*}
        0=  \left.\frac{d}{dt}\right|_{t=0}\rifun^\hhh_\mathrm{mink}\left(\Phi_t(S)\right)=\sigma^\hhh(S)^{-\frac{2}{3}}\left(\delta^\hhh(S)[\Phi]\right),
    \end{equation*}
    whence $\delta^\hhh(S)[\Phi]=0$, and the thesis follows.
\end{proof}
Nevertheless, as we will discuss in \Cref{sec_rotinvcritpointstmc}, the converse implication is surprisingly false. 
     Finally, we discuss some equivalent notions of stability of critical points. As in many classical settings (cf. e.g. \cite{MR0731682,hkleon}), while the first order behavior of $\tmc^\hhh$ agrees with that of $\tmc^\hhh-\lagrange\area^\hhh$, these two formulations are no longer equivalent at second-order.
     \begin{proposition}[Characterization of stability]\label{carstabprop}
          Let $S\subseteq\hh^1$ be an embedded closed surface. Assume that $S$ is an area-preserving critical point along non-characteristic variations. The following are equivalent:
          \begin{itemize}
              \item [(i)] $S$ is \emph{area-preserving stable along non-characteristic variations}, i.e., by definition,
              \begin{equation*}
                  \delta^2\tmc^\hhh(S)[\Phi]\geq 0
              \end{equation*}
              for any area-preserving non-characteristic variation $\Phi$;
              \item [(ii)] if $\lagrange$ is as in \eqref{derivationeulerlagrangeminkgeneral}, then
              \begin{equation*}
                  \delta^2\tmc^\hhh(S)[\Phi]-\lagrange\delta^2\area^\hhh(S)[\Phi]\geq 0
              \end{equation*}
              for any first-order area-preserving non-characteristic variation $\Phi$. 
          \end{itemize}
     \end{proposition}
     \begin{proof}
         The implication (ii)$\implies$(i) is trivial. We prove (i)$\implies$(ii). Let $\Phi$ be a first-order area-preserving, non-characteristic variation. Denote by $\X$ and $\Z$ its velocity and acceleration respectively, and let $\varphi,\psi$ be as in \eqref{lefunzioniimportanti}. By \Cref{docarmo}, there exists an area-preserving non-characteristic variation $\tilde \Phi$ with velocity $\X$ and a suitable acceleration $\tilde \Z$. Accordingly, if $\tilde\varphi,\tilde \psi$ are as in \eqref{lefunzioniimportanti}, then $\tilde\varphi=\varphi$. Since $\tilde \Phi$ is area-preserving,  
         \begin{equation}\label{auxunocarastab}         
                 0=\delta^2\area^\hhh(S)[\tilde\Phi]
\overset{\eqref{variazsecondaareastatementsubriem}}{=}\delta\area^\hhh(S)[\tilde\psi]+\int_S\Big(\left(J(\vh)\varphi\right)^2-4\varphi^2\left(J(\vh)\alpha+\alpha^2\right)\Big)\,d\sigma^\hhh.
         \end{equation}
         Moreover, by \Cref{characterizationcriticalpointsproposition},
         \begin{equation}\label{auxduecarastab}
             \delta\tmc^\hhh(S)[\psi]-\lagrange\delta\area^\hhh(S)[\psi]=0= \delta\tmc^\hhh(S)[\tilde \psi]-\lagrange\delta\area^\hhh(S)[\tilde \psi].
         \end{equation}
         Therefore
         \begin{equation*}
             \begin{split}
                 \delta^2\tmc^\hhh(S)[\Phi]-\lagrange\delta^2\area^\hhh(S)[\Phi]\overset{\eqref{variazsecondaareastatementsubriem},\eqref{variazsecondatmcsubriemmainstatement}}&{=}\delta\tmc^\hhh(S)\left[\psi\right]+4\int_S\Big(-\s\varphi\,J(\vh)\varphi+\varphi^2\left(2\s\alpha-H^\hhh\alpha^2\right)\Big)\,d\sigma^\hhh\\
                 &\quad                 -\lagrange\delta\area^\hhh(S)\left[\psi\right]-\lagrange\int_S\Big(\left(J(\vh)\varphi\right)^2-4\varphi^2\left(J(\vh)\alpha+\alpha^2\right)\Big)\,d\sigma^\hhh\\
                 \overset{\eqref{auxduecarastab}}&{=}\delta\tmc^\hhh(S)[\tilde \psi]+4\int_S\Big(-\s\varphi\,J(\vh)\varphi+\varphi^2\left(2\s\alpha-H^\hhh\alpha^2\right)\Big)\,d\sigma^\hhh\\
                 &\quad                 -\lagrange\delta\area^\hhh(S)[\tilde \psi]-\lagrange\int_S\Big(\left(J(\vh)\varphi\right)^2-4\varphi^2\left(J(\vh)\alpha+\alpha^2\right)\Big)\,d\sigma^\hhh\\
                 \overset{\eqref{auxunocarastab}}&{=}\delta\tmc^\hhh(S)[\tilde \psi]+4\int_S\Big(-\s\varphi\,J(\vh)\varphi+\varphi^2\left(2\s\alpha-H^\hhh\alpha^2\right)\Big)\,d\sigma^\hhh\\
                 \overset{\eqref{variazsecondatmcsubriemmainstatement}}&{=}\delta^2\tmc^\hhh(S)[\tilde\Phi],
             \end{split}
         \end{equation*}
         and the thesis follows by (i).
     \end{proof}
\section{Total mean curvature: critical points}\label{sec_rotinvcritpointstmc}
In this section, we fully characterize mean convex area-preserving critical points of $\tmc^\hhh$. 
\subsection{Rotationally invariant critical points}
We first deal with rotationally invariant critical points. 
Under rotational invariance, we explicitly solve the Euler-Lagrange equation \eqref{derivationeulerlagrangeminkgeneral}. By \eqref{alfasurotinvggggg} and \eqref{jvinlocalcoords},
\begin{equation}\label{lemmajvalfapiualfaquadro}
         J(\vh)\alpha+\alpha^2=\dot t(s)\left(\frac{\dot x(s)^2\dot t(s)-x(s)\left(\dot x(s)\ddot t(s)-\ddot x(s)\dot t(s)\right)}{\left(\Dot t(s)^2+x(s)^2\Dot x(s)^2\right)^2}\right).
    \end{equation}
%\begin{lemma}
 %   It holds that
 %   \begin{equation}\label{lemmajvalfapiualfaquadro}
  %       J(\vh)\alpha+\alpha^2=\dot t(s)\left(\frac{\dot x(s)^2\dot t(s)-x(s)\left(\dot x(s)\ddot t(s)-\ddot x(s)\dot t(s)\right)}{\left(\Dot t(s)^2+x(s)^2\Dot x(s)^2\right)^2}\right)
  %  \end{equation}
  %  on $S\setminus S_0$.
%\end{lemma}
%\begin{proof}
%   By \eqref{alfasurotinvggggg} and \eqref{jvinlocalcoords},
   % \begin{equation*}
   %     \begin{split}
   %         J(\vh)\alpha
    %        %&            =-\left(\frac{x}{\sqrt{\Dot t^2+x^2\Dot x^2}}\right)\frac{\partial \alpha}{\partial s}\\
     %       =\left(\frac{x}{\sqrt{\Dot t^2+x^2\Dot x^2}}\right)\frac{d}{ds}\left(\frac{\dot x}{\sqrt{\Dot t^2+x^2\Dot x^2}}\right)
 %         %  &=\left(\frac{x}{\sqrt{\Dot t^2+x^2\Dot x^2}}\right)\left(\frac{\ddot x\left(\Dot t^2+x^2\Dot x^2\right)-\dot x\left(\dot t\ddot t+x\dot x^3+x^2\dot x\ddot x\right)}{\left(\Dot t^2+x^2\Dot x^2\right)^{\frac{3}{2}}}\right)\\
     %       =-\frac{x\dot t(\dot x\ddot t-\ddot x\dot t)+x^2\dot x^4}{\left(\Dot t^2+x^2\Dot x^2\right)^2}.
    %    \end{split}
   % \end{equation*}
 %  Recalling again \eqref{alfasurotinvggggg}, the thesis follows.
 %%   \begin{equation*}
   %     \langle J(\vh),\nabla\alpha\rangle+\alpha^2=\frac{\dot x^2\dot t^2-x\dot t(\dot x\ddot t-\ddot x\dot t)}{\left(\Dot t^2+x^2\Dot x^2\right)^2}.
 %   \end{equation*}
  %  \end{proof}
 Therefore, combining \eqref{hotinvinlemma} with \eqref{lemmajvalfapiualfaquadro}, then \eqref{derivationeulerlagrangeminkgeneral} reads as
 \begin{equation*}
     -4\dot t\left(\frac{\dot x^2\dot t-x\left(\dot x\ddot t-\ddot x\dot t\right)}{\left(\Dot t^2+x^2\Dot x^2\right)^2}\right)=\lagrange\left(\dfrac{x^3(\dot{x}\ddot{t}-\ddot{x}\dot{t})+\dot{t}^3}{x\left(x^2\dot{x}^2+\dot{t}^2\right)^{3/2}}\right).
 \end{equation*}
 Equivalently, letting $L\in\rr$ be such that $\lagrange=4L$, we need to solve
 \begin{equation}\label{EUlerLagrange}
     x\dot t\Big(\dot x^2\dot t-x\left(\dot x\ddot t-\ddot x\dot t\right)\Big)+L\sqrt{x^2\dot{x}^2+\dot{t}^2}\left(x^3(\dot{x}\ddot{t}-\ddot{x}\dot{t})+\dot{t}^3\right)=0.
 \end{equation}
 First, we show that there are no rotationally invariant closed surfaces satisfying \eqref{EUlerLagrange} with $L=0$. 
 \begin{proposition}\label{caralugualeazeroprop}
    Let $S$ be complete and without boundary. Assume that  \eqref{EUlerLagrange} holds with $L=0$. Then:
     \begin{itemize}
     \item [(i)] either $S$ is, up to vertical translations, the horizontal plane of \Cref{horizontalplanewxample};
         \item [(ii)] or $S$ is, up to dilations, the vertical cylinder of \Cref{exvertcil};
         \item [(iii)] or the profile of $S$ is $\gamma(s)=(\sqrt{as+b},\pm s)$ for some $a\neq 0$, $b\in\rr$, and  any $s\in\rr$ with $as+b>0$.
     \end{itemize}
 \end{proposition}
 \begin{proof}
     Let $\gamma=(x,t):I\to\rr^2$ be the profile of $S$. Since $x>0$ on $I$, then $\dot t\big(\dot x^2\dot t-x\left(\dot x\ddot t-\ddot x\dot t\right)\big)=0$ on $I$. If $\dot t(\hat s)=0$ for some $\hat s\in I$, then, by uniqueness, $\dot t\equiv 0$. In this case, $S$ is, up to vertical translation, the horizontal plane of \Cref{horizontalplanewxample}. Assume instead $\dot t\neq 0$ on I. In this case we can choose either $\gamma(s)=(x(s),s)$ or $\gamma(s)=(x(s),-s)$. In both cases,
     \begin{equation*}
         0\overset{\eqref{EUlerLagrange}}{=}\dot x^2+x\ddot x=\frac{d}{ds}\left(x\dot x\right)=\frac{1}{2}\frac{d^2}{ds^2}\left(x^2\right).
     \end{equation*}
     Therefore, $x(s)^2=as+b$, for some $a,b\in\rr$. 
     If $a=0$, then $S$ is, up to dilations, the vertical cylinder of \Cref{exvertcil}. If $a\neq 0$, then $x(s)=\sqrt{as+b}$, where $I$ is defined by $as+b>0$. The thesis follows.
 \end{proof}
 Next, we provide an \emph{a priori} bound for $L$ in order for \eqref{EUlerLagrange} to be satisfied by a closed surface. Moreover, we show that closed solutions to \eqref{EUlerLagrange} are the union of two vertical graphs.
 \begin{proposition}\label{punticriticisonograficiproposizione}
     Let $S$ be a rotationally invariant closed surface. Since $S$ is closed, there exists $\hat s$ such that
     \begin{equation*}
         x(\hat s)=\max\{x(s),s\in \overline I\}. 
     \end{equation*}
     Up to dilations, we assume that $x(\hat s)=1$. If $S$ solves \eqref{EUlerLagrange}, then $L\in(0,1)$. If in addition $S$ is mean convex, $S$ is the union of two vertical graphs. 
 \end{proposition}
 \begin{proof}
    \noindent\textbf{Step 1.} Let $\hat s$ be as in the statement. In particular, $\dot x(\hat s)=0$ and $\ddot x(\hat s)\leq 0$. Since the parametrization is counterclockwise, then $\dot t(\hat s)>0$. By evaluating \eqref{EUlerLagrange} at $\hat s$, we infer that
     \begin{equation*}
         \ddot x(\hat s)\dot t(\hat s)^2+L\dot t(\hat s)\left(-\ddot x(\hat s)\dot t(\hat s)+\dot t(\hat s)^3\right)=0.
     \end{equation*}
     Since $\dot t(\hat s)\neq 0$, then $        \ddot x(\hat s)+L\left(-\ddot x(\hat s)+\dot t(\hat s)^2\right)=0$, 
     that is 
     \begin{equation}\label{auxprooftwographs}
         (1-L)\ddot x(\hat s)=-L\dot t(\hat s)^2. 
     \end{equation}
     By \Cref{caralugualeazeroprop}, the right hand side of \eqref{auxprooftwographs} is not zero. Therefore $L\neq 1$, and moreover
     \begin{equation*}
         \ddot x(\hat s)=-\left(\frac{L}{1-L}\right)\dot t(\hat s)^2.
     \end{equation*}
     Since $\ddot x(\hat s)\leq 0$, then $\frac{L}{1-L}\geq 0$. Recalling that $L\neq 0$ and $L\neq 1$, we conclude that $L\in(0,1)$. Assume that $S$ is mean convex. We prove that it is the union of two vertical graphs. Without loss of generality, assume that $\gamma$ is parametrized by arc-length, so that $\kappa=\dot x\ddot t-\ddot x\dot t$ is the curvature of $\gamma$. Notice that a rotationally invariant, closed surface is homeomorphic either to a sphere or to a torus.

     \smallskip
     \noindent\textbf{Step 2.} Assume first that $S$ is homeomorphic to a sphere. We claim that $\hat s$ is the unique point in $I$ such that $\dot x(\hat s)=0$. Assume not by contradiction. Let $s_1\neq\hat s$ be such that $\dot x(s_1)=0$. Then there exists $\tilde s\in I$ such that $\dot x(\tilde s)=0$ and $\kappa(\tilde s)\leq 0$. Indeed, if $\kappa(s_1)\leq 0$, just choose $\tilde s=s_1$. Otherwise, there exists $s_2\in I$ satisfying the desired properties, and in this case we set $\tilde s=s_2$ 
     %\footnote{io metterei un disegnino, in cui si distingue il caso in cui $\dot t(s_1)>0$ e $\dot t(s_1)<0$}. 
     Since $\dot x(\tilde s)=0$, then $\dot t(\tilde s)\neq 0$. We claim that $\dot t(\tilde s)>0$. If not, then 
     \begin{equation*}
         H^\hhh(\tilde s)=\dfrac{x^3(\tilde s)\kappa(\tilde s)+\dot{t}^3(\tilde s)}{x(\tilde s)|\dot{t}(\tilde s)|^{3}}<0,
     \end{equation*}
   a contradiction with the fact that $S$ is mean convex. In particular, $
       0\geq \kappa(\tilde s)=-\ddot x(\tilde s)\dot t(\tilde s)$, 
   whence $\ddot x(\tilde s)\geq 0$. Evaluating \eqref{EUlerLagrange} at $\tilde s$, and since $\dot x(\tilde s)=0$ and $\dot t(\tilde s)>0$, we get
   \begin{equation*}
       x(\tilde s)^2\ddot x(\tilde s)\dot t(\tilde s)^2+L\dot t(\tilde s)\left(-x(\tilde s)^3\ddot x(\tilde s)\dot t(\tilde s)+\dot t(\tilde s)^3\right)=0,
   \end{equation*}
   whence
   \begin{equation}\label{contradictionsphere}
      x(\tilde s)^2 \ddot x(\tilde s)\left(1-L x(\tilde s)\right)=-L\dot t(\tilde s)^2.
   \end{equation}
By definition of $\hat s$, then $x(\tilde s)\leq 1$. Since $L\in(0,1)$, then $1-Lx(\tilde s)>0$. Therefore, \eqref{contradictionsphere} implies that $\ddot x(\tilde s)<0$, a contradiction. Therefore, the profile of $S$ is the union of two vertical graphs over the interval $(0,\hat s)$.

\smallskip
\textbf{Step 3.} Finally, assume that $S$ is homeomorphic to a torus. Let $\check s\in I$ be such that $x(\check s)=\min\{x(s)\,:\,s\in \overline{I}\}$. 
We claim that $\check s$ and $\hat s$ are the unique points where $\dot x=0$. If not, one can argue as above to infer the existence of $\tilde s\in I$ such that $\dot x(\tilde s)=0$ and $\kappa(\tilde s)\leq 0$. Then, since $S$ is mean convex, we deduce as above that $\dot t(\tilde s)>0$. In particular, $\ddot x(\tilde s)\geq 0$. Arguing \emph{verbatim} as above, we conclude that $\ddot x(\tilde s)<0$, reaching a contradiction. Again, the profile of $S$ is the union of two vertical graphs over the interval $(\check s,\hat s)$. The thesis follows.  
 \end{proof}
 By \Cref{punticriticisonograficiproposizione}, mean convex solutions to \eqref{EUlerLagrange} are the union of an upper vertical graph and a lower vertical graph. Denote respectively by $\gamma^+=(s,t^+(s))$ and $\gamma^-=(s,t^-(s))$ their profile. Notice that $\gamma^+$ is parametrized clockwise, while $\gamma^-$ is parametrized counterclockwise. When $S$ is homeomorphic to a sphere, $\gamma^+,\gamma^-:(0,1)\to\rr^2$. Moreover, since $S$ is at least of class $C^1$, then
 \begin{equation*}
     \lim_{s\to1^-}\dot t^-(s)=+\infty,\qquad   \lim_{s\to 1^-}\dot t^+(s)=-\infty.
 \end{equation*}
 Instead, when $S$ is homeomorphic to a torus, then $\gamma^+,\gamma^-:(\check s,1)\to\rr^2$ for some $\check s\in(0,1)$, and moreover
  \begin{equation*}
     \lim_{s\to1^-}\dot t^-(s)=+\infty,\qquad \lim_{s\to \check s^+}\dot t^-(s)=-\infty,\qquad  \lim_{s\to 1^-}\dot t^+(s)=-\infty,\qquad  \lim_{s\to \check s^+}\dot t^+(s)=+\infty.
 \end{equation*}
 Notice that, when a profile $\gamma$ admits a graphical parametrization $\gamma(s)=(s,t(s))$, \eqref{EUlerLagrange} simplifies as 
\begin{equation}\label{EUlerLagrangetgraf}
            s\dot t(s)^2-s^2\dot t(s)\ddot t(s)+L \sqrt{\Dot t(s)^2+s^2}\left(s^3\ddot{t}(s)+\dot{t}(s)^3\right)=0.
        \end{equation}
        In particular, setting $w(s)=\dot t(s)$, \eqref{EUlerLagrangetgraf} reads as \begin{equation}\label{EUlerLagrangetgrafder}
            s w(s)^2-s^2w(s)\dot w(s)+L \sqrt{w(s)^2+s^2}\left(s^3\dot{w}(s)+w(s)^3\right)=0.
        \end{equation}
Our final step consists in characterizing all possible solutions to \eqref{EUlerLagrangetgrafder} under the above-mentioned boundary verticality conditions. We focus on $\gamma_-$, being the characterization of $\gamma_+$ completely analogous. First, we consider the case in which $S$ is homeomorphic to a sphere.

        \begin{lemma}\label{lemmacaratterizzaionegraficacp}
          Let $L\in(0,1)$. Then the system  
\begin{equation}\label{sistemacauchy}
\left\{
\begin{aligned}
& s\,w(s)^2 - s^2 w(s)\,\dot w(s)
+ L\sqrt{w(s)^2+s^2}\,\left(s^3 \dot w(s)+w(s)^3\right)=0, \\[4pt]
& \lim_{s\to 1^-} w(s)=+\infty
\end{aligned}
\right.
\end{equation}
is solvable in $C^1(0,1)$ if and only if $L\in\left(0,\frac{1}{2}\right]$. The solution is unique, and it is given explicitly by           \begin{equation}\label{aistatementper48348348}
    w^-_L(s) = \frac{s}{\sqrt{2L}}\sqrt{
\frac{L s^2-2L+1+s \sqrt{L^2 s^2 -2L+1  }}{1-s^2 }
}.
\end{equation}
        \end{lemma}
\begin{proof}
    When $L\in\left(0,\frac{1}{2}\right]$, simple computation shows that $w^-_L$ is well-defined and solves \eqref{sistemacauchy}. Conversely, fix $L\in(0,1)$ and let $w$ be a solution to \eqref{sistemacauchy}. Set 
    \begin{equation*}
        z(s)=\frac{w(s)}{\sqrt{w(s)^2 +s^2}},\qquad s\in(0,1).
    \end{equation*}
    Then $z\in C^1(0,1)$, and  
    \begin{equation*}
        \dot z(s)=\dfrac{1}{w(s)^2 +s^2}\left(\dot w(s)\sqrt{w(s)^2 +s^2}-w(s)\left(\dfrac{w(s)\dot w(s)+s}{\sqrt{w(s)^2 +s^2}}\right)\right)=s\left(\dfrac{\dot w(s)s-w(s)}{\left(w(s)^2 +s^2\right)^\frac{3}{2}}\right).
    \end{equation*}
    In particular,
    \begin{equation}\label{auxinprufouniejfns}
        w(s)=z(s)\sqrt{w(s)^2 +s^2},\qquad \dot w(s)s-w(s)=\frac{\dot z(s)\left(w(s)^2 +s^2\right)^\frac{3}{2}}{s}.
    \end{equation}
    Therefore,
    \begin{equation*}
        \begin{split}
            0\overset{\eqref{sistemacauchy}}&{=} sw(s)^2 - s^2 w(s)\dot w(s)
+ L\sqrt{w(s)^2+s^2}\,\left(s^3 \dot w(s)+w(s)^3\right)\\
&= sw(s)\left(w(s)- s\dot w(s)\right)
+ L\sqrt{w(s)^2+s^2}\,\left(s^2\left(s \dot w(s)-w(s)\right)+w(s)\left(w(s)^2 +s^2\right)\right)\\
\overset{\eqref{auxinprufouniejfns}}&{=}-z(s)\dot z(s)\left( w(s)^2 +s^2\right)^2+L\sqrt{w(s)^2 +s^2}\left(s\dot z(s)\left( w(s)^2 +s^2\right)^\frac{3}{2}+z(s)\left( w(s)^2 +s^2\right)^\frac{3}{2}\right)\\
&=\left( w(s)^2 +s^2\right)^2\Big(-z(s)\dot z(s)+L\left(s\dot z(s)+z(s)\right)\Big)\\
&=-\frac{1}{2}\left( w(s)^2 +s^2\right)^2\frac{d}{ds}\Big(z(s)^2 -2L sz(s)\Big).
        \end{split}
    \end{equation*}
    Since $\left( w(s)^2 +s^2\right)^2\neq 0$ on $(0,1)$, then there exists $a\in\rr$ such that 
    \begin{equation}\label{parabolainproof}
        z(s)^2 -2Ls z(s)+a=0,\qquad s\in(0,1).
    \end{equation}
    By the verticality condition of \eqref{sistemacauchy}, we deduce that $\lim_{s\to 1^-}z(s)=1$. Combining this information with \eqref{parabolainproof}, we deduce that $a=2L-1$, whence
    \begin{equation}\label{parabolainproofdue}
        z(s)^2 -2Ls z(s)+2L-1=0,\qquad s\in(0,1).
    \end{equation}
    Notice that \eqref{parabolainproofdue} is a quadratic equation in $z(s)$, whose discriminant is given by $4L^2 s^2-4(2L-1)$. By \eqref{parabolainproofdue}, then, $4L^2 s^2-4(2L-1)\geq 0$ for any $s\in(0,1)$, whence we deduce that $L\leq\frac{1}{2}$. We conclude that in this case $w=w_L$. Indeed, since $L\leq\frac{1}{2}$, the discriminant $4L^2 s^2-4(2L-1)$ is positive for any $s\in(0,1)$, whence
    \begin{equation*}
        z(s)\equiv Ls+\sqrt{L^2s^2-2L+1}\qquad\text{or}\qquad z(s)\equiv Ls-\sqrt{L^2s^2-2L+1},\qquad s\in(0,1).
    \end{equation*}
    The second possibility can be discarded, because in that case $\lim_{s\to 1^-}z(s)=2L-1\neq 1$. We conclude that
    \begin{equation*}
        z(s)=Ls+\sqrt{L^2s^2-2L+1},\qquad s\in(0,1).
    \end{equation*}
    Since $0<z(s)<1$ for any $s\in(0,1)$, then $w(s)>0$ for any $s\in(0,1)$, and, by simple computations,
    \begin{equation*}
        w(s)=\frac{s z(s)}{\sqrt{1-z(s)^2 }}=w^-_L(s).
    \end{equation*}
\end{proof}
Finally, we show that $S$ cannot be homeomorphic to a torus.
   \begin{lemma}
          Let $L\in(0,1)$ and  $\check s\in(0,1)$. The system  
\begin{equation}\label{sistemacauchytorus}
\left\{
\begin{aligned}
& s\,w(s)^2 - s^2 w(s)\,\dot w(s)
+ L\sqrt{w(s)^2+s^2}\,\left(s^3 \dot w(s)+w(s)^3\right)=0, \\[4pt]
& \lim_{s\to \check s^+} w(s)=-\infty
 \\[4pt]
& \lim_{s\to 1^-} w(s)=+\infty
\end{aligned}
\right.
\end{equation}
is not solvable in $C^1(\check s,1)$.
        \end{lemma}
        \begin{proof}
           Assume by contradiction that $w$ solves \eqref{sistemacauchytorus} for some $L,\check s\in(0,1)$. Let $z$ be as in the proof of \Cref{lemmacaratterizzaionegraficacp}. Arguing \emph{verbatim} as above, the equation and the verticality condition at $1$ imply   \begin{equation}\label{auxunoesxluditoro}
        z(s)^2 -2Ls z(s)+2L-1=0,\qquad s\in(\check s,1).
    \end{equation}
    Moreover, the verticality condition at $\check s$ and the definition of $z$ imply
    \begin{equation}\label{auxdueescluditoro}
        \lim_{s\to\check s^+}z(s)=-1.
    \end{equation}
    By \eqref{auxunoesxluditoro} and \eqref{auxdueescluditoro} it follows that $ 2L(\check s+1)=0,$
    whence either $L=0$ or $\check s=-1$, a contradiction.
        \end{proof}
        In the same way, the unique possible profiles of vertical upper graphs are of the form $\gamma^+_L(s)=(s,-t^-_L(s))$ for $L\in\left(0,\frac{1}{2}\right]$. Therefore, $S$ is the union of two vertical graphs with profiles $\gamma^+_L$ and $\gamma^-_{L}$ for some $L\in\left(0,\frac{1}{2}\right]$.  We then proved the following rigidity statement. 
       \begin{proposition}\label{rigiditymapergraicidmrvojfrn4fijerneitvjn}
           Let $S$ be rotationally invariant, closed, mean convex. The following are equivalent:
           \begin{itemize}
               \item [(i)] $S$ is an area-preserving critical point for $\tmc^\hhh$ along non-characteristic variations;
               \item [(ii)] $S$ is the union of two vertical graphs with profiles $\gamma^+_L$ and $\gamma^-_L$ for some $L\in\left(0,\frac{1}{2}\right]$.
           \end{itemize}          
       \end{proposition}
       For $L\in\left(0,\frac{1}{2}\right]$, denote the corresponding critical point by $S_L$. We first show that it is possible to describe them by means of a global parameterization. 
      % To this aim, fix  $L\in\left(0,\frac{1}{2}\right].$ 
      Let $\gamma_L=(x_L,t_L):[-\arccos\sqrt{1-2L},\arccos\sqrt{1-2L}]\to\rr^2$ be the regular parametrization given by 
\begin{equation}\label{curpamsejnefcjnecvjn}
\begin{cases}
x_L(s)=\dfrac{1}{2L}\left(\cos s-\dfrac{1-2L}{\cos s}\right), \\[1em]
t_L(s)=\dfrac{1}{4L^2}\left(\dfrac{s}{2}+\dfrac{\sin 2s}{4}-(1-2L)^2\tan s\right).
\end{cases}
\end{equation}
We show that $\gamma_L$ is indeed a parametrization of $S_L$.
We write $x_L=x$ and $t_L=t$, and we set $k=1-2L$. 
Set 
\begin{equation}\label{cosefinalidefofhfoga}
    h(s)=\frac{1}{2L}\left(\cos s+\frac{k}{\cos s}\right).
\end{equation}
Notice that both $x$ and $h$ are positive on $(-\arccos \sqrt{k},\arccos \sqrt{k})$. Moreover, 
\begin{equation*}
    \dot h(s)=\frac{1}{2L}\left(-\sin s+\frac{k\sin s}{\cos^2 s}\right)=-\tan s\, x(s).
\end{equation*}
Therefore
\begin{equation*}\label{xdottdoteastereggs}
\begin{cases}
\dot x(s)=\dfrac{1}{2L}\left(-\sin s-\dfrac{k\sin s}{\cos^2 s}\right)=-\tan s \,h(s), \\[1em]
\dot t(s)=\dfrac{1}{4L^2}\left(\cos^2 s-\dfrac{k^2}{\cos^2 s}\right)=x(s)\, h(s),
\end{cases}
\qquad 
\begin{cases}
\ddot x(s)=-\dfrac{h(s)}{\cos^2 s}+\tan^2 s\,x(s), \\[1em]
\ddot t(s)=-\tan s\,h(s)^2 -\tan s\, x(s)^2 .
\end{cases}
\end{equation*}
In particular
\begin{equation}\label{auxpredhsigma}
    x(s)^2 \dot x(s)^2 +\dot t(s)^2 =\tan^2 s\,x(s)^2 \,h(s)^2 +x(s)^2 \,h(s)^2 =\frac{x(s)^2 \,h(s)^2 }{\cos^2 s}
\end{equation}
and 
\begin{equation}\label{auxcurvaturaeuclideapansumink}
    \dot x(s)\ddot t(s)-\ddot x(s)\dot t(s)=\tan(s)^2 h(s)^3+\frac{x(s)h(s)^2}{\cos^2s}.
\end{equation}
A simple check reveals that $\gamma_L$ solves \eqref{EUlerLagrange}. By \Cref{rigiditymapergraicidmrvojfrn4fijerneitvjn}, we conclude that $\gamma_L$ is the profile of $S_L$.
\begin{remark}\label{pansuiscriticalpoint}
    Notice that 
    $$\gamma_\frac{1}{2}(s)=\left(\cos s,\frac{s}{2}+\frac{\sin 2s}{4}\right),\qquad s\in\left(-\frac{\pi}{2},\frac{\pi}{2}\right)$$
    is precisely the Pansu sphere as in \Cref{examplepansu}. In particular, \Cref{rigiditymapergraicidmrvojfrn4fijerneitvjn} shows that the Pansu sphere is an area-preserving critical point for $\tmc^\hhh$ under non-characteristic variations. However, \Cref{pansunotmin} demonstrates that the non-characteristic condition is essential: it constructs a variation of the Pansu sphere that fixes the characteristic points but is not non-characteristic, and for which the first variation does not vanish.
\end{remark}
\begin{remark}\label{remark_hkelforpansuinremarkfrefv}
    It is interesting to observe that the Pansu sphere is also a critical point for the volume-constrained Heintze-Karcher problem \eqref{optiprobdue}. Indeed, by \Cref{teoremavariazionigeneralimainsubriem} and \cite[Proposition 1.19]{MR4676392}, the Euler-Lagrange equation associated with the minimization of $\timc$ under volume constraint reads as 
    \begin{equation}\label{hkelforpansuinremark}
     \hat\Delta^{\hhh,S}\left(\frac{1}{\left(H^\hhh\right)^2}\right)+\frac{4\left(J(\vh)\alpha+\alpha^2\right)}{\left(H^\hhh\right)^2}+2
      =\lagrange,
    \end{equation}
    where $\lagrange\in\rr$ is the Lagrange multiplier arising from the volume constraint. Since the Pansu sphere has constant mean curvature and satisfies  \eqref{derivationeulerlagrangeminkgeneral}, then $J(\vh)\alpha+\alpha^2$ is constant: the Pansu sphere satisfies \eqref{hkelforpansuinremark}.
\end{remark}
Motivated by \Cref{pansuiscriticalpoint}, we refer to our critical points as \emph{Pansu-Minkowski spheres}. It is then natural to understand, among all Pansu-Minkowski spheres, which is a critical configuration for \eqref{minkquointro}, as well as which is the optimal shape for the same problem. Notice that 
\begin{equation}\label{rev_vfforpm38385858585}
        \begin{split}
J(\vh)\overset{\eqref{jvinlocalcoords}}&{=}-\left(\frac{x}{\sqrt{\Dot t^2+x^2\Dot x^2}}\right)\frac{\partial}{\partial s}+\left(\frac{\dot t}{x\sqrt{\Dot t^2+x^2\Dot x^2}}\right)\frac{\partial}{\partial \theta}=-\left(\frac{\cos s}{h}\right)\frac{\partial}{\partial s}+\frac{\cos s}{x}\frac{\partial}{\partial \theta},\\
        \s\overset{\eqref{sinloccords}}&{=}\left(\frac{\dot t}{\Dot t^2+x^2\Dot x^2}\right)\frac{\partial}{\partial s}+\left(\frac{\dot x^2}{\Dot t^2+x^2\Dot x^2}\right)\frac{\partial}{\partial \theta}=\left(\frac{\cos^2 s}{xh}\right)\frac{\partial}{\partial s}+\left(\frac{\sin^2 s}{x^2}\right)\frac{\partial}{\partial\theta}.
           \end{split}
    \end{equation}
    Moreover, 
\begin{equation}\label{desigmahforpmsphereserijenrviervjnc}
    \begin{split}
        d\sigma^\hhh        \overset{\eqref{auxpredhsigma}}&{=}\frac{x(s)^2 h(s)}{\cos s}\,ds\,d\theta
   %     &=\frac{1}{8L^3\cos s}\left(\cos^2-\frac{k^2}{\cos^2}\right)\left(\cos s-\frac{k}{\cos s}\right)\,ds\,d\theta\\
        =\frac{1}{8L^3}\left(\cos^2 s-k-\frac{k^2}{\cos^2s}+\frac{k^3}{\cos^4 s}\right)\,ds\,d\theta.
    \end{split}
\end{equation}
In addition,
\begin{equation}\label{hmeanperpansumink}
    \begin{split}
        H^\hhh\overset{\eqref{auxcurvaturaeuclideapansumink}}{=}\frac{\cos(s)^3  }{x(s)h(s)^3  }\left(\tan ^2s\,h(s)^3  +\frac{x(s)h(s)^2 }{\cos(s)^2 }+h(s)^3  \right)=\frac{x(s)+h(s)}{x(s)h(s)}\cos(s)=\frac{4L\cos^4 s}{\cos^4 s-k^2}
    \end{split}
\end{equation}
whence 
\begin{equation}\label{hdesigmah}
    H^\hhh\,d\sigma^\hhh=x(s)\left(x(s)+h(s)\right)=\frac{1}{2L^2}\left(cos^2 s-k\right)\,ds\,d\theta.
\end{equation}
Therefore
\begin{equation}\label{calculoareaeasteregggggggg}
    \begin{split}
        \area^\hhh(S_L)&=\frac{\pi}{4L^3}\int_{-\arccos{\sqrt{k}}}^{\arccos\sqrt{k}}\left(\cos^2 s-k-\frac{k^2}{\cos^2s}+\frac{k^3}{\cos^4 s}\right)\,ds\\
        &=\frac{\pi}{4L^3}\int_{-\arccos{\sqrt{k}}}^{\arccos\sqrt{k}}\left(\frac{1+\cos 2s}{2}-k-k^2\frac{d}{ds}(\tan s)+k^3\frac{d}{ds}(\tan s)\left(1+\tan^2 s\right)\right)\,ds\\
        &=\frac{\pi}{2L^3}\left[\frac{s}{2}+\frac{\sin s\cos s}{2}-ks-k^2\tan s+k^3\tan s+\frac{k^3}{3}\tan^3 s\right]_0^{\arccos\sqrt{k}}\\
        &=\frac{\pi}{2L^3}\left(\frac{(1-2k)\arccos\sqrt{k}}{2}+\frac{\sqrt{k(1-k)}}{2}-k^2(1-k)\sqrt{\frac{1-k}{k}}+\frac{k^3}{3}\left(\frac{1-k}{k}\right)^\frac{3}{2}\right)\\
        &=\frac{\pi}{4L^3}\left(\left(4L-1\right)\arccos\sqrt{1-2L}+\left(1-\frac{8L(1-2L)}{3}\right)\sqrt{2L(1-2L)}\right).
    \end{split}
\end{equation}
Moreover,
\begin{equation}\label{calculotmceastereggggggg}        \tmc^\hhh(S_L)=\frac{\pi}{L^2}\int_{-\arccos{\sqrt{k}}}^{\arccos\sqrt{k}}\left(\cos^2 s-k\right)\,ds=\frac{\pi}{L^2}\left((4L-1)\arccos\sqrt{1-2L}+\sqrt{2L(1-2L)}\right)
\end{equation}
Therefore, recalling that $  \rifun_\mathrm{mink}^\hhh\coloneqq\left(\area^\hhh\right)^{-2/3}\tmc^\hhh$, we conclude that
\begin{equation*}
    \rifun^\hhh(S_L)=2(2\pi)^{\frac{1}{3}}\frac{(4L-1)\arccos\sqrt{1-2L}+\sqrt{2L(1-2L)}}{\left((4L-1)\arccos\sqrt{1-2L}+\frac{1}{3}(3-8L+16L^2)\sqrt{2L(1-2L)}\right)^\frac{2}{3}}.
\end{equation*}
The next result highlights the distinguished role played by the \emph{optimal Pansu-Minkowski sphere $S_\frac{1}{4}$.} 
\begin{proposition}\label{lemmaincuisievinceellunquarto}
  The following holds.
    \begin{equation*}
        \rifun_\mathrm{mink}^\hhh(S_L)\geq (18\pi)^\frac{1}{3},\qquad   \rifun_\mathrm{mink}^\hhh(S_L)= (18\pi)^\frac{1}{3}\text{ if and only if }L=\frac{1}{4}.
    \end{equation*}
    Moreover, 
    \begin{equation}\label{espressionerigiditalagrange}
        \lagrange=\frac{2\tmc^\hhh(S_L)}{3\sigma^\hhh(S_L)}\text{ if and only if }L=\frac{1}{4},
    \end{equation}
    and $S_\frac{1}{4}$ is the unique rotationally invariant critical point of $ \rifun_\mathrm{mink}^\hhh$ under non-characteristic variations.
\end{proposition}
\begin{proof}
    We show that $S_\frac{1}{4}$ is the unique minimum point of $\rifun^\hhh_\mathrm{mink}$ within Pansu-Minkowski spheres. Consider the re-parametrization 
    \begin{equation*}
        \ell=2\arccos\sqrt{1-2L}\in(0,\pi].
    \end{equation*}
    A simple computation shows that
    \begin{equation*}
        (4L-1)\arccos{\sqrt{1-2L}}=-\frac{1}{2}\ell\cos\ell,\qquad
        \sqrt{2L(1-2L)}=\frac{1}{2}\sin\ell,\qquad
        \frac{1}{3}(3-8L+16L^2)=1-\frac{1}{3}\sin^2\ell.
    \end{equation*}
    Therefore
    \begin{equation*}
        \begin{split}
             \rifun_\mathrm{mink}^\hhh\left(S_{L(\ell)}\right)=(2\pi)^\frac{1}{3}\frac{\sin\ell-\ell\cos\ell}{\left(\frac{1}{2}\left(1-\frac{1}{3}\sin^2\ell\right)\sin\ell-\frac{1}{2}\ell\cos\ell\right)^\frac{2}{3}}
         %   &=4(2\pi)^\frac{1}{3}\frac{\sin\ell-\ell\cos\ell}{\left(4\sin\ell-\frac{4}{3}\sin^3\ell-4\ell\cos\ell\right)^\frac{2}{3}}\\
            =4(2\pi)^\frac{1}{3}\frac{\sin\ell-\ell\cos\ell}{\left(3\sin\ell+\frac{1}{3}\sin 3\ell-4\ell\cos\ell\right)^\frac{2}{3}},
        \end{split}
    \end{equation*}
    where in the last equality we exploited the identity $-4\sin^3\ell=\sin 3\ell-3\sin\ell$. Notice that $L\left(\frac{\pi}{2}\right)=\frac{1}{4}$. To conclude, it suffices to check that 
    \begin{equation}\label{segnoderivata}
        \frac{d \rifun_\mathrm{mink}^\hhh\left(S_{L(\ell)}\right)}{d\ell}<0\text{ on }\left(0,\frac{\pi}{2}\right),\qquad \left.\frac{d \rifun_\mathrm{mink}^\hhh\left(S_{L(\ell)}\right)}{d\ell}\right|_{\ell=\frac{\pi}{2}}=0,\qquad \frac{d \rifun_\mathrm{mink}^\hhh\left(S_{L(\ell)}\right)}{d\ell}>0\text{ on }\left(\frac{\pi}{2},\pi\right).
    \end{equation}
    To this aim,
    \begin{equation*}
        \begin{split}
            &\frac{d \rifun_\mathrm{mink}^\hhh\left(S_{L(\ell)}\right)}{d\ell}=4(2\pi)^\frac{1}{3}\dfrac{\ell\sin\ell\left(3\sin\ell+\frac{1}{3}\sin 3\ell -4\ell\cos\ell\right)-\frac{2}{3}\left(\sin\ell-\ell\cos\ell\right)\left(\cos 3\ell-\cos\ell+4\ell\sin\ell\right)}{\left(3\sin\ell+\frac{1}{3}\sin 3\ell-4\ell\cos\ell\right)^\frac{5}{3}}\\
            &\quad=\frac{4(2\pi)^\frac{1}{3}}{3}\dfrac{\sin 2\ell\left(-\frac{1}{2}\ell\sin 2\ell-2\ell^2-\cos 2\ell+2\sin^2\ell+1\right)+\ell\left(2\sin^2\ell\left(\frac{1+\cos 2\ell}{2}\right)-4\cos^2\ell\left(\frac{1-\cos 2\ell}{2}\right)\right)}{\left(3\sin\ell+\frac{1}{3}\sin 3\ell-4\ell\cos\ell\right)^\frac{5}{3}}\\
            &\quad=-\frac{4(2\pi)^\frac{1}{3}}{3}\sin 2\ell\left(\dfrac{\ell\sin 2\ell+2\ell^2+\cos 2\ell-2\sin^2\ell-1}{\left(3\sin\ell+\frac{1}{3}\sin 3\ell-4\ell\cos\ell\right)^\frac{5}{3}}\right).
        \end{split}
    \end{equation*}
    Then, \eqref{segnoderivata} follows if  $f(\ell)=\ell\sin 2\ell+2\ell^2+\cos 2\ell-2\sin^2\ell-1$ is positive on $(0,\pi)$. Observe that
    \begin{equation*}
        \dot f(\ell)=2\ell\cos 2\ell+4\ell-3\sin2\ell.
    \end{equation*}
    In particular, $f(0)=\dot f(0)=0.$ Then $f>0$ on $(0,\pi)$ provided that $\ddot f>0$ on $(0,\pi)$. To this aim,
    \begin{equation*}
        \ddot f(\ell)=4\left(1-\cos 2\ell-\ell\sin2\ell\right)=8\sin\ell\left(\sin\ell-\ell\cos\ell\right).
    \end{equation*}
   Therefore $\ddot f$ is positive. Then \eqref{segnoderivata} follows. We prove \eqref{espressionerigiditalagrange}. Combining \eqref{calculoareaeasteregggggggg} and \eqref{calculotmceastereggggggg},
   \begin{equation*}
       \lagrange-\frac{2\tmc^\hhh(S_L)}{3\sigma^\hhh(S_L)}=4L(4L-1)\left(\frac{\arccos\sqrt{1-2L}+(4L-1)\sqrt{2L(1-2L)}}{3\left(4L-1\right)\arccos\sqrt{1-2L}+(3-8L+16L^2)\sqrt{2L(1-2L)}}\right).
   \end{equation*}
   Then \eqref{espressionerigiditalagrange} follows provided that $\arccos\sqrt{1-2L}+(4L-1)\sqrt{2L(1-2L)}\neq 0$ for any $L\neq\frac{1}{4}$. Changing variables as above, 
   \begin{equation*}
       \arccos\sqrt{1-2L}+(4L-1)\sqrt{2L(1-2L)}=\frac{1}{4}\left(2\ell-\sin 2\ell\right), \qquad\ell\in(0,\pi],
   \end{equation*}
   whence \eqref{espressionerigiditalagrange} easily follows. Finally, fix an arbitrary non-characteristic variation. Then, recalling \Cref{propimplicazionenonequaivmeovtrbvnrtvonrgvo},
   \begin{equation}\label{bohrevisionifinalifoga}
   \begin{split}
        \left.\frac{d}{dt}\right|_{t=0} \rifun_\mathrm{mink}^\hhh\left(\Phi_t(S_L)\right)&=\sigma^\hhh(S_L)^{-\frac{2}{3}}\left(-\frac{2\tmc^\hhh(S_L)}{3\sigma^\hhh(S_L)}\left(\delta\sigma^\hhh(S_L)[\Phi]\right)+\left(\delta\tmc^\hhh(S_L)[\Phi]\right)\right)\\
       \overset{\eqref{variazioneprimazerotmcstatement}}&{=}\sigma^\hhh(S_L)^{-\frac{2}{3}}\left(\delta\sigma^\hhh(S_L)[\Phi]\right)\left(-\frac{2\tmc^\hhh(S_L)}{3\sigma^\hhh(S_L)}+\lagrange\right).
   \end{split}
   \end{equation}
By \eqref{espressionerigiditalagrange}, $S_\frac{1}{4}$ is a critical point. We claim that it is the unique rotationally invariant critical point. If not, let $S$ be another rotationally invariant critical point. By \Cref{propimplicazionenonequaivmeovtrbvnrtvonrgvo} and \Cref{rigiditymapergraicidmrvojfrn4fijerneitvjn}, $S=S_L$ for some $L\in\left(0,\frac{1}{2}\right]$ and $L\neq\frac{1}{4}$. Let $\Phi$ be a non-characteristic variation for which $\delta\sigma^\hhh(S_L)[\Phi]\neq 0$. Then \eqref{bohrevisionifinalifoga} and \eqref{espressionerigiditalagrange} grant that 
\begin{equation*}
    \left.\frac{d}{dt}\right|_{t=0} \rifun_\mathrm{mink}^\hhh\left(\Phi_t(S)\right)= \left.\frac{d}{dt}\right|_{t=0} \rifun_\mathrm{mink}^\hhh\left(\Phi_t(S_L)\right)\neq 0,
\end{equation*}
from which uniqueness follows.
\end{proof}
%By \Cref{lemmaincuisievinceellunquarto}, $S_\frac{1}{4}$ plays a distinguished role in the class of Pansu-Minkowski spheres. We refer to it as \emph{optimal Pansu-Minkowski sphere}. The results of this section can be summarized in the following statement.
%\color{purple} Lo ho spostato in intro, toglilo da qui, ma commenta il fatto che la sua proof segue dai risultati precedenti.
%\begin{theorem}
  % For any   $L\in\left(0,\frac{1}{2}\right]$, define $\gamma_L=(x_L,t_L):[-\arccos\sqrt{1-2L},\arccos\sqrt{1-2L}]\to\rr^2$ by 
%\begin{equation*}
%\begin{cases}
%x_L(s)=\dfrac{1}{2L}\left(\cos s-\dfrac{1-2L}{\cos s}\right), \\[1em]
%t_L(s)=\dfrac{1}{4L^2}\left(\dfrac{s}{2}+\dfrac{\sin 2s}{4}-(1-2L)^2\tan s\right).
%\end{cases}
%\end{equation*}
  %  Denote by $S_L$ its associated rotationally invariant surface. We call them \emph{Pansu-Minkowski spheres}.
 %   Let $S\subseteq\hh^1$ be a rotationally invariant, closed, mean convex surface. The following are equivalent:
 %   \begin{itemize}
     %      \item [(i)] $S$ is an-area-preserving critical point along non-characteristic variations;
      %         \item [(ii)] $S$ is, up to dilations and vertical translation,  a Pansu-Minkowski sphere $S_L$ for some $\left(0,\frac{1}{2}\right]$.
  %  \end{itemize}
 %   In addition, the optimal Pansu-Minkowski sphere $S_{\frac{1}{4}}$ minimizes $\mathcal \rifun^\hhh $ in the class of Pansu-Minkowski spheres:
     %   \begin{equation*}
     %   \rifun ^\hhh(S_L)\geq (18\pi)^\frac{1}{3},\qquad  \rifun ^\hhh(S_L)= (18\pi)^\frac{1}{3}\text{ if and only if }L=\frac{1}{4}.
  %  \end{equation*}
%\end{theorem}
%\color{black}
\subsection{Arbitrary critical points} We conclude by showing that the rotational invariance assumption in \Cref{rigiditymapergraicidmrvojfrn4fijerneitvjn} can be dropped.
\begin{proposition}\label{rotational_symmetry_critical_points}
    Let $S\subseteq\hh^1$ be embedded, closed and mean convex. Assume that $S$ is an area-preserving critical point for $\tmc^\hhh$ along non-characteristic variations. Then $S$ is, up to left-translations and intrinsic dilations, a Pansu-Minkowski sphere.
\end{proposition}
\begin{proof}\label{rotational_symmetry_proof_start}
 Let $\lagrange=4L$ be as in \Cref{characterizationcriticalpointsproposition}, so that
\begin{equation}\label{rotational_el}
    J(\vh)\alpha+\alpha^2+LH^\hhh=0\qquad\text{on }S\setminus S_0.
\end{equation}
We divide the proof into five steps.

\smallskip
\noindent\textbf{Step 1.} We prove that $S_0$ is finite. Assume that $S_0\neq \emptyset$.
Fix a characteristic point and translate it to the origin. Assume that the outward Riemannian normal is $\n(0)=-T$, the opposite case being analogous. Locally, $S=\{t=u(z)\}$, where $z=(x,y)$, $u(0)=0$ and $Du(0)=0$. Set
\begin{equation*}
    J(z)=(-y,x),\qquad F(z)=Du(z)+J(z),\qquad W(z)=|F(z)|.
\end{equation*}
With abuse of notation, denote by $J$ the skew matrix associated with $J$. Notice that $(z,t)\in S_0$ if and only if $F(z)=0$.
Throughout the proof, $D$, $D^2$ and $\divv$ refer to the Euclidean differential operators in $\rr^2$. Adopt $(x,y)$ as local coordinates on $S$. If $f$ is defined on $S$, we write $f(z)=f(z,u(z))$, and we extend it to $\hh^1$ by $f(z,t)=f(z)$. The same identification applies to vector fields. In this way
\begin{equation}\label{rev_espressioneinloccords}
    \vh=\frac{F}{W},\qquad\alpha=-\frac1W,\qquad J(\vh)=\frac{JF}{W},\qquad\s=\frac{F}{W^2},\qquad d\sigma^\hhh=W\,dx\,dy\qquad\text{on }F\neq 0.
\end{equation}
A direct computation and \eqref{rotational_el} give
\begin{equation}\label{rotational_graph_equation}
    H^\hhh=\divv\left(\frac{F}{W}\right)=\frac{\langle D^2u\,JF,JF\rangle}{W^3},\qquad
    \langle D^2u\,F,JF\rangle+LW\langle D^2u\,JF,JF\rangle=0\qquad\text{on }F\neq 0.
\end{equation}
 Set $A=D^2u(0)$. Let $w\in \mathrm{Im}(A+J)\setminus\{0\}$, so that $w=(A+J)v$ for some $v\in\rr^2$. Notice that, as $Du(0)=0,$
 \begin{equation*}
     Du(z)=Az+o(|z|),\qquad F(z)=(A+J)z+o(|z|)\qquad\text{as }z\to 0.
 \end{equation*}
 Therefore, if $\eps>0$,
 \begin{equation*}
     F(\eps v)=\eps(A+J)v+o(\eps)\qquad\text{as }\eps\to 0.
 \end{equation*}
 In particular, $F(\eps v)\neq 0$ when $\eps$ is sufficiently small. Evaluating the second equation in \eqref{rotational_graph_equation} at $\eps v$, dividing by $\eps^2$ and letting $\eps\to 0$ yields 
 \begin{equation}\label{rev_scal_1}
     \left\langle Aw,Jw\right\rangle=0\qquad\text{for any }w\in \mathrm{Im}(A+J).
 \end{equation}
 We prove that $A+J$ is invertible. Assume not. Then $0=\det (A+J)=\det(A)+1$, whence $\det(A)=-1$. On the other hand, a direct computation yields $AJA=\det(A) J=-J$. Therefore, if $w=(A+J)v\in\mathrm{Im}(A+J),$  
 \begin{equation}\label{rev_scal_due_aux_uno}
     AJw=AJ(A+J)v=(AJA+AJJ)v=-(A+J)v=-w.
 \end{equation}
Since $A$ is symmetric and $J$ is skew, then $A+J\neq 0$, whence there exists $\tilde w\in\mathrm{Im}(A+J)\setminus\{0\}$. Then
\begin{equation*}
    0\overset{\eqref{rev_scal_1}}{=}\langle A\tilde w,J\tilde w\rangle=\langle \tilde w,AJ\tilde w\rangle\overset{\eqref{rev_scal_due_aux_uno}}{=}-|\tilde w|^2,
\end{equation*}
which is impossible. Thus $A+J$ is invertible. Noticing that $DF(0)=A+J$, the inverse function theorem ensures that $F$ has an isolated zero at the origin. Compactness of $S$ gives that $S_0$ is finite. Moreover, since $A+J$ is invertible, \eqref{rev_scal_1} holds for any $w\in\rr^2$. Therefore, there exists $c\in\rr$ such that, $\text{as }z\to 0$,
\begin{equation}\label{rotational_pole_asymptotics}
\begin{split}
    D^2u(0)&=cI,\qquad 
    H^\hhh(z)=\frac{c}{\sqrt{1+c^2}\,|z|}+o\big(|z|^{-1}\big),\qquad c\geq0, \\
    Du(z)&=cz+o(|z|),\qquad D^2u(z)=cI+o(1),\qquad W(z)=\sqrt{1+c^2}|z|+o(|z|),
    \end{split}
\end{equation}
where the inequality for $c$ follows by mean convexity of $S$.

\smallskip
\noindent\textbf{Step 2.} We prove that $L>0$ and that $S_0\neq 0$. First, we claim that 
\begin{equation}\label{rotational_positive_multiplier}
    \int_S\alpha^2\,d\sigma^\hhh=L\int_S H^\hhh\,d\sigma^\hhh.
\end{equation}
Notice that \eqref{rotational_positive_multiplier} and mean convexity immediately yield $L\geq 0$. Moreover, if $L=0$, \eqref{rotational_positive_multiplier} would imply that $\alpha\equiv 0$ on $S\setminus S_0$. This would force $S$ to be vertical, contradicting its compactness. Therefore \eqref{rotational_positive_multiplier} implies $L>0$. In order to show \eqref{rotational_positive_multiplier},
set $\A=\alpha J(\vh)$. Notice that, on $S\setminus S_0$, 
\begin{equation}\label{rev_divergenza_A}
    \divv^{\hhh,S}\A+2\alpha\langle \A,J(\vh)\rangle =\left\langle \nabla^{\hhh,S}_{J(\vh)}\Big(\alpha J(\vh)\Big),J(\vh)\right\rangle+2\alpha^2=J(\vh)\alpha+2\alpha^2\overset{\eqref{rotational_el}}{=}\alpha^2-LH^\hhh.
\end{equation}
The integration by parts formulas of \cite[Proposition 3.10]{formuledivariazioneriem} and \eqref{rev_divergenza_A} give
\begin{equation}\label{rev_ibp_preeps}
    \int_S\varphi\Big( \alpha^2-LH^\hhh\Big)\,d\sigma^\hhh=-\int_S\alpha J(\vh)\varphi\,d\sigma^\hhh\qquad\text{for any $\varphi\in C^\infty_c(S\setminus S_0)$.}
\end{equation}
If $S_0=\emptyset$, \eqref{rev_ibp_preeps} implies \eqref{rotational_positive_multiplier} by choosing $\varphi\equiv 1$. If $S_0\neq 0$, we find a sequence $(\varphi_\eps)_\eps$ such that 
\begin{equation}\label{rev_duelimiti}
     \lim_{\eps\to 0}\int_S\varphi_\eps\Big( \alpha^2-LH^\hhh\Big)\,d\sigma^\hhh=\int_S\Big( \alpha^2-LH^\hhh\Big)\,d\sigma^\hhh,\qquad \lim_{\eps\to 0}\int_S\alpha J(\vh)\varphi_\eps\,d\sigma^\hhh=0,
\end{equation}
which yields \eqref{rotational_positive_multiplier} thanks to \eqref{rev_ibp_preeps}. Since $S_0$ is finite, we assume for notational simplicity that $S_0=\{p\}$. This clearly suffices up to repeating the following construction finitely many times. For a fixed $\eps>0$ sufficiently small, $\varphi_\eps$ can be chosen as follows. Assume as in Step 1 that $p=0$, $\n(0)=-T$ and $S$ admits the above-mentioned graph representation. In local coordinates $z=(x,y)$ as in Step 1, we let $\varphi_\eps(x,y)$ be a cut-off function between $\{|z|<\eps\}$ and $\{|z|<3\eps\}$, such that $|D\varphi_\eps|\leq\eps^{-1}$.
First, \eqref{rev_espressioneinloccords} and \eqref{rotational_pole_asymptotics} imply that $$\alpha^2-LH^\hhh\in L^1(S,\sigma^\hhh),$$ yielding the first limit in \eqref{rev_duelimiti} by dominated convergence. Moreover, 
\begin{equation*}
  \left| \int_S\alpha J(\vh)\varphi_\eps\,d\sigma^\hhh\right|\overset{\eqref{rev_espressioneinloccords}}{\leq}\frac{1}{\eps}|\{\eps<|z|<3\eps\}|=8\pi\eps,
\end{equation*}
 yielding the second limit in \eqref{rev_duelimiti}. Finally, we prove $S_0\neq\emptyset$. Assume by contradiction that $S_0=\emptyset$. Then $J(\vh)$ is defined and smooth on $S$. Since $S$ is compact, \cite[Theorem 9.16]{MR2954043} implies that $J(\vh)$ is a complete vector field. Accordingly, let $\gamma:\rr\to S$ be any integral curve of $J(\vh)$. Then, along $\gamma$, \eqref{rotational_el} and $L>0$ grant
 \begin{equation}\label{rev_diseqazione}
     \dot\alpha\leq-\alpha^2\leq 0.
 \end{equation}
 We claim that $\alpha\equiv 0$ along $\gamma$. Assume not by contradiction. Let $s_0\in\rr$ be such that $\alpha(s_0)\neq 0$. We may assume $\alpha(s_0)>0,$ the opposite case being similar. By \eqref{rev_diseqazione} $\alpha$ is decreasing, whence $\alpha(s)>0$ for any $s\leq s_0$. Therefore, dividing \eqref{rev_diseqazione} by $-\alpha^2$ and integrating on $(s,s_0)$,
 \begin{equation}\label{rev_tfcstep2}
     \frac{1}{\alpha(s_0)}-\frac{1}{\alpha(s)}=\int_s^{s_0}\left(-\frac{\dot\alpha(\varsigma)}{\alpha(\varsigma)^2}\right)\,d\varsigma\geq s_0-s.
 \end{equation}
Since $\alpha(s_0)\ge 0$, then $s\coloneqq s_0-\alpha(s_0)^{-1}\le s_0$. This choice in \eqref{rev_tfcstep2} yields $\alpha(s)\leq 0$, a contradiction. Since $S_0\neq\emptyset,$ we assume henceforth that $0\in S_0$, and that $S$ admits the graphical representation as in Step 1.

\smallskip
\noindent\textbf{Step 3.} We show that $S$ is rotationally invariant in a neighborhood of $0$. Assume first $c\ge 0$. As will be clearer, this choice will identify Pansu-Minkowski spheres for $0\le L\le\frac{1}{2}$.  We adopt local polar coordinates $(r,\theta)$ in $\rr^2$ centered at $0$. Recall that 
\begin{equation}\label{cambiovariabilipoloarorifdjriejcneficjnecv}
    \partial_x=\cos\theta\,\partial_r-\frac{\sin\theta}{r}\,\partial_\theta,\qquad \partial_y=\sin\theta\,\partial_r+\frac{\cos\theta}{r}\,\partial_\theta,\qquad J(\partial_r)=\frac{1}{r}\,\partial_\theta,\qquad J(\partial_\theta)=-r\,\partial_r.
\end{equation}
Therefore, setting
\begin{equation*} 
    v=\frac{ u_\theta}{r},\qquad P= u_r,\qquad Q=r+v,
\end{equation*}
a direct computation yields
\begin{equation}\label{espressionecampipolarifejefferev}
    F=P\,\partial_r+\frac{Q}{r}\,\partial_\theta,\qquad JF=-Q\,\partial_r+\frac{P}{r}\,\partial_\theta,\qquad W=\sqrt{P^2+Q^2}.
\end{equation}
Since $DF=D^2u+J$, direct differentiation shows that
\begin{equation*}
    \divv\left(\frac{JF}{W^2}\right)=-\frac{2\langle D^2u\,F,JF\rangle}{W^4},\qquad
    \divv\left(\frac{F}{W}\right)=\frac{\langle D^2u\,JF,JF\rangle}{W^3}.
\end{equation*}
Thus \eqref{rotational_graph_equation} is equivalent to
\begin{equation}\label{rotational_planar_divergence}
    \divv\left(\frac{JF}{W^2}-2L\frac{F}{W}\right)=0.
\end{equation}
Recalling that, for a vector field $A$,
\begin{equation*}
    \divv\left(A^r\,\partial_r+A^\theta\,\partial_\theta\right)=\frac{1}{r}\,\partial_r(rA^r)+\partial_\theta A^\theta,
\end{equation*}
we can rewrite \eqref{rotational_planar_divergence} as
\begin{equation}\label{rev_eq_in_div_polar1}
    \partial_r\left[r\left(-\frac Q{W^2}-2L\frac P W\right)\right]
    +\partial_\theta\left(\frac P{W^2}-2L\frac QW\right)=0.
\end{equation}
Set 
\begin{align*}
    \kappa=\frac{2c}{(1+c^2)^2},\qquad a&=\frac{r^2}{\kappa}\left(\frac{2PQ}{W^4}-\frac{2LQ^2}{W^3}\right),\qquad
    b=\frac{r^2}{\kappa}\left(\frac{Q^2-P^2}{W^4}+\frac{2LPQ}{W^3}\right),\\
    d&=\frac{r^2}{\kappa}\left(-\frac{2PQ}{W^4}-\frac{2LP^2}{W^3}\right),\qquad m=ra_r+b_\theta-a.
\end{align*}
As $\partial_\theta P=rv_r+v$ and $\partial_\theta Q=v_\theta$, differentiating \eqref{rev_eq_in_div_polar1} with respect to $\theta$ and multiplying it by $\kappa^{-1}r^2$ yields
\begin{equation*}
    r^2\partial_r\left[\frac{1}{r}\left(arv_r+bv_\theta\right)\right]+\partial_\theta\left(brv_r+dv_\theta\right)+r^2\partial_r\left(\frac{av}{r}\right)+\partial_\theta\left(bv\right)=0.
\end{equation*}
In particular, 
\begin{equation}\label{rotational_angular_equation}
    r\partial_r(arv_r+bv_\theta)+\partial_\theta(brv_r+dv_\theta)+mv=0.
\end{equation}
Equation \eqref{rotational_angular_equation} is valid at $r\ge 0$ sufficiently small,, where its coefficients and $v$ are smooth. Notice that, by \eqref{rotational_pole_asymptotics} and uniformly in $\theta$,
\begin{equation}\label{rev_asimptoticspQ}
    u_r=cr+o(r),\qquad u_\theta=o(r^2)\qquad\text{as $r\to 0$},
\end{equation}
so that 
\begin{equation*}\label{rev_asintpqw1}
    P=cr+o(r),\qquad Q=r+o(r),\qquad W=\sqrt{1+c^2}r +o(r)\qquad\text{as $r\to 0$}.
\end{equation*}
Therefore, uniformly in $\theta$,
\begin{equation}\label{rotational_coefficient_limits}
    a\to 1,\qquad d\to-1\qquad\text{as $r\to 0$}.
\end{equation}
Similarly, noticing that 
\begin{equation}\label{rev_stime_coefficienti_eqdiv_osmall}
    v_r=o(1),\quad v_\theta=o(r),\quad P_r=c+o(1),\quad P_\theta=o(r),\quad Q_r=1+o(1),\quad Q_\theta=o(r)\quad\text{as $r\to 0$},
\end{equation}
one sees that, uniformly in $\theta$,
\begin{equation}  \label{rotational_coefficient_limits2} ra_r,rb_r,rd_r,b_\theta\to 0,\qquad m\to-1\qquad\text{as $r\to 0$}.
\end{equation}
By \eqref{rotational_coefficient_limits} and \eqref{rotational_coefficient_limits2}, there exists $\tilde r>0$ such that
\begin{equation}\label{stimecoefficientiabcdederivate}
    |a-1|,|d+1|,|ra_r|,|rb_r|,|rd_r|,|m+1| \leq \frac{1}{10} \qquad\text{for any }0\le r\le\tilde r, 0\leq\theta\leq 2\pi.
\end{equation}
 Set
\begin{equation*}
    \mathcal E(r)=\frac12\int_0^{2\pi}\big(a(rv_r)^2-dv_\theta^2\big)\,d\theta.
\end{equation*}
Then \eqref{rev_stime_coefficienti_eqdiv_osmall} and \eqref{stimecoefficientiabcdederivate} yield
\begin{equation}\label{rev_energiaasintotica1}
    \frac9{20}\int_0^{2\pi}\big((rv_r)^2+v_\theta^2\big)\,d\theta
    \leq\mathcal E(r),
    \qquad\mathcal E(r)=o(r^2)\quad\text{as $r\to 0$}.
\end{equation}
Notice that 
\begin{align*}
    \int_0^{2\pi}r^2 v_r\partial_r(a r v_r)\,d\theta&=\frac{r}{2}\frac{d}{dr}\int_0^{2\pi}a(r v_r)^2\,d\theta+\frac{1}{2}\int_0^{2\pi}(ra_r)(r v_r)^2\,d\theta,\\
     \int_0^{2\pi}r^2 v_r\partial_r(b v_\theta)\,d\theta&= \int_0^{2\pi}r^2 v_rb_r v_\theta\,d\theta+\int_0^{2\pi}r^2 v_rb v_{r\theta}\,d\theta,\\   
       \int_0^{2\pi}r v_r\partial_\theta(b r v_r+d v_\theta)\,d\theta&=-\int_0^{2\pi}r^2 v_rb v_{r\theta}\,d\theta-\frac{r}{2}\frac{d}{dr}\int_0^{2\pi}dv_\theta^2\,d\theta+\frac{1}{2}\int_0^{2\pi}(rd_r)v_\theta^2\,d\theta.
\end{align*}
Therefore, multiplying \eqref{rotational_angular_equation} by $rv_r$ and integrating over $(0,2\pi)$ yields
\begin{equation*}
    r\mathcal E'(r)=-\frac12\int\Big((ra_r)(rv_r)^2+(rd_r)v_\theta^2
    +2(rb_r)(rv_r)v_\theta+2mv(rv_r)\Big)\,d\theta.
\end{equation*}
Moreover, the very definition of $v$ implies that it has zero average on $(0,2\pi)$. Therefore, by \eqref{stimecoefficientiabcdederivate} and Wirtinger inequality \cite[Theorem 4.24]{MR0990890}, 
\begin{equation*}
    r\mathcal E'(r)\leq\frac{13}{20}\int\big((rv_r)^2+v_\theta^2\big)\,d\theta
    \overset{\eqref{rev_energiaasintotica1}}{\leq}\frac32\mathcal E(r), 
\end{equation*}
so that 
\begin{equation*}
    \frac{d}{dr}\left(r^{-\frac{3}{2}}\mathcal E(r)\right)=r^{-\frac{5}{2}}\left(r \mathcal E'(r)-\frac{3}{2}\mathcal E(r)\right)\leq 0.
\end{equation*}
In particular, $r\mapsto r^{-\frac{3}{2}}\mathcal E(r)$ is non-increasing on $(0,\tilde r)$. Accordingly, for $0<r<r_1<\tilde r$, it follows that
\begin{equation*}
    0\leq\mathcal E(r_1)\leq\left(\frac{r_1}{r}\right)^{3/2}\mathcal E(r)\overset{\eqref{rev_energiaasintotica1}}{\to}0
    \qquad\text{as $r\to 0$}.
\end{equation*}
Thus $\mathcal E\equiv 0$ on $(0,\tilde r)$. Accordingly, $u_\theta\equiv 0$ in a neighborhood of $0$: $S$ is locally rotationally invariant near $0$. 

\smallskip
\noindent\textbf{Step 4.} Assume instead $c=0$. This choice will identify the horizontal plane and the Pansu sphere. Define 
\begin{equation}\label{defofvrev38475938575}
    \mathbf V=\s+L J(\vh)\qquad\text{on $S\setminus S_0$.}
\end{equation}
Then, recalling \eqref{commutaturejvhesse} and \eqref{duederivatedivartetajvhesse} and by \eqref{rotational_el}, 
\begin{equation*}
    J(\vh)\vartheta=H^\hhh,\qquad \mathbf V\vartheta=\alpha^2,\qquad
    [J(\vh),\mathbf V]=\alpha(2L-H^\hhh)J(\vh)-2\alpha\mathbf V.
\end{equation*}
In particular,
\begin{equation}\label{rotational_curvature_transport}
\begin{split}
      \mathbf V H^\hhh&=\mathbf V J(\vh) \vartheta\\
      &=J(\vh)\mathbf V\vartheta+[\mathbf V,J(\vh)]\vartheta\\
      &=2\alpha J(\vh)\alpha+\alpha(H^\hhh-2L)H^\hhh+2\alpha^3\\
      \overset{\eqref{rotational_el}}&{=}\alpha H^\hhh(H^\hhh-4L).
\end{split}
\end{equation}
For $R\ge 0$, set $B_R=\{|z|\le R\}$ and $B_R^*=\{0\le |z|\le R\}$. Fix $R$ small to ensure that $u\in C^2(B_R)\cap C^\infty(B_R^*)$, $|W|\leq 1$ and $0$ is the unique zero of $F$ on $B_R$. 
 Since $c=0$, \eqref{rotational_pole_asymptotics} and \eqref{rev_asimptoticspQ} yield, uniformly in $\theta$, 
\begin{equation}\label{coseadattateacugualeazero}
    P=o(r),\qquad Q=r+o(r),\qquad W=r+o(r),\qquad H^\hhh=o(r^{-1})\qquad\text{as $r\to 0$.}
\end{equation}
Recalling the identification of vector fields with their planar projection as in Step 3, set 
\begin{equation}\label{rev_definizionediA}
    \A\coloneqq -\frac{1}{\alpha}\mathbf V\overset{\eqref{rev_espressioneinloccords}}{=}\frac FW+LJF
    \overset{\eqref{espressionecampipolarifejefferev}}{=}\left(\frac PW-LQ\right)\partial_r
      +\frac1r\left(\frac QW+LP\right)\partial_\theta.
\end{equation}
Since
$\divv(JF)=-2$, \eqref{rotational_graph_equation} and \eqref{rotational_curvature_transport} give
\begin{equation}\label{rev_derivbatadihlungoa}
    \A H^\hhh=-H^\hhh(H^\hhh-4L),\qquad \divv\A=H^\hhh-2L.
\end{equation}
Set $\beta=\sqrt{|H^\hhh(H^\hhh-4L)|}$ and 
$\beta_\delta=(\beta^4+\delta^4)^{\frac{1}{4}}$ for $\delta>0$. Then $\beta_\delta$ is smooth, and $\beta_\delta\to \beta$ uniformly as $\delta\to 0$.
Moreover, 
\begin{equation*}
    \A\beta^4=2\A H^\hhh\left(H^\hhh(H^\hhh-4L)^2+(H^\hhh)2(H^\hhh-4L)\right)\overset{\eqref{rev_derivbatadihlungoa}}{=}4(2L-H^\hhh)\beta^4,
\end{equation*}
 whence
\begin{equation}\label{equationestimadivergenzaaadelta}
    \divv(\beta_\delta\A)
    =(H^\hhh-2L)\frac{\delta^4}{\beta_\delta^3},
    \qquad |\divv(\beta_\delta\A)|\leq\delta|H^\hhh-2L|.
\end{equation}
Let $D\Subset B_R^*$ be a piecewise smooth domain. Denote by $\nnn$ the outward unit Euclidean normal to $\partial D$.
The divergence theorem yields
\begin{equation*}
    \int_{\partial D}\beta_\delta\langle \A,\nnn\rangle\,d\mathcal H^1=\int_D\divv (\beta_\delta \A)\,dz.
\end{equation*}
Since $H^\hhh$ is bounded on $D$, \eqref{equationestimadivergenzaaadelta} allows to let $\delta\to 0$ in the previous equation, so that
\begin{equation}\label{rev_flussonulloingeneralesenzalimite}
    \int_{\partial D}\beta\langle \A,\nnn\rangle\,d\mathcal H^1=0.
\end{equation}
Moreover, since $L\ge 0$, $H^\hhh\geq 0$ and $|W|\leq 1$, then $|\A|\leq1+LW\leq 1+L$ and $\beta\leq H^\hhh+2L$. In particular, by \eqref{coseadattateacugualeazero}, $r\beta=o(1)$ as $r\to 0$.
Therefore we conclude that 
\begin{equation}\label{rev_flussosiannullanellorigine}
    \left|\int_{\partial B_\varrho}\beta\langle\A,\nnn\rangle\,d\mathcal H^1\right|\leq 2\pi(1+L)\varrho \sup_{\partial B_\varrho}\beta\to 0\qquad\text{as $\varrho\to 0$.}
\end{equation}
Fix $\eps>0$. By \eqref{coseadattateacugualeazero}, up to choosing a smaller $R$ we may assume that 
\begin{equation}\label{rev_stimaunmezzoepsilon}
    \frac QW+LP\geq\frac{1}{2},\qquad \left|\frac{\frac PW-LQ}{\frac QW+LP}\right|\leq\eps.
\end{equation}
 For $0<r_0<Re^{-2\pi\eps}$ and $\theta_0\in\rr$, consider the unique local solution to 
\begin{equation}\label{rev_oderrprimoprelogfvrgtvt}
    r'(\theta)=\frac{\frac PW-LQ}{\frac{1}{r(\theta)}\left(\frac QW+LP\right)},\qquad r(\theta_0)=r_0,
\end{equation}
where $P,Q,W$ are evaluated at $(r(\theta),\theta)$.
As long as the solution exists and $\theta_0\leq\theta\leq\theta_0+2\pi$, \eqref{rev_stimaunmezzoepsilon} and \eqref{rev_oderrprimoprelogfvrgtvt} grant
\[
    \left|\log\frac{r(\theta)}{r_0}\right|
    \leq\eps(\theta-\theta_0)\leq2\pi\eps,
    \qquad 0<e^{-2\pi\eps}r_0\leq r(\theta)\leq e^{2\pi\eps}r_0<R.
\]
Therefore, the solution $(r(\theta),\theta)$ is confined in a region where the
right-hand side of \eqref{rev_oderrprimoprelogfvrgtvt} is smooth. In particular, $r(\theta)$ extends to a smooth solution to \eqref{rev_oderrprimoprelogfvrgtvt} defined on $[\theta_0,\theta_0+2\pi]$.  Define $\gamma(\theta)=(r(\theta)\cos\theta,r(\theta)\sin\theta)$ and $k(\theta)=H^\hhh(\gamma(\theta))$ on $[\theta_0,\theta_0+2\pi]$.  Then
\begin{equation*}
   \gamma'(\theta)\overset{\eqref{cambiovariabilipoloarorifdjriejcneficjnecv}}{=}r'(\theta)\partial_r+\partial_\theta\overset{\eqref{rev_oderrprimoprelogfvrgtvt}}{=}\left(\frac{1}{r(\theta)}\left(\frac QW+LP\right)\right)^{-1}\A_{\gamma(\theta)}
\end{equation*}
and 
\begin{equation}\label{rev_scalareqfork}
    k'\overset{\eqref{rev_derivbatadihlungoa}}{=}-\left(\frac{1}{r(\theta)}\left(\frac QW+LP\right)\right)^{-1}k(k-4L).
\end{equation}
We claim that $\beta\equiv 0$ on $B_{Re^{-2\pi\eps}}^*$. Suppose not by contradiction. Let $(r_0,\theta_0)$ be as above, and such that $\beta(r_0,\theta_0)>0$. Let $\gamma$ and $k$ as above. Since $r$ is defined for every $\theta\in[\theta_0,\theta_0+2\pi]$, $\gamma$ makes a complete turn. We claim that this turn cannot close, i.e. $r_0\neq r_1\coloneqq r(\theta_0+2\pi)$. Since $\beta(r_0,\theta_0)\ge 0$, uniqueness for \eqref{rev_scalareqfork} and $k\geq 0$ yields either
$k(\theta)\in(0,4L)$ for every $\theta\in [\theta_0,\theta_0+2\pi]$ or $k(\theta)\ge 4L$ for every $\theta\in [\theta_0,\theta_0+2\pi]$. In turn, this and \eqref{rev_scalareqfork} imply that $k$ is strictly monotone on $[\theta_0,\theta_0+2\pi]$. Therefore, $k(\theta_0)\neq k(\theta_0+2\pi)$. But then $r_0\neq r_1$,
    since otherwise $\gamma(\theta_0)=\gamma(\theta_0+2\pi)$, and in this case $k(\theta_0)= k(\theta_0+2\pi)$. Hence $r_0\neq r_1$. Assume that $r_0\le r_1$, the other case being analogous. Denote by $I$ the straight radial segment joining $\gamma(\theta_0)$ and $\gamma(\theta_0+2\pi)$. Set $\Gamma=\gamma([\theta_0,\theta_0+2\pi])$. Then $\Gamma\cup I$ is a simple, closed, piecewise smooth curve. Denote by $\Om$ the bounded domain it encloses. For any $\varrho\le r_0 e^{-2\pi\eps}$, set $D_\varrho=\Om\setminus\overline B_\varrho$. We apply \eqref{rev_flussonulloingeneralesenzalimite} with $D=D_\varrho$. Noticing that $\langle A,\nnn\rangle =0$ on $\Gamma$, letting $\varrho\to 0$ and exploiting \eqref{rev_flussosiannullanellorigine} yields
    \begin{equation*}
        \int_I\beta\langle \A,\nnn\rangle\,d\mathcal H^1=0.
    \end{equation*}
Since I is radial, then $\nnn=r^{-1}\partial\theta$ on $I$. Therefore, by \eqref{rev_definizionediA} and since $\langle\partial_\theta,\partial_\theta\rangle=r^2$,
\begin{equation}\label{rev_integralfinestep4}
    \int_{r_0}^{r_1}
    \beta(r,\theta_0)\left(\frac QW+LP\right)(r,\theta_0)\,dr=0.
\end{equation}
By \eqref{rev_stimaunmezzoepsilon}, the integrand in \eqref{rev_integralfinestep4} is non-negative. Moreover, as we assumed $\beta(r_0,\theta_0\ge 0$, it is strictly positive on a neighborhood of $r_0$. Therefore the integral in \eqref{rev_integralfinestep4} is positive, a contradiction.  Therefore, $\beta\equiv0$ on $B_{Re^{-2\pi\eps}}^*$. Since $H^\hhh$ is continuous on $B_{Re^{-2\pi\eps}}^*$ and the latter is connected, the definition of $\beta$ yields either
 $H^\hhh\equiv0$ or $H^\hhh\equiv4L$ on $B_{Re^{-2\pi\eps}}^*$. In the first case, \cite[Theorem 4.15]{MR2435652} implies that $S$ is locally a horizontal disk near $0$. In the second case, the same result grants that $S$ is locally a Pansu sphere cap near $0$. In both cases, we conclude that $S$ is rotationally invariant near $0$.
 
\smallskip
\noindent\textbf{Step 5.} By the previous steps, $S$ is rotationally invariant near $0$. We prove that this property is global, whence the thesis follows by \Cref{rigiditymapergraicidmrvojfrn4fijerneitvjn}. Let $u$ be as above. Assume first that $S$ is not locally planar. Locally near $0$, $u=u(r)$. Set $w=\dot u$, and $$z=\frac{w}{\sqrt{w^2+r^2}}$$ 
as in \Cref{lemmacaratterizzaionegraficacp}. Arguing as in its proof, there exists $a\in\rr$ such that 
\begin{equation*}
        z(r)^2 -2Lr z(r)+a=0\qquad\text{locally near $0$.}
    \end{equation*}
Therefore, by \eqref{rev_asimptoticspQ} and noticing that $w,z\geq 0$ by Step 3 and Step 4,  
\begin{equation}\label{rotational_radial_first_integral}
    z(r)^2-2Lr\,z(r)-\frac{c^2}{1+c^2}=0,\qquad
    z(r)=Lr+\sqrt{L^2r^2+\frac{c^2}{1+c^2}}\qquad\text{locally near $0$.}
\end{equation}
 Up to replacing $S$ with $\delta_\frac{1}{\rho}(S)$, where 
\begin{equation*}
    \rho=\frac{1}{2L(1+c^2)},
\end{equation*}
and up to a vertical translation, a simple computation grants that $L\in \left(0,\frac{1}{2}\right]$ and that $S$ is, locally, a lower cap of the Pansu-Minkowski sphere $S_L$. We prove that $S=S_L$. Let $P(s,\theta)$  be a parametrization of $S_L$ as in \Cref{secdefproprotinv}. Its profile is chosen as in \eqref{pansuminkparamtheorem}. Set $s_-=-\arccos\sqrt{1-2L},$ $s_+=\arccos\sqrt{1-2L}$. Assume by contradiction that $S\neq S_L$. By the previous considerations, there exists $\bar s\in (s_-,s_+)$ such that 
\begin{equation*}
   P([\s_-,\bar s]\times [0,2\pi])=S\cap  P([s_-,\bar s]\times [0,2\pi]).
\end{equation*}
Let $\bar s$ be maximal with this property.
Denote by $\vh$, $\n$ and $\alpha$ the geometric quantities associated with $S_L$. Set 
\begin{equation*}
    \B=\frac{\vh+\alpha T}{1+\alpha^2}.
\end{equation*} 
Let $\delta\ge 0$ be such that $(\bar s-\delta,\bar s+\delta)\subseteq (s_-,s_+)$. Set $U=(\bar s-\delta,\bar s+\delta)\times [0,2\pi]$.    Extend $\B$, locally near $P(U)$, to a smooth vector field in $\hh^1$. 
By construction, $\B$ is proportional to $\n$. 
let $\Phi(\tau,p)$ be its (local) flow. For any $f: U\to \rr$, $2\pi$-periodic and sufficiently small, then  $P_f(s,\theta)=\Psi (f(s,\theta),P(s,\theta))$ is a well-defined graph over $S_L$ along $\B$, say $S_f$. Moreover, as $P(\{\bar s\}\times [0,2\pi])\subseteq S_L\setminus (S_L)_0$, for any such $f$ we can choose $\delta$ small so that $S_f$ is non-characteristic. For any such (smooth) function, denote by $\vh_f$, $\alpha_f$ and $H^\hhh_f$ the relevant geometric quantities of $S_f$, and set 
\begin{equation*}
    \mathscr E(f)\coloneq -4\Big(J(\vh_f)\alpha_f+\alpha_f^2+LH_f^\hhh\Big)\circ P_f.
\end{equation*}
Since $\alpha_f$ and $\vh_f$ depends only on $(s,\theta,f,f_s,f_\theta)$, it is easy to verify that the above equation is second-order and quasi-linear in $(s,\theta)$, namely
\begin{equation}\label{eq:prop59-graph}
    \mathscr E(f)
    =A(s,\theta,f,f_s,f_\theta) f_{ss}+2B(s,\theta,f,f_s,f_\theta) f_{s\theta}+C(s,\theta,f,f_s,f_\theta) f_{\theta\theta}+R(s,\theta,f,f_s,f_\theta).
\end{equation} 
Let $\varphi$ be a smooth function, compactly supported in $U$. With abuse of notation, we may consider it as a function defined over $P(U)$.
For $\eps\ge 0$ small, denote by 
$\nu_\varepsilon$, $\alpha_\varepsilon$ and $H_\varepsilon^{\mathcal H}$
the geometric quantities of $S_{\eps\varphi}$. let $m, m_\eps$ be such that 
\[
  d\sigma^\hhh=m\,ds\,d\theta,\qquad d\sigma^\hhh_{S_{ \eps \varphi }}=m_\eps\,ds\,d\theta.
\]
Define 
\[
\begin{aligned}
    F(\varepsilon)=\int_U
        \big(H_\varepsilon^{\mathcal H}\circ P_\varepsilon-4L\big)
             m_\varepsilon\,ds\,d\theta
\end{aligned}
\]
By \eqref{variazioneprimaareasubriemmain} and \eqref{primavariaztmcdtatementmain},
\begin{equation}\label{rev_aux38596058385}
    F'(\varepsilon)
      =\int_U\varphi k_\varepsilon\,
           \mathscr E(\varepsilon\varphi)m_\varepsilon\,ds\,d\theta,\qquad\text{where }k_\varepsilon
       :=\langle\mathbf B,\nu_\varepsilon+
                      \alpha_\varepsilon T\rangle\circ P_\varepsilon.
\end{equation}
Define 
\[
    D\mathscr E(0)[\varphi]
       :=\left.\frac{d}{d\varepsilon}\right|_{\varepsilon=0}
                 \mathscr E(\varepsilon\varphi).
\]
Differentiating \eqref{rev_aux38596058385} at $\eps=0$, and recalling that $\mathscr E(0)=0$ and that $k_0=1$, 
\[
\begin{aligned}
    F''(0)=\int_U\varphi\Big[
           \dot k_0\mathscr E(0)m
           +k_0D\mathscr E(0)[\varphi]m
           +k_0\mathscr E(0)\dot m_0
         \Big]ds\,d\theta=\int_U \varphi D\mathscr E(0)[\varphi]\,m\,ds\,d\theta.
\end{aligned}
\]
On the other hand, let $\Phi_\varphi$ be the flow of $\varphi \mathbf B$. Recalling \eqref{variazsecondaareastatementsubriem} and \eqref{variazsecondatmcsubriemmainstatement}, and by definition of $\mathbf B$, let $\mathcal L^{\mathbb H}$ be the second-order operator such that 
\begin{equation*}
    \delta^2\tmc^\hhh(S_L)[\Phi_\varphi]-\lagrange\delta^2\area^\hhh(S_L)[\Phi_\varphi]=\int_S\varphi\Big(\mathcal L^\hhh\varphi\circ P\Big)\,m\,ds\,d\theta.
\end{equation*}
Combining the above two equations, we deduce that 
\begin{equation}\label{rev348458696959384}
    \int_U \varphi D\mathscr E(0)[\varphi]\,m\,ds\,d\theta=\int_S\varphi\Big(\mathcal L^\hhh\varphi\circ P\Big)\,m\,ds\,d\theta.
\end{equation}
By \eqref{rev348458696959384}, $D\mathscr E(0)[\varphi]$ and $L^\hhh\varphi\circ P$ share the same second-order part. 
Let $A_0,B_0,C_0$ denote the coefficients $A,B,C$ in \eqref{eq:prop59-graph} evaluated at
$(s,\theta,0,0,0)$. Thus the second order part of $ D\mathscr E(0)[\varphi]$ is 
\[
A_0\varphi_{ss}+2B_0\varphi_{s\theta}+C_0\varphi_{\theta\theta}.
\]
We claim that 
\begin{equation}\label{iperbolicitàpereqlinap}
    A_0\le 0,\qquad C_0\ge 0.
\end{equation}
 By the above consideration, it suffices to consider $L^\hhh\varphi\circ P$.
By \eqref{variazsecondaareastatementsubriem}, \eqref{variazsecondatmcsubriemmainstatement} and  \cite[Proposition 3.10]{formuledivariazioneriem}, the second order part of $L^\hhh\varphi$ is $4\mathbf V\left(J(\vh)\varphi\right)$, where $\mathbf V$ is as in \eqref{defofvrev38475938575}. Recalling \eqref{rev_vfforpm38385858585},
\begin{equation*}     
J(\vh)=-\left(\frac{\cos s}{h}\right)\frac{\partial}{\partial s}+\frac{\cos s}{x}\frac{\partial}{\partial \theta},\qquad \mathbf V=
        \left(\frac{\cos^2 s}{xh}-\frac{L\cos s}{h}\right)\frac{\partial}{\partial s}+\left(\frac{\sin^2 s}{x^2}+\frac{L\cos s}{x}\right)\frac{\partial}{\partial\theta}.
    \end{equation*}
where $x=x_L(s)$ and $h=h(s)$ are as in \eqref{curpamsejnefcjnecvjn} and \eqref{cosefinalidefofhfoga}.
Therefore,
 \begin{align*}
 A_0
  &=\frac{4\cos^2s}{h^2}\left(L-\frac{\cos s}{x}\right)=\frac{4L\cos^2s}{h^2}\left(1-\frac{2\cos^2 s}{\cos^2 s-(1-2L)}\right) \le 0,\\
  C_0
 &=4\left(\frac{\sin^2s\cos s}{x^3}
                +\frac{L\cos^2s}{x^2}\right)\ge 0.
 \end{align*}
Let $f:U\to \rr$ be the function representing $S$ over $S_L$ via $P_f$. Notice that, by definition of $\bar s$,
 \begin{equation}\label{fezerosuspacelike}
     f(\bar s,\theta)=f_s(\bar s,\theta)=f_\theta(\bar s,\theta)=0.
 \end{equation}
  Moreover, since both $S$ and $S_L$ are critical points, \eqref{eq:prop59-graph} yields \begin{equation}\label{twoelrev38487576}
 \begin{split}
 A(s,\theta,f,f_s,f_\theta) f_{ss}+2B(s,\theta,f,f_s,f_\theta) f_{s\theta}+C(s,\theta,f,f_s,f_\theta) f_{\theta\theta}+R(s,\theta,f,f_s,f_\theta)&=0,\\
 R(s,\theta,0,0,0)&=0.
  \end{split}
 \end{equation} 
 Denote by $a_1,a_2,a_3$ the derivatives of $R$ with respect to its third, fourth and fifth entry. Then, by \eqref{twoelrev38487576},
 \begin{equation*}
     \begin{split}
         R(s,&\theta, f,f_s,f_\theta)=R(s,\theta, f,f_s,f_\theta)- R(s,\theta, 0,0,0)\\
         &=\int_0^1\frac{d}{d\tau}R(s,\theta,\tau f,\tau f_s,\tau f_\theta)\,d\tau\\
         &= f\int_0^1 a_1(s,\theta,\tau f,\tau f_s,\tau f_\theta)\,d\tau+f_s\int_0^1 a_2(s,\theta,\tau f,\tau f_s,\tau f_\theta)\,d\tau+f_\theta\int_0^1 a_3(s,\theta,\tau f,\tau f_s,\tau f_\theta)\,d\tau.
     \end{split}
 \end{equation*}
 Therefore, setting 
 \begin{equation*}
     A_1(s,\theta)=A(s,\theta,f,f_s,f_\theta),\quad B_1(s,\theta)=B(s,\theta,f,f_s,f_\theta), \quad C_1(s,\theta)=C(s,\theta,f,f_s,f_\theta)
 \end{equation*}
 and 
 \begin{equation}
     D_j(s,\theta)=\int_0^1 a_j(s,\theta,\tau f,\tau f_s,\tau f_\theta)\,d\tau,\qquad j=1,2,3,
 \end{equation}
 then $f$ solves the \emph{linear}, second-order problem
 \begin{equation}\label{iperbolicsistemrev383859695}
 \begin{split}
      A_1 f_{ss}+2B_1 f_{s\theta}+C_1 f_{\theta\theta}+D_1f+D_2f_s+D_3 f_\theta&=0\qquad \text{on }[\bar s,\bar s+\eps)\times [0,2\pi],\\
      f(\bar s,\cdot )&=0\qquad \text{on } [0,2\pi],\\
      f_s(\bar s,\cdot )&=0\qquad \text{on } [0,2\pi].
 \end{split}
 \end{equation}
By definition, and recalling \eqref{fezerosuspacelike},
$A_1(\bar s,\theta)=A_0(\bar s,\theta)$, $B_1(\bar s,\theta)=B_0(\bar s,\theta)$ and  $C_1(\bar s,\theta)=C_0(\bar s,\theta)$. Therefore, by \eqref{iperbolicitàpereqlinap} and continuity, we may choose $\delta$ small so that
\[
B_1^2-A_1C_1>0,\qquad A_1\neq 0.
\]
In particular, the first equation in \eqref{iperbolicsistemrev383859695} is hyperbolic, and the level sets $\{s=s_0\}$ are \emph{spacelike} in the sense of \cite{MR2744150}. Therefore, \cite[Chapter 2, Proposition 8.1]{MR2744150} grants that $f\equiv 0$ on $U$. This extends the coincidence of $S_L$ and $S$ beyond $\bar s$, violating the maximality of $\bar s$. A contradiction follows. To conclude the proof it suffices to deal with the case in which $S$ coincides locally with the horizontal plane. The proof follows \emph{verbatim} the previous approach. To this aim, just notice that, by \Cref{jvwithresptolocalchart}, 
\[
 J(\vh)=-\frac{\partial}{\partial s},\qquad \mathbf V=-L\frac{\partial}{\partial s}+\frac{1}{s^2}\frac{\partial}{\partial_\theta},  
\]
whence
\[
 A_0=4L>0,\qquad B_0=-\frac2{s^2},\qquad C_0=0,
 \qquad B_0^2-A_0C_0=\frac4{s^4}>0.
\]
Therefore, $S$ would be the horizontal plane, contradicting its compactness. The thesis follows.
\end{proof}
\begin{proof}[Proof of \Cref{intro_teo2}]
   It follows by \Cref{rigiditymapergraicidmrvojfrn4fijerneitvjn}, \Cref{lemmaincuisievinceellunquarto} and \Cref{rotational_symmetry_critical_points}.
\end{proof}
\section{Total mean curvature: stability of Pansu-Minkowski spheres}\label{sec_stab}
Next, we discuss the stability of Pansu-Minkowski spheres. By \Cref{carstabprop}, we focus on the penalized functional $\penafun^\hhh_L$ as introduced in \eqref{intro_penalfun}.
%\begin{equation*}
  %  \penafun_L^\hhh(S)\coloneqq\tmc^\hhh(S)-4L\area^\hhh(S),\qquad L\in\left(0,\frac{1}{2}\right]
%\end{equation*}
Fix a non-characteristic variation $\Phi$. Let $\varphi$ and $\psi$ be as in \eqref{lefunzioniimportanti}. %With some abuse of notation, we may write $\varphi=\varphi(s,\theta)$ meaning the composition of $\varphi $ with the parametrization of $S_L$ provided by \eqref{pansuminkparamtheorem}.
First, we evaluate the second variation of $\penafun^\hhh_L$ at the Pansu-Minkowski sphere $S_L$. By \eqref{variazsecondaareastatementsubriem} and \eqref{variazsecondatmcsubriemmainstatement},
\begin{equation*}
    \begin{split}
        \delta^2\penafun^\hhh_L&(S_L)[\Phi]=\delta\tmc^\hhh(S_L)\left[\psi\right]-4L\delta\area^\hhh(S_L)\left[\psi\right]\\
        &\quad+4\int_{S_L}\Big(-\s\varphi\,J(\vh)\varphi+\varphi^2\left(2\s\alpha-H^\hhh\alpha^2\right)\Big)\,d\sigma^\hhh-4L\int_{S_L}\Big(\left(J(\vh)\varphi\right)^2-4\varphi^2\left(J(\vh)\alpha+\alpha^2\right)\Big)\,d\sigma^\hhh\\
        \overset{\eqref{variazioneprimazerotmcstatement},\eqref{derivationeulerlagrangeminkgeneral}}&{=}4\int_{S_L}\Big(\big(-\s\varphi\,J(\vh)\varphi-L\left(J(\vh)\varphi\right)^2\big)+\big(2\s\alpha-H^\hhh\left(\alpha^2+4L^2\right)\big)\varphi^2\Big)\,d\sigma^\hhh.
    \end{split}
\end{equation*}
We make the zero-order term explicit, and we provide a sharp lower bound. 
\begin{lemma}\label{zeroorderstabilitilemma}
    It holds that 
    \begin{equation*}
        2\s\alpha-H^\hhh\left(\alpha^2+4L^2\right)=\frac{2(1-2L)}{x(s)^3h(s)},
    \end{equation*}
   where $h$ is as in \eqref{cosefinalidefofhfoga}. In particular, 
    \begin{equation*}
        2\s\alpha-H^\hhh\left(\alpha^2+4L^2\right)\geq  \left.\Big[2\s\alpha-H^\hhh\left(\alpha^2+4L^2\right)\Big]\right|_{s=0}=\frac{2L(1-2L)}{1-L}.
    \end{equation*}
\end{lemma}
\begin{proof}
    Notice that 
    \begin{align*}
        \alpha\overset{\eqref{alfasurotinvggggg}}{=}-\frac{\dot x}{\sqrt{ \dot t^2+x^2\Dot x^2 }}=\frac{\sin s}{x},\qquad
        \s\alpha\overset{\eqref{sinloccords}}{=}\frac{\dot t\dot\alpha}{ \dot t^2+x^2\Dot x^2 }=\frac{\cos s}{x^2 h}\left(\cos^2 s+\frac{h}{x}\sin^2 s\right).
    \end{align*}
    Therefore, recalling \eqref{hmeanperpansumink}, 
    \begin{equation*}
    \begin{split}
           2\s\alpha-H^\hhh\left(\alpha^2+4L^2\right)&=\frac{2\cos s}{x^2 h}\left(\cos^2 s+\frac{h}{x}\sin^2 s\right)-\frac{2\cos^2 s}{2Lxh}\left(\frac{\sin^2s}{x^2}+4L^2\right)\\
           &=\frac{2\cos s}{x^3 h}\left(x\cos^2 s+h\sin^2 s-\frac{\sin^2 s\cos s}{2L}-2L x^2\cos s\right)\\
           &=\frac{2\cos s}{x^3 h}\left((x-h)\cos^2 s+h-\frac{\sin^2 s\cos s}{2L}-\frac{\cos^3}{2L}+\frac{2(1-2L)\cos s}{2L}-\frac{(1-2L)^2}{2L\cos s}\right)\\
           &=\frac{\cos s}{Lx^3 h}\left(\cos s+\frac{1-2L}{\cos s}-\sin^2 s\cos s-\cos^3-\frac{(1-2L)^2}{\cos s}\right)\\
           &=\frac{2(1-2L)}{x^3 h}.
    \end{split}
    \end{equation*}
    To conclude, recall that $x$ and $h$ achieve their maximum at $s=0$, and moreover $x(0)=1$ and  $h(0)=\frac{1-L}{L}$.
\end{proof}
Next, we deal with the first-order term.
\begin{lemma}\label{firstorderstabilitylemma}
    It holds that 
    \begin{equation}\label{primeinstabgenericocondedetheta}
    \begin{split}
           -\s\varphi\,J(\vh)\varphi&-L\left(J(\vh)\varphi\right)^2=\left(\frac{\cos^3 s}{x h^2}-\frac{L\cos^2 s}{ h^2}\right)\left(\frac{\partial\varphi}{\partial s}\right)^2\\
           &\quad+\left(\frac{2L\cos^2 s}{ x h}-\frac{\cos s\cos 2s}{x^2 h}\right)\frac{\partial \varphi}{\partial s}\frac{\partial \varphi}{\partial \theta}-\left(\frac{\sin^2 s\cos s}{x^3}+\frac{L\cos^2 s}{x^2}\right)\left(\frac{\partial\varphi}{\partial\theta}\right)^2.
    \end{split}    
    \end{equation}
    If in addition $\varphi$ is independent of $\theta$, then the following sharp lower bound holds: 
    \begin{equation*}
         -\s\varphi\,J(\vh)\varphi-L\left(J(\vh)\varphi\right)^2\geq(1-L)\left(J(\vh)\varphi\right)^2.
    \end{equation*}
\end{lemma}
\begin{proof}
    Exploiting \eqref{rev_vfforpm38385858585},
    \begin{equation*}
        \left(J(\vh)\varphi\right)^2=\left(\frac{\cos^2 s}{ h^2}\right)\left(\frac{\partial\varphi}{\partial s}\right)^2-\left(\frac{2\cos^2 s}{x h}\right)\frac{\partial \varphi}{\partial s}\frac{\partial \varphi}{\partial \theta}+\left(\frac{\cos^2 s}{x^2}\right)\left(\frac{\partial\varphi}{\partial\theta}\right)^2
    \end{equation*}
    and
    \begin{equation*}
        -\s\varphi J(\vh)\varphi=\left(\frac{\cos^3 s}{x h^2}\right)\left(\frac{\partial\varphi}{\partial s}\right)^2-\left(\frac{\cos s\cos 2s}{x^2 h}\right)\frac{\partial \varphi}{\partial s}\frac{\partial \varphi}{\partial \theta}-\left(\frac{\sin^2 s\cos s}{x^3}\right)\left(\frac{\partial\varphi}{\partial\theta}\right)^2,
    \end{equation*}
    whence \eqref{primeinstabgenericocondedetheta} follows. Next, assume that $\varphi$ is independent of $\theta$. Let $\mu\geq 0$. Then 
    \begin{equation*}
    \begin{split}
         -\s\varphi\,J(\vh)\varphi-\left(L+\mu\right)\left(J(\vh)\varphi\right)^2\overset{\eqref{primeinstabgenericocondedetheta}}&{=}\left(\frac{\cos^3 s}{x h^2}-\frac{(L+\mu)\cos^2 s}{ h^2}\right)\left(\frac{\partial\varphi}{\partial s}\right)^2\\
         &=\left(\frac{\cos s}{2Lx h^2}\right)\left((L-\mu)\cos^2 s+(L+\mu)(1-2L)\right).
    \end{split}
    \end{equation*}
    Assume first that $L-\mu\geq 0$. Since $\cos^2 s>(1-2L)$ on $(-\arccos\sqrt{1-2L},\arccos\sqrt{1-2L})$, then
    \begin{equation*}
       (L-\mu)\cos^2 s+(L+\mu)(1-2L)\geq 2L(1-2L)\geq 0.
    \end{equation*}
    Assume instead $L-\mu<0$. Since $\cos^2 s\leq \cos^2(0)=1$ on $(-\arccos\sqrt{1-2L},\arccos\sqrt{1-2L})$, then
    \begin{equation*}
        (L-\mu)\cos^2 s+(L+\mu)(1-2L)\geq  (L-\mu)+(L+\mu)(1-2L)=2L(1-L)-2L\mu,
    \end{equation*}
    whence $ -\s\varphi\,J(\vh)\varphi-\left(L+\mu\right)\left(J(\vh)\varphi\right)^2\geq 0$ provided that $\mu\leq 1-L$.
\end{proof}
Motivated by \Cref{firstorderstabilitylemma}, we introduce a relevant class of variations. Precisely, we say that a non-characteristic variation of $S_L$ is \emph{rotationally invariant} if $\varphi$ is independent of $\theta$. 
\begin{remark}
    We stress that requiring a variation to be rotationally invariant is much weaker than restricting to the class of rotationally invariant competitors, as no restriction on $\psi$ is imposed.
\end{remark}
Combining \Cref{zeroorderstabilitilemma}, \Cref{firstorderstabilitylemma} and \Cref{carstabprop}, 
we deduce that Pansu-Minkowski spheres are (much more than) stable for $\tmc^\hhh$ along area-preserving, rotationally invariant, non-characteristic variations.
\begin{proposition}\label{muchmorethanstabproprotinv}
    Let $L\in\left(0,\frac{1}{2}\right].$ Let $\Phi$ be a non-characteristic rotationally invariant variation. Then
    \begin{equation}\label{stabilitformulateoremaequazidrvorjvntrgvojtvnrtvojnnonintro}
  \delta^2\penafun^\hhh_L(S_L)[\Phi]\geq 4(1-L)\int_{S_L}\left(J(\vh)\varphi\right)^2\,d\sigma^\hhh+\frac{8L(1-2L)}{1-L}\int_{S_L}\varphi^2\,d\sigma^\hhh.
    \end{equation}
    In particular, $S_L$ is stable for $\tmc^\hhh$ along area-preserving non-characteristic  rotationally invariant variations.
\end{proposition}
\begin{remark}
    When $L=\frac{1}{2}$, i.e. when $S_L$ is the standard Pansu sphere, the coefficient multiplying the zero-order term in the right hand side of \eqref{stabilitformulateoremaequazidrvorjvntrgvojtvnrtvojnnonintro} vanishes. Nevertheless, by the Poincaré inequality, it is still possible to provide a lower bound of the form 
    \begin{equation}\label{verolbdoveusipoincare}
        \delta^2\penafun^\hhh_L(S_L)[\Phi]\geq c_{\mathrm{I}}(\delta,L)\left(\|\varphi\|^2_{L^2(I(\delta))}+\|\dot\varphi\|^2_{L^2(I(\delta))}\right),
    \end{equation}
    where $\supp\varphi\subseteq I(\delta)\coloneqq\left(-\arccos\sqrt{1-2L}+\delta,\arccos\sqrt{1-2L}-\delta\right)$ and $c_{I}(\delta,L)>0$, and $c_{I}(\delta,L)$ tends to $0$ as $\delta\to 0^+$. 
\end{remark}
We conclude this section with the proof of \Cref{teononstabintro}: the rotational invariance constraint cannot be removed, as Pansu-Minkowski spheres are unstable under more general variations. 
%\color{purple}
%\begin{theorem}
 %    Let $L\in\left(0,\frac{1}{2}\right].$ There is a non-characteristic first-order area-preserving variation $\Phi$ such that 
 %   \begin{equation*}
%  \delta^2\penafun^\hhh_L(S_L)[\Phi]<0,
%    \end{equation*}
 %  i.e. Pansu-Minkowski spheres are not area-preserving stable along arbitrary non-characteristic variations.
%\end{theorem}
%\color{black}
\begin{proof}[Proof of \Cref{teononstabintro}]
    Fix $L\in\left(0,\frac{1}{2}\right].$  Let $0<\delta<\arccos\sqrt{1-2L}$. Let $M\in\mathbb{N}_+$. Let $\psi\in C^\infty_c(-\delta,\delta)$ be not identically vanishing. Define
    \begin{equation*}
        \varphi(s,\theta)=\psi(s)\sin M\theta, \qquad s\in\left(-\arccos\sqrt{1-2L},\arccos\sqrt{1-2L}\right),\,\theta\in(0,2\pi). 
    \end{equation*}
    Let $\X$ be such that $\X|_S=\varphi\vh$. Set $\Z=0$. Let $\Phi$ be a variation as in \eqref{costruirevariazioniinhn}. Then $\Phi$ is a non-characteristic horizontally normal variation with velocity $\X$ and with no acceleration. Moreover, 
    \begin{equation*}
        \delta\area^\hhh(S_L)[\Phi]\overset{\eqref{variazioneprimaareasubriemmain}}{=}\int_{S_L}\varphi H^\hhh\,d\sigma^\hhh\overset{\eqref{hdesigmah}}{=}\left(\int_0^{2\pi}\sin M\theta\, d\theta\right)\left(\int_{-\delta}^\delta\psi x(x+h)\,ds\right)=0,
    \end{equation*}
    whence $\Phi$ is first-order area-preserving.
    Notice that 
    \begin{equation*}
        \frac{\partial\varphi}{\partial s}=\dot \psi\sin M\theta,\qquad \frac{\partial\varphi}{\partial\theta}=M\psi\cos M\theta. 
    \end{equation*} 
    Since
    \begin{equation*}
        \int_0^{2\pi}\sin M\theta\cos M\theta\,d\theta=0,\qquad \int_0^{2\pi}\sin^2\theta\,d\theta=\int_0^{2\pi}\cos^2\theta\,d\theta=\pi,
    \end{equation*}
  %  Let $c_1>0$ be such that 
 %   \begin{equation*}
  %    \left|\frac{\cos^3 s}{x h^2}-\frac{L\cos^2 s}{ h^2}\right|,\,\left|\frac{2L\cos^2 s}{ x h}-\frac{\cos s\cos 2s}{x^2 h}\right|,\,\left|\frac{2(1-2L)}{x^3h}\right|<c_1\qquad\text{on $(-\delta,\delta)$}.
  %  \end{equation*}
  %  Moreover, it is straightforward to show that there exist $c_2,c_3,c_4>0$ such that
  %  \begin{equation*}
   %    c_2<\frac{\sin^2 s\cos s}{x^3}+\frac{L\cos^2 s}{x^2},\qquad c_3<\frac{x^2 h}{\cos s}<c_4\qquad\text{on $(-\delta,\delta)$}.
  %  \end{equation*}
  %  Finally, let $c_5>0$ be such that 
  %   \begin{equation*}
  %    |\psi|,\,|\dot\psi|<c_5\qquad\text{on $(-\delta,\delta)$}.
  %  \end{equation*}
     then \eqref{desigmahforpmsphereserijenrviervjnc} , \Cref{zeroorderstabilitilemma} and \Cref{firstorderstabilitylemma} imply that $\delta^2\penafun_L^\hhh(S_L)[\Phi]$ equals
    \begin{equation*}     
4\pi
\int_{-\delta}^{\delta}
\Bigg[
\left(
\frac{x\cos^2 s}{h}
-
\frac{Lx^2\cos s}{h}
\right)\dot\psi^2
+
\left(
\frac{2(1-2L)}{x\cos s}
\right)\psi^2\Bigg]\,ds-
4\pi M^2\int_{-\delta}^{\delta}\left(
\frac{h\sin^2 s}{x}
+
Lh\cos s
\right)
\psi^2
\,ds.
    \end{equation*}
    Since
    \begin{equation*}
        \int_{-\delta}^{\delta}\left(
\frac{h\sin^2 s}{x}
+
Lh\cos s
\right)
\psi^2
\,ds>0, 
    \end{equation*}
    then $\delta^2\penafun_L^\hhh(S_L)[\Phi]<0$ provided that $M$ is sufficiently large. The thesis follows by \Cref{carstabprop}.
\end{proof}
\section{Total mean curvature: local minimality of Pansu-Minkowski spheres}\label{sec_minloc}
In this section we prove that Pansu-Minkowski spheres are local minimizers, in the sense of \eqref{optiprobuno}, in the class of rotationally invariant surfaces. To this aim, fix $L\in\left(0,\frac{1}{2}\right]$, and denote the profile of $S_L$ simply by $(x,t)$. Fix $\delta>0$, recall that  $I(\delta)= (-\arccos\sqrt{1-2L}+\delta,\arccos\sqrt{1-2L}-\delta).$ Fix $\varphi\in C^\infty_c(I(\delta))$, and set $S_L^\varphi\coloneqq S\cdot \varphi\vh$. By \eqref{eq:Hproduct} and  \eqref{normalerotinv}, $S^\varphi_L$ can be parametrized by $\left(\xi^\varphi(s,\theta),\eta^\varphi(s,\theta), t^\varphi(s,\theta)\right)$, where
\begin{align*}
    \xi^\varphi&\coloneqq\left(x+\frac{\varphi\dot t}{\sqrt{\dot t^2+x^2\dot x^2}}\right)\cos\theta-\left(\frac{\varphi x\dot x}{\sqrt{\dot t^2+x^2\dot x^2}}\right)\sin\theta,\\
    \eta^\varphi&\coloneqq \left(x+\frac{\varphi\dot t}{\sqrt{\dot t^2+x^2\dot x^2}}\right)\sin\theta+\left(\frac{\varphi x\dot x}{\sqrt{\dot t^2+x^2\dot x^2}}\right)\cos\theta,\\
    t^\varphi&\coloneqq t-\frac{\varphi x^2\dot x}{\sqrt{\dot t^2+x^2\dot x^2}}.
\end{align*}
Since
\begin{equation*}
    \sqrt{\left(\xi^\varphi(s,\theta)\right)^2+\left(\eta^\varphi(s,\theta)\right)^2}=\sqrt{\left(x+\frac{\varphi\dot t}{\sqrt{\dot t^2+x^2\dot x^2}}\right)^2+\left(\frac{\varphi x\dot x}{\sqrt{\dot t^2+x^2\dot x^2}}\right)^2}\eqqcolon x^\varphi(s),
\end{equation*}
then $S_L^\varphi$ is a rotationally invariant surface, with profile $(x^\varphi,t^\varphi)$.
\begin{proposition}\label{localminimalitipropositionen}
    Fix $L\in\left(0,\frac{1}{2}\right]$ and $\delta>0$. There exists $\varepsilon=\varepsilon(\delta,L)>0$ such that, if $\varphi\in C^\infty_c(I(\delta))$ satisfies 
    \begin{equation*}
        \area^\hhh\left(S_L^\varphi\right)=\area^\hhh\left(S_L\right),\qquad\|\varphi\|_{C^2(I(\delta))}\leq\varepsilon,
    \end{equation*}
    then 
    \begin{equation*}
        \tmc^\hhh\left(S_L\right)\leq\tmc^\hhh\left(S_L^\varphi\right).
    \end{equation*}
\end{proposition}
\begin{proof}
For $p\in S_L$ and $\tau\in[-1,1]$, set 
\begin{equation*}
    \Phi(\tau,p(s,\theta))=\left(\xi^{\tau\varphi}(s,\theta),\eta^{\tau\varphi}(s,\theta),t^{\tau\varphi}(s,\theta)\right),
\end{equation*}
and extend it smoothly into a variation of $\hh^1$. For any $\tau\in[0,1]$, $\Phi(\tau,S)$ is the rotationally invariant surface with profile $(x^{\tau\varphi},t^{\tau\varphi})$. Moreover, denoting by $\X$ the normal velocity of $\Phi$, then $\X|_S=\varphi\vh$. Define $f:\rr^5\to\rr$ by
\begin{equation*}
    f(p_1,q_1,q_2,r_1,r_2)\coloneqq\frac{p_1^3\left(q_1r_2-q_2r_1\right)+q_2^3}{q_2^2+p_1^2q_1^2}-4Lp_1\sqrt{q_2^2+p_1^2q_1^2}.
\end{equation*}
Moreover, define $Q:[0,1]\to\rr^5$ by 
\begin{equation*}
    Q(\tau)\coloneqq\left(x^{\tau\varphi},\dot x^{\tau\varphi},\dot t^{\tau\varphi},\ddot x^{\tau\varphi},\ddot t^{\tau\varphi}\right).
\end{equation*}
With these definitions,
\begin{equation*}
    \penafun_L^\hhh\left(\Phi(\tau,S)\right)=2\pi\int_{I(\delta)}f\left(Q(\tau)\right)\,ds.
\end{equation*}
By Taylor's formula with Lagrange remainder, and since $S_L$ is an area-preserving critical point along non-characteristic variations, there exists $\tilde\tau\in(0,1)$ such that 
\begin{equation*}
    \begin{split}
        \penafun_L^\hhh\left(S_L^\varphi\right)&=\penafun_L^\hhh\left(S_L\right)+\left.\frac{d}{d\tau}\right|_{\tau=0}\penafun_L^\hhh\left(\Phi(\tau,S)\right)+\frac{1}{2}\left.\frac{d^2}{d\tau}\right|_{\tau=\tilde \tau}\penafun_L^\hhh\left(\Phi(\tau,S)\right)\\
        &=\penafun_L^\hhh\left(S_L\right)+\delta\penafun_L^\hhh(S)[\Phi]+\frac{1}{2}\delta^2\penafun_L^\hhh(S)[\Phi]+\pi\int_{I(\delta)}\left(\left.\frac{\partial^2}{\partial\tau}\right|_{\tau=\tilde \tau}f(Q(\tau))-\left.\frac{d^2}{d\tau}\right|_{\tau=0}f(Q(\tau))\right)\,ds\\
        \overset{\eqref{variazioneprimazerotmcstatement}}&{=}\penafun_L^\hhh\left(S_L\right)+\underbrace{\frac{1}{2}\delta^2\penafun_L^\hhh(S)[\Phi]}_{\mathrm{I}}+\underbrace{\pi\int_{I(\delta)}\left(\left.\frac{\partial^2}{\partial\tau}\right|_{\tau=\tilde \tau}f(Q(\tau))-\left.\frac{d^2}{d\tau}\right|_{\tau=0}f(Q(\tau))\right)\,ds}_{\mathrm{II}}.\\
    \end{split}
\end{equation*}
Since $f$ is affine in the variables $r_1$ and $r_2$, a simple computation yields that 
\begin{equation*}
    \frac{\partial^2}{\partial\tau^2}f(Q(\tau))=A^{\tau\varphi}\varphi^2+B^{\tau\varphi}\varphi\dot\varphi+C^{\tau\varphi}\varphi\ddot\varphi+D^{\tau\varphi}\dot\varphi^2+E^{\tau\varphi}\dot\varphi\ddot\varphi,
\end{equation*}
where $A^{\tau\varphi},B^{\tau\varphi},D^{\tau\varphi}$ depend smoothly on $\tau\varphi,\tau\dot\varphi,\tau\ddot\varphi$ and $C^{\tau\varphi},E^{\tau\varphi}$ depend smoothly on $\tau\varphi,\tau\dot\varphi$. Therefore, integrating by parts,
\begin{equation*}
     \int_{I(\delta)}\frac{\partial^2}{\partial\tau^2}f(Q(\tau))\,ds=\int_{I(\delta)}\left(A^{\tau\varphi}\varphi^2+\tilde B^{\tau\varphi}\varphi\dot\varphi+\tilde D^{\tau\varphi}\dot\varphi^2\right)\,ds,
\end{equation*}
where
\begin{equation*}
    \tilde B^{\tau\varphi}=B^{\tau\varphi}-\frac{\partial }{\partial s}C^{\tau\varphi},\qquad \tilde D^{\tau\varphi}=D^{\tau\varphi}- C^{\tau\varphi}-\frac{1}{2}\frac{\partial }{\partial s}E^{\tau\varphi}.
\end{equation*}
In particular, $A^{\tau\varphi},\tilde B^{\tau\varphi},\tilde D^{\tau\varphi}$ depend smoothly on $\tau\varphi,\tau\dot\varphi,\tau\ddot\varphi$. Therefore, since $\supp\varphi\subseteq I(\delta)$, there exists $c_{\mathrm{II}}(\delta,L)$ such that 
\begin{equation}\label{nonsopiuchenomeinventareperilabel}
    \left|\mathrm{II}\right|\leq c_{\mathrm{II}}(\delta,L)\|\varphi\|_{C^2(I(\delta))}\left(\|\varphi\|^2_{L^2(I(\delta))}+\|\dot\varphi\|^2_{L^2(I(\delta))}\right).
\end{equation}
Combining \eqref{nonsopiuchenomeinventareperilabel} and \eqref{verolbdoveusipoincare}, we conclude that 
\begin{equation*}
    \penafun_L^\hhh\left(S_L^\varphi\right)\geq \penafun_L^\hhh\left(S_L\right)+\left(\frac{1}{2}c_\mathrm{I}(\delta,L)-c_{\mathrm{II}}(\delta,L)\|\varphi\|_{C^2(I(\delta))}\right)\left(\|\varphi\|^2_{L^2(I(\delta))}+\|\dot\varphi\|^2_{L^2(I(\delta))}\right).
\end{equation*}
Since $\area^\hhh(S_L^\varphi)=\area^\hhh(S_L)$, then $\penafun_L^\hhh\left(S_L^\varphi\right)- \penafun_L^\hhh\left(S_L\right)=\tmc^\hhh\left(S_L^\varphi\right)- \tmc^\hhh\left(S_L\right)$, whence the thesis follows.
\end{proof}
\begin{proof}[Proof of \Cref{intro_teo3}]
    It follows combining \Cref{muchmorethanstabproprotinv} and \Cref{localminimalitipropositionen}.
\end{proof}
\section{ Global minimality properties}\label{sec_glominprop}
In view of \Cref{intro_teo2}, the optimal Pansu-Minkowski sphere is the unique candidate global (strictly mean convex) minimizer for the Minkowski quotient. However, this possibility is ruled out by the unrestricted instability statement in \Cref{teononstabintro}. Nevertheless, as suggested by local minimality under rotational invariance as in \Cref{intro_teo3}, we show that the optimal Pansu-Minkowski sphere is indeed (the unique) global minimizer in the class of rotationally invariant surfaces. We begin with the following technical result. 
\begin{lemma}\label{lemma_rev_aux_urururururu}
  For $0\leqslant r\leqslant1$ and $0\leq u\le\frac{\pi}{2}$, set
\begin{equation}\label{rotmink:f}
 \Phi(r,u)=\frac{r}{\cos u}-\tan u+2u-\pi,\qquad f(r)=\sup
\left\{\Phi(r,u)\,:\,0\leqslant u<\frac \pi 2\right\}.
\end{equation}
Then 
\begin{equation}\label{rotmink:fbounds}
 -1-\frac\pi2\leqslant f(r)\leqslant 0
\end{equation}
and
\begin{equation}\label{rotmink:fintegral}
 2\int_0^1 rf(r)\,dr=-\frac43.
\end{equation}
\end{lemma}
\begin{proof}
   
Fix $r\in [0,1)$. Then 
\[
\Phi_u(r,u)= \frac{1+r\sin u-2\sin^2u}{\cos^2u}.
\]
The numerator, as a polynomial in $\sin u$, has exactly one
positive root at $\frac{r+\sqrt{r^2+8}}{4}\in\left[\frac{\sqrt{2}}{2},1\right).$ Then, setting 
\begin{equation*}
    u(r)=\arcsin \left(\frac{r+\sqrt{r^2+8}}{4}\right)\in \left[\frac\pi4\,\frac\pi2\right),
\end{equation*}
 $u(r)$ is (the unique) maximum point of $\Phi(r,u)$ on $\left[0,\frac \pi 2\right)$, and satisfies 
\begin{equation}\label{rotmink:maximizer}
 r=2\sin u(r)-\frac1{\sin u(r)}.
\end{equation}
In particular, since by construction $\Phi_u(r,u(r))=0$,
\begin{equation}\label{rotmink:fderivative}
 f(r)=\Phi(r,u(r))\overset{\eqref{rotmink:maximizer}}{=}2u(r)-\pi-\frac{ \cos u(r)}{\sin u(r)},\qquad
 f'(r)=\Phi_r(r,u(r))=\frac1{\cos u(r)}>0\qquad \text{$0\leqslant r<1$}.
\end{equation}
In addition, notice that $u(r)\to \frac{\pi}{2}$ as $r\to 1$, so that \eqref{rotmink:fderivative} implies that $f(r)\to 0$ as $r\to 1$. Arguing as above, $f(1)=0$, whence $f$ is continuous on $[0,1]$. Since $f(0)=-1-\frac\pi2$, \eqref{rotmink:fbounds} follows by monotonicity.
We prove  \eqref{rotmink:fintegral}. Indeed, recalling 
\eqref{rotmink:maximizer} and \eqref{rotmink:fderivative},
\begin{align*}
 2\int_0^1 rf(r)\,dr
 &=-\int_0^1 r^2f'(r)\,dr\\
 &=-\int_{\pi/4}^{\pi/2}
 \left(2\sin u-\frac1{\sin u}\right)^2
 \left(2+\frac1{\sin^2u}\right)\,du\\
 &=-\int_{\pi/4}^{\pi/2}
 \left(8\sin^2u-4-\frac2{\sin^2u}+\frac1{\sin^4u}\right)\,du\\
 &
 %=-\left[-2\sin(2u)+\cot u-\frac{\cot^3u}{3} \right]_{\pi/4}^{\pi/2}
 =-\frac43.
\end{align*}
\end{proof}
\begin{proof}[Proof of \Cref{rotmink:theorem}]
 Let $\gamma=(x,t):I\to\rr^2$ be the profile of $S$. Set $I=[s_0,s_1]$.
We may assume as in \Cref{sec_rotinvcritpointstmc} that $\max\{x(s)\,:\,s\in I\}=1$.
Up to re-parameterizing $\gamma$, we assume that $x^2\dot x^2+\dot t^2=1$, whence there is $\beta\in C^1(s_0,s_1)$ such that 
\begin{equation}\label{rotmink_angle}
 x\dot x=\cos\beta,\qquad \dot t=\sin\beta.
\end{equation}
Differentiating \eqref{rotmink_angle} yields
\begin{equation}\label{cosautilerev3759684728564838}
     \dot\beta=x(\dot x\ddot t-\ddot x\dot t)-\dot x^2\dot t,
\end{equation}
so that, recalling \eqref{subriemareaelementrotinv} and \eqref{hotinvinlemma},
\begin{equation}\label{rotmink:profile}
 xH^{\mathcal{H}}=x^2\dot\beta+\sin\beta,\qquad
 \sigma^{\mathcal{H}}(S)=2\pi\int_I x\,ds,\qquad 
 \mathscr{H}^{\mathcal{H}}(S=2\pi\int_I xH^{\mathcal{H}}\,ds.
\end{equation}
We claim that 
\begin{equation}\label{rotmink:difference}
 \frac{\sigma^{\mathcal{H}}(S)-\tmc^{\mathcal{H}}(S)}{2\pi}\leq 
 -\frac{4}{3}.
\end{equation}
First, notice that, by \eqref{rotmink:profile},
\begin{equation*}
     \frac{\sigma^{\mathcal{H}}(S)-\tmc^{\mathcal{H}}(S)}{2\pi}=\int_I(x-xH^\hhh)\,ds=\int_I(x- \sin\beta-x^2\dot \beta)\,ds.
\end{equation*}
We show how to handle the term involving $\dot \beta$. Let $q:\rr\to\rr$ be any bounded, $2\pi$-periodic, Lipschitz continuous function. Notice that 
\begin{equation}\label{rev_intderivatadifunzioneperhandlebeta}
 \int_I \frac{d}{ds}\left(\frac{x^2}{2}q(\beta)\right)\,ds=0.
\end{equation}
Indeed, we know from \Cref{sec_rotinvcritpointstmc} that $\gamma$ either is  a closed curve or its endpoints lie on the $t$-axis. In both cases, \eqref{rev_intderivatadifunzioneperhandlebeta} follows by the fundamental theorem of calculus, for which $2\pi$-periodicity of $q$ is needed only in the case of a closed curve. Therefore, \eqref{rev_intderivatadifunzioneperhandlebeta} yields
\begin{equation*}
\begin{split}
    \frac{\sigma^{\mathcal{H}}(S)-\tmc^{\mathcal{H}}(S)}{2\pi}&=\int_I\left(x- \sin\beta-x^2\dot \beta-\left(\frac{x^2}{2}q(\beta)\right)'\right)\,ds\\
    &=\int_I\left(x- \sin\beta-x^2\dot \beta\left(1+\frac{q'(\beta)}{2}\right)-q(\beta)\cos\beta \right)\,ds.
\end{split}     
\end{equation*}
Exact cancellation of the term containing $\dot\beta$ would suggest
$q'\equiv-2$, which is nevertheless incompatible with periodicity. Instead, we choose 
\begin{equation}\label{rotmink:q}
 q(\vartheta)=
 \begin{cases}
 \pi-2\vartheta,&0\leqslant\vartheta\leqslant\pi,\\
 2\vartheta-3\pi,&\pi\leqslant\vartheta\leqslant2\pi.
 \end{cases}
\end{equation}
This choice gives exact cancellation when $\dot t>0$. Instead,  when $\dot t<0$,
the remaining term will turn into a contribution with negative sign. Notice that $q$ is smooth when $\dot t\neq 0$. When $\dot t \ge 0$, then $q'(\beta)=-2,$ so that 
\begin{equation}\label{aboveidentity484858696965959595}
    x- \sin\beta-x^2\dot \beta\left(1+\frac{q'(\beta)}{2}\right)-q(\beta)\cos\beta=x- |\sin\beta|-q(\beta)\cos\beta \qquad\text{a.e. on $I$.}
\end{equation}
When $\dot t=0$, \eqref{aboveidentity484858696965959595} holds as well, because $\dot \beta=0$ a.e. on $\{\dot t=0\}$. Instead, when $\dot t\le 0$, then $q'(\beta)=2$, whence
\begin{equation*}
    x- \sin\beta-x^2\dot \beta\left(1+\frac{q'(\beta)}{2}\right)-q(\beta)\cos\beta=x- |\sin\beta|-q(\beta)\cos\beta-2 x H^\hhh.
\end{equation*}
In conclusion,
\begin{equation}\label{rotmink:primitive}
 \begin{split}
 \frac{\sigma^{\hhh}(S)-\tmc^{\hhh}(S)}{2\pi}
 =\int_I\bigl(x-|\sin\beta|-q(\beta)\cos\beta\bigr)\,ds-2\int_{\{\dot t<0\}}xH^{\hhh}\,ds.
 \end{split}
\end{equation}
We apply \Cref{lemma_rev_aux_urururururu} to the first
integrand in \eqref{rotmink:primitive}. For every $s\in I$, choose $u=u(s)\in[0,\pi/2]$ such that
\[
 \sin u=|\sin\beta|,\qquad \cos u=|\cos\beta|.
\]
A simple computation shows that 
\[
 q(\beta)\cos\beta=(\pi-2u)\cos u.
\]
Therefore, when  $\cos\beta\neq0$, \eqref{rotmink:f} yields
\begin{equation}\label{rotmink:pointwise}
 \begin{split}
 x-|\sin\beta|-q(\beta)\cos\beta
 =x-\sin u-(\pi-2u)\cos u=\Phi(x,u)\cos u
 \leqslant f(x)|\cos\beta|.
 \end{split}
\end{equation}
If $\cos\beta=0$, \eqref{rotmink:pointwise} reduces to $x-1\leqslant0$,
which is true by the radius normalization. Hence, combining 
\eqref{rotmink:primitive} and 
\eqref{rotmink:pointwise} and recalling that 
$|\cos\beta|=x|\dot x|$, we obtain
\begin{equation}\label{rotmink:integrated}
 \frac{\sigma^{\hhh}(S)-\tmc^{\hhh}(S)}{2\pi}
 \leqslant\int_I xf(x)|\dot x|\,ds
 -2\int_{\{\dot t<0\}}xH^{\hhh}\,ds.
\end{equation}
Recall that $S$ is either homeomorphic to a sphere or to a torus. Assume the first case, i.e. $x(s_0)=x(s_1)=0$. 
Choose $\hat s$ with $x(\hat s)=1$. By \eqref{rotmink:fbounds}, $xf(x)\leqslant0$.
Multiplying by $x f(x)$ the identities $|\dot x|\geqslant\dot x$ on $[s_0,\hat s]$ and
$|\dot x|\geqslant-\dot x$ on $[\hat s,s_1]$,
we obtain
\begin{equation*}
 \int_I xf(x)|\dot x|\,ds
 \leqslant\int_{s_0}^{\hat s}xf(x)\dot x\,ds
       -\int_{\hat s}^{s_1}xf(x)\dot x\,ds=2\int_0^1 rf(r)\,dr\overset{\eqref{rotmink:fintegral}}{=}-\frac43.
\end{equation*}
Therefore, recalling \eqref{rotmink:integrated} and since $H^\hhh\geq 0$, \eqref{rotmink:difference} follows. Instead, assume that $S$ is homeomorphic to a torus, whence $\gamma$ is a closed curve. Let
$\check s\in I$ be such that $0\le x(\check s)=\min\{x(s)\,:\,s\in I\}$. Arguing ad above,
\begin{equation}
 \int_I xf(x)|\dot x|\,ds
 \leqslant2\int_{x(\check s)}^1rf(r)\,dr
 =-\frac43-2\int_0^{x(\check s)} rf(r)\,dr\overset{\eqref{rotmink:fbounds}}{\leq}-\frac43+\left(1+\frac\pi2\right)x(\check s)^2.
 \label{rotmink:closed}
\end{equation}
By definition of $\check s$, and recalling the orientation of $\gamma$, $ \dot x(\check s)=0$ and $\dot t(\check s)=-1$. Hence, $\hat s\in \{\dot t\le 0\}$.
Let $I_-=[s_2,s_3]$ be the connected component of $\{\dot t<0\}$ containing $\check s$. Notice that $\dot t(s_2)=\dot t(s_3)=0$. Therefore, setting $g(s)=-x(s)\dot t(s)$, then $g(s_2)=g(s_3)=0$ and $g(\check s)=x(\check s)$. Therefore
\begin{equation}\label{sitma_rev_2857693756}
 \int_{I_-}|\dot g|\,ds\geqslant \int_{s_2}^{\check s}\dot g(s)\,ds-\int_{\check s}^{s_3}\dot g(s)\,ds=2x(\check s).
\end{equation}
Moreover, 
\[
 \dot g=-\dot x\sin\beta-x\cos\beta\dot\beta
 \overset{\eqref{rotmink_angle}}{=}-\left(\frac{\sin\beta}{x}+x\dot\beta\right)\cos\beta
 \overset{\eqref{rotmink:profile}}{=}-H^{\hhh}\cos\beta.
\]
Hence, recalling again that $H^\hhh\geq 0$,
\begin{equation}\label{rotmink:downward}
 2\int_{\{\dot t<0\}}xH^{\hhh}\,ds
 \geqslant 2x(\check s)\int_{I_-}H^{\hhh}\,ds
 \geqslant 2x(\check s)\int_{I_-}|\dot g|\,ds
\overset{\eqref{sitma_rev_2857693756}}{\geq} 4x(\check s)^2.
\end{equation}
Combining \eqref{rotmink:integrated}, \eqref{rotmink:closed}, and
\eqref{rotmink:downward}, we conclude that \eqref{rotmink:difference} holds.
By \eqref{rotmink:difference},  $\tmc^\hhh(S)\geq \frac{8\pi}{3}$. Set 
\begin{equation*}
    k:\left(\frac{8\pi}{3},+\infty\right)\to \rr,\qquad k(r)=r\left(r-\frac{8\pi}{3}\right)^{-\frac{2}{3}}.
\end{equation*}
By \eqref{rotmink:difference}, $ \mathcal Q^{\hhh}_{\mathrm{mink}}(S)\geq k\left(\tmc ^\hhh(S)\right)$. Moreover,
\[ k'(r)
 =\frac{1}{3}\left(h-8\pi\right)\left(h-\frac{8\pi}{3}\right)^{-\frac{5}{3}},
\]
whence $k$ has a global minimum at $r=8\pi$. Since $k(8\pi)=\left(18\pi\right)^\frac{1}{3}$, \eqref{rotmink:sharp} follows. Moreover, \Cref{lemmaincuisievinceellunquarto} implies that $S_{\frac{1}{4}}$ realizes equality in \eqref{rotmink:sharp}.
It remains to prove that it the unique. Assume that $S$ is a minimizer in the class of rotationally invariant competitors. To conclude the proof, it suffices to show that $S$ is strictly mean convex, because in this case $S$ is a critical point of $ \mathcal Q^{\hhh}_{\mathrm{mink}}$ under non-characteristic variations, and uniqueness follows by \Cref{lemmaincuisievinceellunquarto}. To see this, set \begin{equation}\label{rotmink:L}
 L=\frac{\tmc^{\hhh}(S)}{6\sigma^{\hhh}(S)}>0,
\end{equation}
let $\Phi$ be any non-characteristic variation, and et 
$\varphi$ be as in \eqref{lefunzioniimportanti}.
% These are (4.6) and (4.8) in arXiv:2605.29912v2.
Then \eqref{variazioneprimaareasubriemmain} and \eqref{primavariaztmcdtatementmain} yield
\begin{equation}\label{rotmink:firstvariation}
 \begin{aligned}
 \delta\mathcal Q^{\hhh}_{\mathrm{mink}}(S)[\Phi]
\overset{\eqref{rotmink:L}}&{=}\sigma^{\hhh}(S)^{-2/3}
   \bigl(\delta\tmc^{\hhh}(S)[\Phi]
         -4L\,\delta\sigma^{\hhh}(S)[\Phi]\bigr)\\
 &=-4\sigma^{\hhh}(S)^{-2/3}
   \int_S\bigl(J(\nu)\alpha+\alpha^2+LH^{\hhh}\bigr)
           \varphi\,d\sigma^{\hhh}.
 \end{aligned}
\end{equation}
Set $U=\{p\in S\setminus S_0\,:\,H^\hhh(p)\ge 0\}$. Whenever $\varphi(s,\theta)=\varphi(s)$ is supported in $U$, then variations $\Phi$ modeled as described in \Cref{sec_minloc} produce admissible, mean convex, rotationally invariant competitors. Since $J(\nu)\alpha+\alpha^2+LH^{\hhh}$ is independent of $\theta$ by \eqref{hotinvinlemma} and \eqref{lemmajvalfapiualfaquadro}, minimality of $S$ yields 
\begin{equation}\label{elonu384576958375969472}
    J(\vh)\alpha+\alpha^2+LH^{\hhh}\equiv 0\qquad\text{on }U.
\end{equation}
Therefore, we just need to show that $U=S\setminus S_0$. 
Notice that $U$ is open in  $S\setminus S_0$. Moreover, $U\neq \emptyset$ by \Cref{nominimalandclosed}. In addition, $S\setminus S_0$ is open and connected in $S$. To conclude, it suffices to prove that $U$ is closed. Assume that the profile of $S$ is parametrized as in
\eqref{rotmink_angle}. Then \eqref{lemmajvalfapiualfaquadro}, \eqref{cosautilerev3759684728564838} and \eqref{elonu384576958375969472} yield
\[
 J(\vh) \alpha+\alpha^2=-\sin\beta\,\dot\beta,\qquad L H^\hhh=\sin\beta\,\dot\beta\qquad\text{on $U$.}
\]
In particular, setting $z=\frac{\sin\beta}{x}$,  \eqref{rotmink:profile}  implies that 
\begin{equation*}
    z H^\hhh=\sin\beta\,\dot\beta+z^2=L H^\hhh+z^2\qquad\text{on }U,
\end{equation*}
whence
\begin{equation*}
 H^{\hhh}(z-L)=z^2\qquad\text{on }U.
\end{equation*}
In particular, since $H^{\hhh}>0$ on $U$, then $z\neq 0$, so that $z>L$. In conclusion,
\begin{equation}\label{rotmink:gap}
 H^{\hhh}=\frac{z^2}{z-L}
 =4L+\frac{(z-2L)^2}{z-L}\geqslant4L
 \qquad\text{on }U.
\end{equation}
Since $H^\hhh$ is continuous on $S\setminus S_0$ and $L\ge 0$, $U$ is closed by \eqref{rotmink:gap}, whence $U=S\setminus S_0$. The thesis follows. 
\end{proof}
Finally, we address the global behavior of the Minkowski quotient beyond the rotationally invariant class, showing the failure of any positive lower bound for the Minkowski quotient.
\begin{proof}[Proof of \Cref{ellipsoids_zero_infimum}]
Fix $\lambda\geq 1$. Clearly, $S_\lambda$ is a smooth, closed, embedded surface. Recalling the notation $z=(x,y),$ set 
\[
    M_\lambda=\begin{pmatrix}\lambda^{-2}&0\\0&\lambda^2\end{pmatrix},
    \qquad q(z)=\langle M_\lambda z,z\rangle,\qquad s(z)=\sqrt{1-q(z)},
    \qquad D_\lambda=\{z\in\rr^2:q(z)<1\}.
\]
Observe that $S_\lambda\setminus\{t=0\}$ is the union of the graph 
\[
    t=u_\lambda(z)\coloneqq-s(z),
    \qquad z=(x,y)\in D_\lambda,
\]
and of its reflection with respect to $\{t=0\}$. Both functionals in \eqref{twoquantitiesrevlambdaellpisdoistametjne} are invariant under such reflection. Indeed, $S_\lambda$ is invariant by the isometry $(x,y,t)\mapsto (x,-y,-t).$ Moreover,
$S_\lambda\cap\{t=0\}$ is non-characteristic, and $\sigma^\hhh\left(S_\lambda\cap\{t=0\}\right)=0$. Therefore, both functionals in \eqref{twoquantitiesrevlambdaellpisdoistametjne} can be computed by evaluating them on the lower cap of $S_\lambda$ and by doubling the corresponding values.
We adopt the same notation as in the proof of \Cref{rotational_symmetry_critical_points}. Accordingly, recall that $Jz=(-y,x)$, $F=Du_\lambda+Jz$ and $W=|F|$, and moreover
\begin{equation}\label{ellipsoids_graph_identities}
    d\sigma^\hhh=W\,dz,
    \qquad H^\hhh
       =\frac{\langle D^2u_\lambda\,JF,JF\rangle}{W^3}
       \quad\text{on }\{F\neq0\}.
\end{equation}
Since $Dq(z)=2M_\lambda z$ and $Ds(z)=-s(z)^{-1}M_\lambda z$, then
\begin{equation}\label{ellipsoids_graph_hessian}
    Du_\lambda(z)=\frac{M_\lambda z}{s},\quad
    F(z)=\frac{M_\lambda z}{s}+Jz,\quad  JF(z)=\frac{JM_\lambda z}{s}-z,\quad D^2u_\lambda(z)=\frac{M_\lambda}{s}
       +\frac{(M_\lambda z)\otimes(M_\lambda z)}{s^3}.
\end{equation}
Notice that $F(0)=0$, while $\langle F(z),z\rangle>0$ for any $z\neq 0$. Therefore, $(S_\lambda)_0=\{(0,0,-1),(0,0,1)\}$.
A direct computation shows that 
\begin{equation}\label{rev_scambiarematriciiiicfnojnvoe}
    J^{\mathsf T}M_\lambda J=M_\lambda^{-1},\qquad J^{\mathsf T}M_\lambda^{-1}J=M_\lambda.
\end{equation}
Therefore, recalling that $\langle(v\otimes v)w,w\rangle=\langle v,w\rangle^2$ for any $v,w\in\rr^2,$
\[
    \langle M_\lambda z,JF\rangle=-q,
    \qquad
    \langle M_\lambda JF,JF\rangle
      =\langle M_\lambda^{-1}F,F\rangle
      =\frac q{s^2}+\frac2s\langle z,Jz\rangle+q
      =\frac q{s^2}+q.
\]
Consequently, \eqref{ellipsoids_graph_hessian} and $s^2+q=1$ yield
\begin{equation*}
\begin{split}
     \langle D^2u_\lambda\,JF,JF\rangle \overset{\eqref{ellipsoids_graph_hessian}}&{=}\frac{1}{s}\langle M_\lambda JF,JF\rangle+\frac{1}{s^3}\langle M_\lambda z,JF\rangle^2\\
      \overset{\eqref{ellipsoids_graph_hessian},\eqref{rev_scambiarematriciiiicfnojnvoe}}&{=}\frac{1}{s}\langle M_\lambda^{-1} F,F\rangle+\frac{q^2}{s^3}\\
      \overset{\eqref{ellipsoids_graph_hessian}}&{=}\frac{1}{s^2}\langle z,F\rangle+\frac{1}{s}\langle M_\lambda^{-1} J z,F\rangle+\frac{q^2}{s^3}\\
      \overset{\eqref{rev_scambiarematriciiiicfnojnvoe}}
      %&{=}\frac{q}{s^3}+\frac{q}{s}+\frac{q^2}{s^3}\\
     &=\frac{2q}{s^3}.
\end{split}
\end{equation*}
Therefore, \eqref{ellipsoids_graph_identities} yields
\begin{equation}\label{ellipsoids_graph_curvature}
    H^\hhh=\frac{2q}{s^3W^3}
       =\frac{2q}{|M_\lambda z+sJz|^3}>0
       \qquad\text{on }D_\lambda\setminus\{0\}.
\end{equation}
In particular, when $z$ tends to a point $z_0\in\partial D_\lambda$, then $H^\hhh$ 
tends to $2|M_\lambda z_0|^{-3}>0$. Therefore, $S_\lambda$ is strictly mean convex.
In order to compute the two quantities in $\eqref{twoquantitiesrevlambdaellpisdoistametjne}$, we introduce a change of variable which splits the relevant integrals into the product of two one-dimensional ones. Accordingly, define $$\Om_\lambda=\{(\varsigma,\vartheta)\in\rr^2\,:\,\varsigma\in(0,1),\,\vartheta\in (\arctan \varsigma,\arctan \varsigma+2\pi)\},$$ and set
\begin{equation*}
    \Phi_\lambda(\varsigma,\vartheta)=\Big(\lambda\sqrt{1-\varsigma^2}\cos\big(\vartheta-\arctan \varsigma\big), \lambda^{-1}\sqrt{1-\varsigma^2}\sin\big(\vartheta-\arctan \varsigma\big)\Big)\qquad \text{on }\Om_\lambda.
\end{equation*}
Heuristically, $\Phi_\lambda$ is an angular modification of the standard elliptic parametrization. It is easy to check that  
\begin{equation}\label{rev_determinantecovinfzero}
    |\det D\Phi_\lambda|=\varsigma,
\end{equation}
and that $\Phi_\lambda$ is (up to a negligible set) a smooth parametrization of $D_\lambda$. By construction, $s(\Phi_\lambda(\varsigma,\vartheta))=\varsigma$, so that $F(\Phi_\lambda(\varsigma,\vartheta))$ writes as
\[
      \frac{\sqrt{1-\varsigma^2}}{\varsigma}
       \Big(\lambda^{-1}    \big(\cos\big(\vartheta-\arctan \varsigma\big)-\varsigma \sin\big(\vartheta-\arctan \varsigma\big)\big),
             \lambda\big(\sin\big(\vartheta-\arctan \varsigma\big)+\varsigma\cos\big(\vartheta-\arctan \varsigma\big)\big)\Big).
\]
On the other hand, the standard addition formulas for trigonometric functions yield
\begin{align*}
    \cos\big(\vartheta-\arctan \varsigma\big)-\varsigma \sin\big(\vartheta-\arctan \varsigma\big)&=(1+ \varsigma^2)\cos\big(\arctan  \varsigma\big)\cos\vartheta,\\
    \sin\big(\vartheta-\arctan \varsigma\big)+\varsigma\cos\big(\vartheta-\arctan \varsigma\big)&=(1+ \varsigma^2)\cos\big(\arctan  \varsigma\big)\sin\vartheta.
\end{align*}
Therefore, recalling that $\cos\big(\arctan  \varsigma\big)=(1+ \varsigma^2)^{-\frac{1}{2}}$,
\begin{equation}\label{rev_wincovmerifvj4ntrvitrvn}
        F(\Phi_\lambda(\varsigma,\vartheta))  =\frac{\sqrt{1-\varsigma^4}}{\varsigma}
       \big(\lambda^{-1}    \cos\vartheta,
             \lambda\sin\vartheta\big),\qquad W(\Phi_\lambda(\varsigma,\vartheta))=\frac{\sqrt{1-\varsigma^4}}{\varsigma}\sqrt{\lambda^{-2}\cos^2\vartheta+\lambda^2\sin^2\vartheta}.
\end{equation}
 By \eqref{ellipsoids_graph_identities}, \eqref{rev_determinantecovinfzero} and \eqref{rev_wincovmerifvj4ntrvitrvn}, and by $2\pi$-periodicity of the angular integrand,
\begin{equation*}
\begin{split}
    \sigma^\hhh(S_\lambda)
       =2\left(\int_0^1\sqrt{1-\varsigma^4}\,d\varsigma\right)
          \left(\int_0^{2\pi}\sqrt{\lambda^{-2}\cos^2\vartheta+\lambda^2\sin^2\vartheta}\,d\vartheta\right)\geq 2\lambda\left(\int_0^1\sqrt{1-\varsigma^4}\,d\varsigma\right)
          \left(\int_0^{2\pi}|\sin\vartheta|\,d\vartheta\right),
\end{split}
\end{equation*}
so that $\sigma^\hhh(S_\lambda)\to+\infty$ as $\lambda\to+\infty$. On the other hand, recalling also \eqref{ellipsoids_graph_curvature},
\begin{equation*}
\begin{split}
    \tmc^\hhh(S_\lambda)
       &=4\left(\int_0^1\frac{1}{1+\varsigma^2}\,d\varsigma\right)
          \left(\int_0^{2\pi}\frac{1}{\lambda^{-2}\cos^2\vartheta+\lambda^2\sin^2\vartheta}\,d\vartheta\right)=\pi\int_0^{2\pi}\frac{1}{\lambda^{-2}\cos^2\vartheta+\lambda^2\sin^2\vartheta}\,d\vartheta.
\end{split}
\end{equation*}
Finally, the substitution
$v=\lambda^2\tan\vartheta$ yields
\[
    \int_0^{2\pi}\frac{1}{\lambda^{-2}\cos^2\vartheta+\lambda^2\sin^2\vartheta}\,d\vartheta
       =4\int_0^{\pi/2}
           \frac{1}{\lambda^{-2}\cos^2\vartheta+\lambda^2\sin^2\vartheta}\,d\vartheta
       =4\int_0^\infty\frac{1}{1+v^2}\,dv
       =2\pi,
\]
whence
$\tmc^\hhh(S_\lambda)=2\pi^2$. The thesis follows.
\end{proof}

\section*{Declarations}
\footnotesize{
%\noindent{\it \textbf{Acknowledgements.}} The authors thank Fares Essebei, Andrea Pinamonti, Francesco Serra Cassano and Giacomo Vianello for interesting and valuable conversations on the topic of the paper.
%\smallskip
%
\noindent \textbf{AI usage statement} 
Many ideas in the proof of \Cref{rotational_symmetry_critical_points} emerged after a series of conversations with ChatGPT-6 Astra. Subsequently, the proof of \Cref{rotational_symmetry_critical_points} has been substantially improved, optimized and simplified by the authors. ChatGPT-6 Astra also helped establishing the proofs of \Cref{rotmink:theorem} and \Cref{ellipsoids_zero_infimum}, especially by suggesting the relevant re-parameterizations and changes of variables.
Finally, ChatGPT-5.4 helped to detect the explicit parametrization \eqref{pansuminkparamtheorem} starting from the graphical parametrization \eqref{aistatementper48348348}. The paper has been entirely written and checked by the authors. The authors take full
responsibility for the content of the paper.
\smallskip

\noindent{\textbf{Conflict of interest.}} The authors have no financial interests or conflicts of interest related to the subject matter.
\smallskip

\noindent \textbf{Data availability statement.} Data sharing not applicable as no datasets were generated or analyzed during the current study.
\bibliographystyle{abbrvmr_macro}
\bibliography{biblio}
\end{document}